\documentclass[12pt]{article}
\usepackage{amsmath,amsthm, amssymb }
\newtheorem{theorem}{\noindent Theorem}
\newtheorem{lemma}[theorem]{\noindent Lemma}
\newtheorem{corollary}[theorem]{\noindent Corollary}
\newtheorem{proposition}[theorem]{\noindent Proposition}

\theoremstyle{remark}
\newtheorem{remark}[theorem]{\it Remark}

\makeatletter

\@addtoreset{equation}{section}
\makeatother

\newcommand{\al}{\alpha}
\newcommand{\de}{\delta}
\newcommand{\ep}{\varepsilon}
\newcommand{\lam}{\lambda}

\newcommand{\ga}{\gamma}
\newcommand{\De}{\Delta}
\newcommand{\Ga}{\Gamma}

\newcommand{\beq}{\begin{eqnarray*}}
\newcommand{\eeq}{\end{eqnarray*}}
\newcommand{\beqn}{\begin{equation}}
\newcommand{\eeqn}{\end{equation}}
\newcommand{\ds}{{\rm ds}}

\newcommand{\n}{\noindent}
\newcommand{\pf}{\n{\it Proof. \,}}
\def\v2{\vskip2mm}

\begin{document}
\begin{center}
{\Large Estimates of potential functions of 
  random walks on $\mathbb{Z}$\\
with zero mean and infinite variance and their applications }
\vskip6mm
{K\^ohei UCHIYAMA} \\
\vskip2mm
{Department of Mathematics,\\ Tokyo Institute of Technology} \\
{Oh-okayama, Meguro Tokyo 152-8551\\
}
\end{center}

\vskip6mm

\begin{abstract}
Let $S_n =X_1+\cdots +X_n$ be  an irreducible   random walk (r.w.) on the one dimensional integer lattice with zero mean, infinite variance and 
 i.i.d. increments $X_n$. We obtain an upper and lower bounds of the potential function, $a(x)$, of  $S_n$ in the form $a(x)\asymp x/m(x)$ under a reasonable condition on the distribution of $X_n$; we especially show  that   as $x\to\infty$
$$a(x) \asymp \frac{x}{m_-(x)} \quad\mbox{and}\quad   
\frac{a(-x)}{a(x)} \to 0 \quad\;\;\mbox{if}\quad  \lim_{x\to +\infty} \frac{m_+(x)}{m_-(x)} =0,$$
 where  $m_\pm(x) = \int_0^xdy\int_y^\infty P[\pm X_1>u]du$ and $m=m_++m_-$.
Under certain conditions on the tails of the distribution  of $X$  we derive precise asymptotic forms of $a(x)$ as $x\to +\infty$ or/and $-\infty$. The results are applied  to derive a sufficient condition for the relative stability of the ladder height and  estimates of some escape probabilities from the origin; we show among others that under the above condition on $m_+/m-$,  $P[S_n>0] \to 1/\alpha$  if and only if   the  probability  of exiting   a long interval $[-Q,R]$ through  the upper boundary converges to $\lambda^{\alpha-1}$ as $Q/(Q+R) \to \lambda$ for any $0<\lambda<1$.
\vskip8mm \noindent
\textbf{Key words}: recurrent r.w.;  potential function;  infinite variance; escape probability;  two-sided exit  problem; relatively stable; ladder height\\
\textbf{  AMS MSC (2010)}: Primary 60G50,  Secondary 60J45.
 \end{abstract}

\section{Introduction and Statements of Results}
We study asymptotic properties of the potential function $a(x)$ of a recurrent random walk (r.w.) on the integer lattice $\mathbb{Z}$ with independent and identically distributed (i.i.d.) increments. Denote by $F$  the common distribution  function  of the increments  and let $X$  be a random variable having the distribution  $F$. Let $(\Omega, {\cal F}, P)$ be the probability space on which $X$, as well as the r.w., is defined, and let  $E$ be the expectation  by $P$. Then $a(x)$ is defined by
\[
a(x) = \sum_{n=0}^\infty  \big(P[S_n=0]-P[S_n=-x]\big),
\]
where $S_n$ denotes the r.w.  started at $0$; the infinite series
on the RHS is convergent, and  
if $\sigma^2= E X^2<\infty$ then $a(x+1)-a(x) \sim \pm 1/2\sigma^2$ as $x\to \pm \infty$ ($\sim$ designates that the ratio of two  sides of it approaches  unity), whereas in case $\sigma^2=\infty$ the known result on the behaviour of $a(x)$ for large values of $|x|$ is less precise: it  says  only that 
$
a(x+1)-a(x)\to 0
$
 and 
\beqn\label{bar _a}
\bar a(x) :=\frac12 \big[a(x)+a(-x)\big] \to \infty
\eeqn
as $|x|\to\infty$ (cf. \cite[Sections 28 and 29]{S}),  although certain explicit asymptotic forms of $a(x+1)-a(x)$ and/or $\bar a(x)$ are given if  $F$ satisfies 
some specific condition like  regular variation of its tails (cf. \cite{Belk}, \cite{Uattrc}, Section 8.1.1).
 In this paper,  we suppose that the r.w. is irreducible (as a Markov chain on $\mathbb{Z}$) and that
\beqn\label{X0}
 EX =0\quad\mbox{and}\quad \sigma^2=\infty,
 \eeqn
unless the contrary is stated explicitly. [We shall include several results for a general r.w. on $\mathbb{Z}$ with $E|X|=\infty$ that can be even not recurrent.] We obtain an expression of the growth rate of $\bar a$ in terms of a certain functional of $F$ for a large class of $F$ and 
observe that if the right-hand tail of $F$ is asymptotically negligible \lq in average' in comparison to  the left-hand tail, then
$a(x)/a(-x) \to0$.
The potential function $a(x)$ plays a fundamental role in the potential theory of the recurrent r.w., especially for the r.w. killed on hitting the origin, the Green function of the killed r.w. being expressed as $g(x,y)=a(x)+a(-y)-a(x-y)$. We shall give some applications of the results based on this fact, mostly in the case when one tail of $F$ is negligible relative to the other. For the previous works on $a(x)$ and related problems readers are referred to \cite{K3}, \cite{Uladd} as well as to \cite{S}.   

To state the results of the paper, we need   the following functionals of $F$ that also bear a great deal of relevance to our analysis:  for $x\geq 0$, $t\geq 0$ \; 
  \beq
  \mu_-(x) = P[X<-x],\; \; \mu_+(x) = P[X> x],\;  \; &&\mu(x)=\mu_-(x)+\mu_+(x);\\
  \eta_\pm(x) =\int_x^\infty \mu_\pm(y)dy,\qquad\qquad &&\eta(x) = \eta_-(x)+\eta_+(x);\\
  c_\pm(x) = \int_0^x y\mu_\pm(y)dy \qquad\qquad && c(x) = c_+(x)+c_-(x);\\
 m_\pm(x) =\int_0^x \eta_\pm(t)dt, \qquad\qquad && m(x) = m_+(x)+m_-(x);\\
\al_\pm(t)= \int_0^\infty \mu_\pm(y) \sin ty\,dy, \qquad &&\al(t) = \al_+(t)+\al_-(t);\\
  \beta_\pm(t) = \int_0^\infty \mu_\pm (y) (1-\cos ty)dy, \quad  
  && \beta(t) = \beta_+(t)+\beta_-(t);\\
  \ga(t)= \beta_+(t)-\beta_-(t) = - E\sin\, tX;
  \eeq
  \[
  \mbox{and}\qquad \psi(t) = E e^{itX}. \qquad\qquad
  \]
  An integration by parts shows 
  $$m_{\pm}(x) =c_\pm(x)+ x\eta_\pm(x).$$
In the above formulae,  $x$ designates a real number, whereas $x$ is always  an  integer in  $a(x)$;  we shall not mention  which usage is adopted if  this duplicity will  cause  no confusion.    
 \vskip2mm\noindent
\begin{theorem}\label{th:1} \, 
For some universal constant $C_*>0$
\[
\bar a(x) \geq C_* x/m(x)  \qquad \mbox{for all sufficiently large} \quad x\geq 1.
\]
\end{theorem}
\v2
This theorem plainly entails the implication
\beqn\label{Z/a}
\sum \bar a(x) P[X>x] <\infty \;\Longrightarrow\; \int_1^\infty \frac{x P[X>x]}{m(x)}dx<\infty\eeqn
which bears immediate relevance to  criteria for the summability  of   the first ascending ladder height, $Z$ say.
 The integrability condition on the RHS of (\ref{Z/a}), which implies  $m_+(x)/m(x)\to 0$ under $\sigma^2=\infty$ as is easily seen,  is equivalent to  the necessary and sufficient condition   for $EZ<\infty$ due to Chow \cite{Cho}. (\ref{Z/a})  is verified in \cite{Unote}  and constitutes 
 one of the key observations  there that lead to  the following equivalence relations
 $$\sum \bar a(x) P[X>x] <\infty \; \Longleftrightarrow \; EZ<\infty  \; \Longleftrightarrow \;  \int_1^\infty \frac{x P[X>x]}{m_-(x)}dx<\infty$$
 without recourse of Chow's result (see Section 4 (especially Lemma 4.1) of \cite{Unote}).
 
  
 
 We bring in the following condition to  obtain  some  upper estimates of $\bar a(x)$:
\[
{\rm (H)} \qquad \qquad{\displaystyle  \delta_H:=\liminf_{t\downarrow 0} \frac{\al(t)+ |\ga(t)| }{\eta(1/t)} >0}.\qquad \qquad
\]
Note that  $1-\psi(t)=t \al(t) + it\ga(t)$.

\begin{theorem}\label{th:1_2} \,
 {\rm (i)}\, If  condition (H)
holds,  then for some  constant $C_H^*$ depending only on $\delta_H$, 
\[ 
  \bar a(x) \leq   C_H^* \, x/m(x)\qquad  \mbox{for all sufficiently large} \quad x\geq 1.
  \]
  \v2
   {\rm (ii)} \; If \;\;  ${\displaystyle \frac{\al(t)+|\ga(t)|}{t\!\;m(1/t)} \to 0}$ ($\, t\downarrow 0$),\; then \;
    ${\displaystyle \frac{\bar a(x)}{x/m(x)} \to\infty}$ ($x\to\infty$).  
       [The converse is not true; a counter example  will be give at the end of Section 8.2.] 
\end{theorem}

\vskip2mm
According to Theorems \ref{th:1} and \ref{th:1_2},  if (H) holds, then $\bar a(x) \asymp x/m(x)$ as $x\to\infty$ (i.e., the ratio of two sides of it is bounded away from zero and infinity), which  entails 
some regularity of $\bar a(x)$ like its being  almost increasing (i.e.,   $\bar a(y)/ \bar a(x)$ is bounded away from zero if $y>x$ and $x$ is large enough), while in general  $\bar a(x)$ may behave very irregularly  as will be exhibited by an  example  (see Section 8.2).

For condition (H) to hold  each of the following conditions (\ref{H2}) to (\ref{H0}) is sufficient:
         \beqn\label{H2}\limsup_{x\to\infty} \;\frac{\mu_+(x)}{ \mu_-(x)} <1;
     \eeqn
      \beqn\label{H1}
    \limsup_{x\to\infty} \frac{x\eta_+(x)}{m_+(x)} < 1 \;\; \mbox{or} \;\;     \limsup_{x\to\infty} \frac{x\eta_-(x)}{m_-(x)} < 1  \;\; \mbox{or} \; \;   \limsup_{x\to\infty} \frac{x\eta(x)}{m(x)} < 1;
    \eeqn
  \beqn\label{H3}
  \frac{m_+(x)}{m(x)} \;\; \mbox{converges as $x\to \infty$ to a number}\;\neq \frac12; 
    \eeqn 
   \beqn\label{H0}
  \limsup_{x\to\infty} \frac{x[\eta_+(x)\wedge \eta_-(x)]}{m_+(x) \vee m_-(x)} < \frac14.
  \eeqn
The last  condition (\ref{H0}) ensures  (H) at least under $x\eta_+(x)/m(x)\to 0$ which can be true even when
$\limsup m_+(x)/m(x)=1$. 

  That  (\ref{H2}) is sufficient for (H) follows from  $\liminf_{t\downarrow 0} \beta(t)/\eta(1/t) >0$ (cf.\,(\ref{m1}b)) since $\liminf_{t\downarrow 0} |\ga(t)|/ \beta(t) >0$ under (\ref{H2}).
       The sufficiency of the   conditions (\ref{H1}) and (\ref{H0})
  will be  verified in Section 2 (Lemmas \ref{Lem20} and \ref{Lem16} for (\ref{H1}) and Lemma \ref{Lem21} for (\ref{H0})); as for (\ref{H3}) see Remark \ref{Rem29}. 
  
 We do not know whether (H) is necessary for $\bar a(x)m(x)/x$ to be bounded. 
 The next result entails that when 
   the distribution of $X$ is nearly symmetric in the sense that the limit in (\ref{H3}) equals $1/2$, namely
   \beqn\label{eta+/-}
   m_-(x)/m_+(x)\to 1 \quad\mbox{as}\quad  x\to\infty,
   \eeqn
   condition (H) is  necessary as well as sufficient in order that  $\bar a(x)$ is comparable to  $x/m(x)$.
       
\begin{theorem}\label{th:1_3} \,Suppose  (\ref{eta+/-}) to hold. Then 
\v2
{\rm (i)}  $\lim_{t\downarrow 0}\ga(t)/[t\;\!m(1/t)] = 0$; and 
\v2
{\rm (ii)}  each of the three inequalities in the disjunction  (\ref{H1}) is  necessary (as well as sufficient)   in order that $\limsup_{x\to\infty}  \bar a(x)m(x)/x <\infty$.\end{theorem}  

\begin{corollary}\label{COR4} Under (\ref{eta+/-}) \, $ \limsup \bar a(x)m(x)/x <\infty$ \,  if and only if  (H) holds.
\end{corollary}
\noindent
\pf\, If $ \limsup \bar a(x)m(x)/x <\infty$, then by Theorem \ref{th:1_3} $\limsup x\eta(x)/m(x)<1$, or equivalently, $\liminf c(x)/m(x)>0$, which implies   $\liminf \alpha(t)/\eta(1/t) >0$ as we shall see  (Lemma \ref{Lem16}), whence (H) holds. The converse follows from Theorem \ref{th:1_2}.  \qed

\v2

 Practically in most cases, Theorem \ref{th:1_3} together with (\ref{H2}) to (\ref{H0}) provides the criterion, expressed in terms of $\mu_\pm$, $m_\pm$ or/and $\eta_\pm$, to judge whether the condition (H) holds.  
  If $F$ is in the domain of attraction of a stable law the result is simplified  so that  $\lim  \bar a(x)m(x)/x =\infty$ if  
  $x\eta(x)/m(x)\to 1$ and  $m_-(x)/m_+(x) \to 1$;  otherwise $\bar a(x)\sim C x/m(x)$  with $C>0$  (see   Proposition \ref{Prop61}(i, iv) and Remark \ref{Rem62}(ii) of Section 8). \v2
\begin{proposition} \label{Prop5} \, If (H) holds, then for $0 < x < 2R$,
$$\bigg| 1- \frac{\bar a(x)}{\bar a(R)}\bigg| \leq C\bigg(1-\frac{x}{R}\bigg)^{1/4}$$
for some constant $C$ that depends only on $\de_H$.
\end{proposition}

From  Proposition  {\ref{Prop5} it  plainly follows   that $\bar a(x)/\bar a(R) \to 1$ as  $x/R\to 1$ if (H) holds, which fact, not trivial, is what we need for some of our applications. 

In the applications as well as  the next result we consider the r.w. $S$ under  the  condition   
 $ m_+(x)/m_-(x) \to 0$ as $ x\to+\infty$, which  for simplicity  we write
   $m_+/m \to0$. (Similar conventions  will also apply to $\eta_+/\eta$, $c/m$, etc.)
\begin{theorem}\label{Thm6} \, 
Suppose $m_+/m\to 0$.  Then $\delta_H =1$ in {\rm (H)} and
\[
\frac{a(-x)}{a(x)} \to 0\quad \mbox{as} \quad x\to +\infty;
\]
in particular $a(x) \sim 2\bar a(x) \asymp x/m(x)$.
\end{theorem}

In the proof of Theorem \ref{Thm6} we shall see that if  $m_+/m\to0$ and
$$b_\pm(x)  =
\frac1{\pi}\int_0^\pi \frac{t\beta_\pm(t)\sin xt}{|1-\psi(t)|^2}dt,$$
then $a(x)\sim b_-(x)$ and  $b_+(x) = o\big(b_-(x)\big)$, which together  indeed imply  $a(-x)/a(x)\to 0$. 

\v2
The applications in this paper of the theorems stated above  are made mostly in case $a(-x)/a(x)\to0$. It is significant to consider the r.w. $S$ under   $\lim\inf_{x\to\infty} a(-x)/a(x)>0$, but we do not know when this condition holds and it is interesting to find it:   
\v2
{\bf Problem.}  {\it Find a reasonable sufficient condition for $\de_*:=\lim\inf_{x\to\infty} a(-x)/a(x)>0$.}
\v2
\noindent
Even if $\liminf (1-F(x))/F(-x)>0$   we do not know  any answer to the above problem except in cases where we can compute the asymptotic form of $a(x)$ more or less explicitly as given in Theorem \ref{thm7} below. In Section 8.1.1 we shall give a criterion for $\de_*>0$ when $F$ is attracted to a stable law by some explicit evaluation of the asymptotics of $a(x)$ as $|x|\to\infty$.
 
 When $F$ is relatively stable,  one can obtain some exact asymptotic form of $a$. $F$ is called {\it relatively stable (r.s.)} if there exists a (non-random) sequence $B_n$ such that  $S_n/B_n \to 1$ in probability. Here  $B_n$ is necessarily either  ultimately positive or ultimately negative. In the positive case we call $F$ {\it positively r.s.} after \cite{KM0}. 
  Put
 \beqn\label{A}
 A(x) = \int_0^x \big[ \mu_+(y)-\mu_-(y)\big]dy.
 \eeqn
 According to \cite{KM0} (see around Eq(1.15)  of it),   
  in order that $F$ is positively r.s. it is necessary and sufficient that
\beqn\label{PRS}
A(x)/x\mu(x) \to\infty.
\eeqn  
This implies $A$ is slowly varying (s.v.) at infinity (in Karamata's sense:see \cite{BGT}).
If  $F$ belongs to a domain of attraction of a stable law of exponent $\alpha \in(0,2]$, (\ref{PRS}) holds  if and only if $\alpha=1$ and $P[S_n>0]\to 1$.
(See Remark \ref{Rem11} and Section 8.2 for related matters.) What are mentioned right above are valid for any r.w. with $P[|X|> x]> 0$ for every real number $x>0$.  Under 
(\ref{PRS}) the function $A$ is positive some point on and in the sequel let $x_0$ be a positive integer such that
$A(x)>0$ for $x\geq x_0$.  In the next theorem, we do not assume the summability of $X$. 
\begin{theorem}\label{thm7} Suppose that (\ref{PRS}) holds. Then both $a(x)$ and $\bar a(x)$ are s.v. at infinity and 
as $x\to\infty$,  
$$a(x)-a(-x) \sim 1/{A(x)}, $$
$$a(x) \sim \int_{x_0}^x\frac{\mu_-(y)}{A^2(y)}dy\quad\mbox{and}\quad a(-x) = \int_{x_0}^x\frac{\mu_+(y)}{A^2(y)}dy +o(a(x)).$$
\end{theorem}
\v2
When  $F$ is transient and satisfies (\ref{PRS}), an asymptotic result analogous to the above for the Green function of the r.w. is given in \cite[Theorem; Remark 7]{Urenw} and Theorem \ref{thm7} is the counterpart to it for the recurrent walk.

We apply Theorems \ref{th:1}, \ref{th:1_2} and \ref{Thm6} to obtain some asymptotic estimates of the upwards overshoot distribution of the r.w. over a high level, $R$ say. 
We also make the application to estimate the probability of escaping the origin, i.e., the one that the r.w. $S$ goes beyond the level $R>1$ or both  $R$ and $-R$ without hitting zero. 
  Many works study the overshoot or the matters related to it   \cite{R}, \cite{GOT}, \cite{DM}, \cite{K2}, etc. for $X$ belonging to the domain of attraction of a stable law, while the escape probabilities---which appear in the limit theorem for the sojourn time distribution (see (\ref{KL1}))---have rarely been investigated. In any way, it  seems  hard to obtain a sharp estimate of these things  in a general setting.  Under  the condition  $m_+/m\to 0$, however,  the theorems above are effectively used to yield natural results. The overshoot distribution is related to the relative stability of the ladder height variable. 
We shall obtain a sufficient condition for the relative stability of the ladder height (Proposition \ref{Prop40}) and  the asymptotic estimates of the escape probabilities  as mentioned above  (Proposition \ref{Prop52} (one-sided) and Propositions \ref{Prop55} and  \ref{Prop58} (two-sided)). As a byproduct of these results  we deduce  under $m_+/m\to0$  
the asymptotic monotonicity of $a(x)$, $x>0$ by showing that
$$P[\sigma^0_{[R,\infty)} < \sigma^0_{\{0\}}] \sim 1/a(R)$$
(see Corollary \ref{Cor54}) as well as  the following result on the classical two-sided exit problem that has not been satisfactorily answered in case $\sigma^2=\infty$. Denote by $\sigma^x_B$ the first entrance time   into  a set $B$ of the r.w. starting at  $x$
and by  $V_{{\rm ds}}(x)$  the renewal function for the weakly  descending ladder height process of the r.w.
It turns out that if $m_+/m\to0$,  then the ratio $V_{\rm ds}(x)/a(x)$ is s.v. at infinity and that
uniformly   for  $1 \leq x\leq  R$,
\beqn\label{intr1}
P[\sigma^x_{[R,\infty)} < \sigma^x_{(-\infty,0]}] \sim V_{{\rm ds}}(x-1)/V_{{\rm ds}}(R) 
\eeqn
as $R\to\infty$  (Proposition \ref{Prop46}).


 
  For L\'evy processes having no upwards jumps  there is an identity for the corresponding probability  (cf. \cite[Section 9.4]{D_L}), and (\ref{intr1}) is  an exact analogue of it for the r.w. In \cite{Uexit}  the asymptotic equivalence (\ref{intr1}) is obtained under some auxiliary conditions other than  $m_+/m\to0$, and the related matters are addressed.
 
 Combined with \cite[Theorem 1.1]{Unote}  formula  (\ref{intr1}) will lead to the following result. 

\begin{theorem}\label{Thm8} Suppose $\sigma^2=\infty$ and $m_+/m\to 0$.  Then
\v2 
{\rm (i)} \;\; $\ell_+(x):= \int_0^x P[Z>t]dt$ is s.v. and $a(x) \sim V_{\rm ds}(x)/\ell_+(x)$\;\;($x\to\infty$).
\v2
{\rm (ii)} \;\, for a constant  $\alpha \geq 1$ 
the following are equivalent
\beqn\label{a-e}
\begin{array}{ll}
{\rm (a)} \;\; P[S_n>0] \to 1/\alpha;\\[1mm]
{\rm (b)} \;\; m_-(x) \; \mbox{ is regularly varying with index  $2-\alpha$;}\\[1mm] 
{\rm (c)} \;\; P[\sigma^x_{[R,\infty)}< \sigma^x_{(-\infty,0]}] \to \lam^{\alpha -1} \;\; \mbox{as}\;\;x/R\to\lambda \quad \mbox{for each} \;\;0 < \lam < 1;\qquad\\[1mm]
{\rm (d)} \;\; V_{\rm ds}(x)\; \mbox{ is regularly varying with index  $\alpha-1$,}\\[1mm] 
{\rm (e)} \;\; a(x)\; \mbox{is regularly varying with index  $\alpha-1$;} 
\end{array}
\eeqn
and each of   (a) to (e) implies that\,  
\beqn\label{a/C/m}
a(x) \sim C_\alpha x/m(x),  \,\quad \mbox{where}\;\;  C_\alpha=1/\Ga(3-\alpha)\Ga(\alpha).
\eeqn 
\end{theorem}

If $\sigma^2<\infty$ (entailing $\sup m(x)<\infty $), then   all the statements  (a) to (e) of (\ref{a-e})  are  valid with $\alpha=2$, and   (i) and (\ref{a/C/m}) are also valid but with   $a(x)$  replaced  by $2a(x)$.

It is not clear how the conditions (a) to (e) are related when the condition $m_+/m\to0$ fails. For $\alpha>1$, however, this latter condition seems to be  a reasonable restriction--- without assuming any additional one on the tails of $F$---for a description of relations of the  quantities of interest. The situation  is somewhat different for $\al=1$ where (a) and (c) are equivalent to each other  and entail (b)  and (d) for every r.w. with 
$\sigma^2=\infty$ (see Remark \ref{Rem11} below and/or  \cite{Uexit}).
 
\begin{corollary}\label{Cor9} \, 
{\rm (i)} If  $m_+/m\to 0$, then $F$  belongs to the domain of attraction of the normal law if and only if 
either one of  conditions (a) to (e) holds with $\alpha =2$. 

{\rm  (ii)} Suppose  $\mu_+/\mu \to0$ and $1<\alpha <2$.   Then $F$  belongs to the domain of attraction of a  stable law of exponent $\alpha$  if and only if either one of  (a) to (e) holds. 
 
{\rm [(ii)  follows immediately  from  Theorem \ref{Thm8} because of the well-known characterization theorem for the domain of attraction. As for  (i) see (2) in Appendix (A).]}
\end{corollary}

  The condition $\lim P[S_n>0]=\rho$\;---apparently stronger than the so-called Spitzer's condition 
 $n^{-1} \sum_{k=1}^n P[S_k>0] \to \rho$
 but equivalent to it (cf. \cite{D_L}, \cite{BD})---plays an  important role in the study of fluctuations of r.w.s. This condition holds if $F$ belongs to the domain of attraction of a stable law but not to its sub-class consisting of those distributions for which the exponent of the limit stable law equals one and the positive and negative tails are asymptotically equivalent. It is of interest to find a condition for the reverse implication to be true. Doney \cite{D} and Emery \cite{Em} observed that for $\alpha>1$, Spitzer's condition (a) implies the domain of attraction if one tail outweighs the other overwhelmingly---conditions much stronger than $\mu_+/ \mu \to 0$.    The part (ii) of Corollary \ref{Cor9} sharpens their results.  
 For $\alpha=1$ the assertion corresponding to (ii) fails since the slow variation of $\eta_-$ does not imply
 the regular variation of $\mu_-$. 
 
   
   \begin{remark}\label{Rem10} \;   Of Theorem \ref{Thm8}  the assertion  that  (a) is  equivalent to (b) under $m_+/m\to 0$  is essentially  Theorem 1.1 of \cite{Unote}  (see Section 7.4 of the present paper), so that  Corollary \ref{Cor9}  restricted to  conditions (a) and (b)  (with \lq\lq each of 
   (a) to (e)''  replaced by \lq\lq (a)'' in the statements) 
 should be considered as its corollary though not stated  in \cite{Unote} as such. For $\al=1$ the equivalence of (a) and (c) is verified in \cite{KM} (see Remark \ref{Rem11} below).  The equivalence of (c) to (e) follows from  what is stated at  (\ref{intr1}) and the essential content supplemented   by  Theorem \ref{Thm8} is that  (a) and (b) are equivalent to the conditions (c) to  (e) in case $\alpha>1$. 
 \end{remark}   
  \v2
 \begin{remark}\label{Rem11}\,   Let $\alpha=1$ in (\ref{a-e}) (in particular (a) becomes $P[S_n>0]\to 1$).  Kesten and Maller  \cite{KM} obtain  a purely analytic  criterion for  (a) to hold  for   r.w.s on $\mathbb{R}$ which entails  the equivalence of (a) and (c) of Theorem \ref{Thm8} (for $\alpha=1$).      They do not assume the recurrence nor any moment condition and below we state a consequence of their results relevant to the case $\alpha=1$ of Theorem \ref{Thm8}. 
 Recall $A(x) = \int_0^x \big[ \mu_+(y)-\mu_-(y)\big]dy.$
  By Theorem 2.1 and Lemma 4.3 of \cite{KM} it follows that (a) is equivalent to $S_n \stackrel{P}\to \infty$ (the symbol \lq$\stackrel{P}\to$' means \lq convergence in probability'), and  whenever 
  $\sigma^2=\infty$,  (a) holds if and only if 
\v2
 (c$'$)  \quad $\exists x_0>0, \; A(x)>0$\; for \;$x\geq x_0$\; and \;  ${\displaystyle \frac{xF(-x)}{A(x)} \to 0,}$
 \v2\noindent
while it is noted in Remark (iii)  to Theorem 2.1 of \cite{KM} that  under $\sigma^2 =\infty$  a result of \cite{GM}  may be equivalently stated as (c) $\Leftrightarrow$ (c$'$).   It therefore follows that 
\beqn\label{A/KM2}
\begin{array}{ll}
\mbox{\it  if $\sigma^2=\infty$,     (a), (c)  (with $\alpha=1$)  and  (c$'$) are equivalent  to one another} \\
\mbox{\it and   necessary and sufficient in order that $S_n\to \infty$ in probability.}
\end{array}
\eeqn
 If $EX=0$,  $(c')$ is rephrased as $\lim \big\{\eta_-(x)-\eta_+(x)\big\}/x\mu_-(x) =\infty$, which  obviously  implies that  $\eta_-(x)/ x\mu_-(x) \to\infty$, hence  $\eta_-$ is s.v. at infinity. \,
\end{remark}  
   
 In addition to the signs $\sim$, $\asymp$ and $\lfloor \cdot \rfloor$ that have already been introduced  we shall use  $\wedge$ and $\vee$ to denote the minimum and  maximum of two terms on their two sides. By $C, C', C_1, etc.$ we denote the generic positive finite constants whose values may change 
 from line to line.
 
\vskip2mm
In the next section we derive some fundamental facts about $a(x)$ as well as  the functionals introduced above, which incidentally yield  (i) of Theorem \ref{th:1_2} (see Lemmas \ref{Lem14}, \ref{Lem17})
and the sufficiency for (H) of (\ref{H1}) mentioned above; also the lower bound of Theorem \ref{th:1} is verified under a certain side condition. The proof of  Theorem \ref{th:1} is more involved and given in Section 3, in which  we also prove  Theorems \ref{th:1_2} and \ref{th:1_3}.  Proposition \ref{Prop5} and Theorem \ref{Thm6}  are proved in Sections 4 and 5, respectively. Theorem \ref{thm7} 
is proved in Section 6.  Applications are discussed  in Section 7. In Section 8 we give two examples: for the first one the r.w. is supposed in the domain of attraction of a stable law with exponent $1\leq \alpha \leq 2$  and some precise asymptotic forms of
$a(x)$ as $|x|\to \infty$ are exhibited and demonstrated, while the second one reveals how $a(x)$ can  irregularly behave  for large values of $x$.  Section 9 consists of  Appendixes (A) and (B): in (A)   some elementary relations involving s.v. functions are shown that are used in verification of Corollary \ref{Cor9}, while in  (B)  we prove a  renewal theorem  for a r.s. distribution on $\{0, 1,2,\ldots\}$.

\section{Preliminaries}

In this section we first present some easy facts 
and then give several lemmas, in particular,   Lemmas \ref{Lem14} and   \ref{Lem18} which together assert that $\bar a(x)\asymp x/m(x)$ under the last inequality in (\ref{H1}) and  whose proofs    involve typical  arguments  that are implicitly  used in Sections 3 to 5. 

As in \cite{Unote} we bring in the following functionals of $F$  in addition to those introduced in Section 1:
\[\tilde c(x) =\frac1{x}\int_0^x y^2\mu(y)dy, \quad \tilde m(x)=\frac2{x}\int_0^x y\eta(y)dy,
\]
\[
h_\ep(x)=\int_0^{\ep x} y\big[\mu(y)-\mu(\pi x+y)\big]dy \quad (0<\ep \leq \pi/2);
\]
also    $\tilde c_\pm(x)$ and  $\tilde m_\pm(x)$ are  defined with $\mu_\pm$ in place of $\mu$, so that   
$\tilde c(x)=\tilde  c_-(x)+ \tilde c_+(x)$, etc.  
By our basic hypothesis (\ref{X0}) $x\eta(x)$, $c(x)$ and $h_\ep(x)$ tend to infinity as $x\to \infty$; $\eta$, $c$ and  $h_\ep$ are monotone.
Note that $\Im\psi(t) =- t \gamma(t)$  as is easily checked by integrating by parts the integral $\int (\sin tx - tx) dF(x)$.
Here and throughout  the rest of this section as well as Sections 3 to 5  $x\geq 0$ and $t\geq 0$.  We shall be  concerned  with the behaviour of these functions  as 
$x\to\infty$ or $t\downarrow 0$ and  omit  \lq\lq$x\to\infty$'' or \lq\lq$t\downarrow 0$'' when it is obvious.

As noted previously the function $m$ admits the  decomposition
\[
m(x) = x\eta(x) +c(x).
\]
$m$ is a rather tractable function: increasing and concave, hence subadditive and 
\[
\mbox{for any}\; k>1,\quad  m(kx) \leq k m(x)\quad\mbox{and}\quad  m(x)/k\leq m(x/k) ,
\]
 while  $c$, though increasing, may vary quite differently. 
The ratio  $c(x)/m(x)$ may converges to 0 or to 1 as $x\to\infty$ depending on   $\mu$   and   possibly  oscillates asymptotically between $0$ and $1$;  and 
$$ c(kx) =  k^2\int_{0}^{x} \mu(ku)udu\leq  k^2 c(x) \quad   (k>1), $$
where the factor $k^2$ cannot be replaced  by $o(k^2)$ for  the upper bound to be valid (cf. Section 8.2).
We also have
\[
 \tilde m(x) = x\eta(x) +\tilde c(x)
 \]
and $\tilde m$ is increasing and concave. Clearly $\tilde c(x)< c(x)$, hence  $\tilde m(x)\leq m(x)$.

It  holds  that
\[
\bar a(x)= \frac1{2\pi}\int_{-\pi}^\pi \Re\frac{1}{1-\psi(t)} (1-\cos xt)dt.
\]
(cf. \cite[Eq(28.2)]{S}).  Recalling
$1-\psi (t) = t \al(t) + i t\gamma(t)$ we have
\[
\frac1{1-\psi(t)}=\frac{\al(t) -i\gamma(t)}{\al^2(t)+ \gamma^2(t)}\cdot \frac{1}{t}.
\]
Hence
\beqn\label{**a}
\bar a(x) = \frac1{\pi }\int_0^\pi \frac{\al(t)}{[\al^2(t) +\gamma^2(t)] t}(1-\cos xt)dt. 
\eeqn
Note that   $\al_\pm(t)$ and $\beta_\pm(t)$ are all positive (for $t>0$); 
 by Fatou's
lemma $\liminf \al(t)/t =\liminf t^{-2}\int_0^\infty (1-\cos tx)d(-\mu(x))\geq \frac12 \sigma^2$, so that $\al(t)/t \to \infty$ under the present setting. 

In order to find asymptotics of $\bar a$ we need to know asymptotics of $\al(t)$ and $\gamma(t)$ as $t\downarrow 0$  (which entail those of $\al_\pm$ and $\beta_\pm$ as functionals of $\mu_\pm$).
 Although  the arguments given for the first two lemmas are virtually the same as in \cite{Unote}, we give the  full proofs  since  some of constants  in \cite{Unote} are   wrong or inadequate for the present need and need to be corrected---the values of the constants involved  are  not significant  in \cite{Unote} but turn out to be of   crucial importance  in our proof of Theorem \ref{th:1}.  Let $0<t\leq \pi$ in the following lemmas.
\begin{lemma}\label{Lem12} For $0<\ep\leq \pi/2$,
\beqn\label{h1} 
[\ep^{-1}\sin \ep]h_\ep(1/t) <   \al(t)/t < [ \pi^2 c(1/t)]\wedge m(1/t).
\eeqn
\end{lemma}
\begin{proof}   \,  By monotonicity of $\mu$ it follows that 
\[
\al(t) > \bigg( \int_0^{\ep/t} + \int_{\pi/t}^{(\pi+\ep)/t}\bigg) \mu(z)\sin tz \,dz = \int_0^{\ep/t} \big[\mu(z)-\mu(\pi/t+z) \big]\sin tz\,dz,
\]
which by $\sin tz \geq \ep^{-1}(\sin \ep)tz$ ($tz \leq\ep$) shows the first inequality of the lemma.
  Obviously we have $\al(t)<\int_0^{\pi/t} \mu(z)\sin tz \,dz \leq tc(\pi/t) \leq \pi^2 tc(1/t)$. On the other hand  splitting this same integral at $1/t$ we see $\al(t)< tc(1/t) +\eta(1/t) =tm(1/t)$. Thus   the second inequality
   follows. 
  \end{proof} 

\vskip2mm
\begin{lemma}\label{Lem13}
\beqn\label{m1}
  \begin{array} {ll}  (a) \quad \frac{1}{2}  \tilde m(1/t)t \leq \beta(t) \leq 2\tilde m(1/t)t,\\[2mm]
 (b) \quad  \frac{1}{3}m(1/t)t \leq \al(t) + \beta(t)\leq 3 m(1/t)t. 
 \end{array}
\eeqn
\end{lemma}
\begin{proof}  
Integrating by parts and using the inequality $\sin u \geq (2/\pi)u$ ($u<\frac12 \pi$) in turn we see 
\[
\beta(t) =t\int_0^\infty \eta(x)\sin tx\,dx \geq \frac{2t^2}{\pi} \int_0^{\pi/2t} \eta(x)x\,dx+ t \int_{\pi/2t}^\infty \eta(x)\sin tx\,dx.
\]
Observing that the first term   of the last member  equals  $\frac12 t \tilde m(\pi/2t) \geq \frac12 t\tilde m(1/t)$ and the second one equals $-\int_{\pi/2t}^\infty  \mu(x)\cos txdx >0$ we obtain  the left-hand inequality of (\ref{m1}a). As for the right-hand one of (\ref{m1}a) we split   the defining integral of $\beta$ at $1/t$ to see  that  $\beta(t) \leq \frac12 t \tilde c(1/t) + 2\eta(1/t) \leq 2t\tilde m(1/t)$. The upper bound of (\ref{m1}b) is immediate from those in  (\ref{H1}) and (a) just proved since $\tilde m(x)\leq m(x)$. To verify the lower bound use 
 (\ref{h1}) and  the inequalities $h_1(x)\geq c(x)-\frac12 x^2\mu(x)$ and $\sin 1> 5/6$
 to obtain
 $$[\al(t)+\beta(t)]/t >  (5/6)\big[c(1/t) -\mu(1/t)/2t^2\big] +2^{-1} \big[\tilde c(1/t) +\eta(t)/t\big].$$
 By $x^2\mu(x) \leq [2c(x)]\wedge [3\tilde c(x)]$ it follows that $ \frac56\mu(1/t)/2t^2 \leq \frac12 c(1/t) + \frac12 \tilde c(1/t) $ and hence $\al(t)+\beta(t)  > \big[\frac26 c(1/t) + \frac12\eta(1/t)/t\big]t> \frac13 m(1/t)t $ as desired.
\end{proof} 
\vskip2mm

For $t > 0$ define
\[
f(t) = \frac{1}{t^2 m^2(1/t)} \quad\mbox{and }\quad f^\circ(t) = \frac{1}{\al^2(t)+\gamma^2(t)}.
\]
Observe   
\beqn\label{*c}
\bigg(\frac{x}{m(x)}\bigg)' = \frac{c(x)}{m^2(x)},
\eeqn
hence  $x/m(x)$ is increasing and  $f(t)$ is decreasing. 
\begin{lemma} \label{Lem14}\, 
\beqn\label{eqL2.1}
 \int_0^\pi \frac{f(t) \al(t)}{t} (1-\cos xt) dt \leq C\frac{x}{m(x)},
\eeqn
 for some universal constant $C$. In particular if  $\liminf  f(t)[\al^2(t) +\gamma^2(t)] > \de$ for some  \, $\de>0$, then   $\bar a(x) < C\de^{-1} x/m(x)$ for all sufficiently large $x$. 
\end{lemma} 
\begin{proof}  \, We break the  integral on the RHS of   (\ref{eqL2.1})  into two parts 
\[
 J(x) =\int_0^{\pi/2x}\frac{f(t)\al(t)}{t}(1-\cos xt) dt   \quad \mbox{and} \quad      K(x) =\int_{\pi/2x}^\pi \frac{f(t)\al(t)}{t}(1-\cos xt)  dt.
\]
 Using  $\al(t)\leq \pi^2 c(1/t)t$ we observe
\beqn\label{K*}
0\leq \frac{K(x)}{\pi^2} \leq \int_{\pi/2x}^\pi \frac{2c(1/t)}{t^2m^2(1/t)}dt =\int_{1/\pi}^{2 x/\pi}\frac{2c(y)}{m^2(y)}dy  =  \bigg[\frac{2y}{m(y)}\bigg]_{y=1/\pi}^{2x/\pi}  < \frac{2x}{m(x)},
\eeqn
where we have applied  $m(2x/\pi) \geq m(x)2/\pi$ for the last inequality. 
Similarly
\[
\frac{J(x)}{\pi^2} \leq  \frac{x^2}2\int_0^{\pi/2x}f(t)c(1/t)t^2 dt =\frac{x^2}2 \int_{2x/\pi}^\infty \frac{c(y)}{y^2m^2(y)}dy
\]
 and, observing
\beqn\label{1/xm}
 \int_{x}^\infty \frac{c(y)}{y^2m^2(y)}dy\leq \int_x^\infty \frac{dy}{y^2m(y)} \leq 
  \frac1{x m(x)},
\eeqn
 we have $J(x) \leq \frac18\pi^4 x/m(x)$, finishing the proof. 
\end{proof} 

\vskip2mm
 If there exists a  constant $B_0>0$ such that for  all $t$ small enough,
 \beqn\label{B0}
 \al (t)\geq B_0c(1/t)t,
 \eeqn
 then  the  estimation  of  $\bar a$ becomes much easier. Unfortunately  condition (\ref{B0}) may fail to hold in general: in fact  the ratio $\al(t)/[c(1/t)t]$ may oscillate between $1-\ep$ and $\ep$ infinitely many times for any $0<\ep <1$ (cf. Section 8.2). To cope with such situation the following lemma will be used in a crucial way. 

 \begin{lemma} \label{Lem15}~  For all $x>0$,
  \[
  h_\ep( x)\geq c(\ep x) -(2\pi)^{-1} \ep^2 x\big[\eta(\ep x) -\eta(\pi x +\ep x)\big].
  \]
 \end{lemma}
 \begin{proof}     On writing 
 $h_\ep(x) =c(\ep x) -\int_0^{\ep x}u\mu(\pi x+u)du$,  the integration by parts yields 
 \[
 h_\ep(x)- c(\ep x)   = -\int_0^{\ep x} \big[\eta(\pi  x+u)-\eta(\pi x+ \ep x) \big]du.
 \]
By monotonicity and convexity  of $\eta$ it follows that if $0<u\leq  \ep x$ (entailing $0\leq \ep x-u < \pi x$),
\[
0\leq \eta(\pi x+u)-\eta(\pi x+ \ep x) \leq \frac{\ep x-u}{\pi x}\big[\eta(\ep x) -\eta(\pi x+ \ep x)\big],
\]
and substitution   readily  leads to the inequality of the lemma.
\end{proof} 

\vskip2mm
\begin{lemma} \label{Lem16}  For any\,  $0<\de\leq 1$, $0<t<\pi $,  if $c(1/t) \geq \de m(1/t)$ and
 $s= \big[1\wedge (\de\pi)\big]  t$, then $\al(s)/s> \lam (\pi^{-1}\wedge\de)^2m(1/s)$ with  $\lam : =\frac12 \pi \sin 1>1$.
In particular if  $\delta:=\liminf c(x)/m(x) >0$, then  $\al(t)/t >(\pi^{-1}\wedge \de)^2  m(1/t)$ for all sufficiently small  $t>0$---so that  (H) holds.
\end{lemma}
\vskip2mm
\begin{proof}   First suppose  $\de \pi \leq 1$. 
Take   $\ep= \de\pi$ in  Lemma \ref{Lem15}. Then  $\ep^{-1}\sin \ep \geq \lam 2/\pi$.  while  the premise of the first statement of the lemma  implies 
$$h_\ep(1/\ep t) > c(1/t) -(2\pi)^{-1}\ep m(1/t) \geq \frac12 \de  m(1/t)  \geq \frac12 \de \ep m(1/\ep t) = \frac12 \pi \de^2 m(1/\de\pi t),$$ 
whence by Lemma \ref{Lem12} $\al(s)/s> \lam\de^2m(1/s)$ for $s=\de\pi t$. Similarly, if  $\de \pi > 1$, then taking $\ep =1$ we  have $\al(t)/t > (\sin 1) (\de- (2\pi)^{-1}) m(1/t)> \lam \pi^{-2}m(1/t)$..
\end{proof}

\vskip2mm
\begin{lemma} \label{Lem17} \, If  (H) holds, then   
$${\rm (H')} \quad \al(t)+|\ga(t)| >[\pi^{-2}\wedge {\textstyle \frac13} \de_H] tm(1/t) \qquad \mbox{for all sufficiently small }\;\; t>0,$$  
and 
\beqn\label{eqL2.6}
\int_0^1 \frac{\beta(t)}{\alpha^2(t) +\ga^2(t)} dt <\infty.
\eeqn
\end{lemma}

(H$'$) together with Lemma \ref{Lem13} ensures that $f^\circ$ is almost decreasing.
\v2
\pf   If $\frac12 m(1/t) \leq  c(1/t)$, then by  Lemma \ref{Lem16} $\al(t) >\frac14  tm(1/t)$ entailing ${\rm (H')}$  (for this $t$), while if  $\frac12 m(1/t) > c(1/t)$, then $\eta(1/t)/t > 2^{-1}m(1/t)$. Hence (H) implies ${\rm (H')}$.

By Lemma \ref{Lem13}  ${\rm (H')}$ implies that the integral in (\ref{eqL2.6}) is at most a constant multiple of
\[\int_0^1 \frac{\tilde m(1/t)}{tm^2(1/t)} dt = \int_1^\infty  \frac{\tilde m(x)}{xm^2(x)} dx = \int_1^\infty\frac{2\int_0^x y\eta(y)dy}{x^2m^2(x)}dx.$$
Interchanging the order of integration one deduces that the last member above equals 
 $$\int_1^\infty\frac{2\int_0^1 y\eta(y)dy}{x^2m^2(x)}dx+ \int_1^\infty\ 2y\eta(y) dy\int_y^\infty \frac{dx}{x^2m^2(x)} < C +\int_1^\infty \frac{2\eta(y)}{m^2(y)}dy <\infty,
 \]
 where  
 the monotonicity of $m$   is used for the first inequality.
Thus (\ref{eqL2.6}) is verified.  \qed



\vskip2mm
\begin{lemma}\label{Lem18}  \,  If $\delta:=\liminf c(x)/m(x) >0$, then   for some constant $C >0$  depending only on  $\delta$,
\[
C^{-1} x/m(x) \leq \bar a(x) \leq C x/m(x) \quad \mbox{for all  $x$  large enough}.
\]
\end{lemma}
 \vskip2mm
\begin{proof}  \, By Lemmas \ref{Lem16} and \ref{Lem17} condition   ${\rm (H')}$  is satisfied, provided $\delta>0$. Hence  the upper bound follows from  Lemma \ref{Lem14}.
  
  Although the lower bound is given by Theorem \ref{th:1} which will be shown independently of Lemma \ref{Lem18} in the next section,  we here provide a direct proof. 
 Let  $K(x)$ be as in the proof of Lemma \ref{Lem14}. By Lemma \ref{Lem16} we may suppose that $\al(t)/t \geq B_1c(1/t)$ with a constant $B_1>0$. 
  Since both $c(1/t)$ and $f(t)$ are decreasing and hence so is their product,  we see 
\[
\frac{K(x)}{B_1} \geq \int_{\pi/2x}^\pi  f(t)c(1/t) (1-\cos xt) dt \geq  \int_{\pi/2x}^{\pi}\frac{c(1/t)}{t^2m^2(1/t)} dt,
\]
from which we deduce, as in (\ref{K*}),    that
\beqn\label{I}
\frac{K(x)}{B_1} \geq   \bigg[\frac{y}{m(y)}\bigg]_{y=1/\pi}^{2x/\pi} \geq \frac2{\pi}\cdot\frac{x}{m(x)} -\frac{1/\pi}{m(1/\pi)}.
\eeqn
Thus the desired lower bound   follows. 
\end{proof} 



\begin{lemma} \label{Lem19}\, Suppose $0=\liminf c(x)/m(x) <\limsup c(x)/m(x)$. Then
for any $\ep>0$ small enough there exists an unbounded sequence $x_n >0$ such that 
$$c(x_n) =\ep m(x_n) \quad\mbox{and}\quad \alpha(t)/t \geq 2^{-1} \ep^2 m(x_n) \quad\mbox{for}\quad  0< t \leq 1/x_n.$$
\end{lemma}
\v2\n
\pf\,
Put $\lam(x)= c(x)/m(x)$ and $\de =\frac12 \limsup \lam(x)$. Let $0<\ep<\de^2$.  Then there exists two sequences $x_n$ and $x_n'$ such that  $x_n \to \infty$, $x_n'<x_n$, 
\beqn\label{lam/de} 
\lam(x_n)= \ep \leq \lam(x)   \;\;  \mbox{for}\;\;  x_n'<x<x_n \;\;\; \mbox{and}\;\; \; \de= \lam(x_n').
\eeqn
Observing   $x_n'/x_n \leq m(x_n')/m(x_n) <\lam(x_n)/\lam(x_n') = \ep/\de<\de$, we see   $1/x_n<\de/x_n'$
and then by using (\ref{lam/de})
$$\frac{\lam(x_n)}{x_n} <\frac1{x_n} <\frac{\lam(x_n')}{x_n'}.$$ 
Hence  the intermediate  value theorem ensures that      there exists a solution of the equation $\lam(x)= x/x_n$ in the interval $x'_n<x<x_n$.  Let $y_n$ be the largest solution   and  put $\ep_n= y_n/x_n$. Then  
$$\ep < \lam(y_n)= c(y_n)/m(y_n) =\ep_n <1 \quad\mbox{and}\quad  y_n =\ep_n x_n.$$ 
Hence by Lemma \ref{Lem15}
$$h_{\ep_n}(x_n) \geq \Big[c(y_n) -\frac{\ep_n}{2\pi}m(y_n)\Big] =  \frac{2\pi-1}{2\pi} \ep_nm(\ep_n x_n) \geq \frac56 \ep^2 m(x_n). $$
Since  $h_\ep$ is non-decreasing,  for $0< t \leq 1/x_n$,
\beqn\label{h/m}
\alpha(t)/t \geq (\sin 1) h_{\ep_n}(1/t)> \frac56 h_{\ep_n}(x_n) \geq 2^{-1} \ep^2 m(x_n).
\eeqn 
Thus the proof is finished. \qed

\begin{lemma}\label{Lem20}  For any  $0<\de <1/\pi$ and $0<t<\pi$,  if  $c_+(1/t)/ m_+(1/t) \geq \de$,  then 
\beqn\label{al+ga/be}\alpha_+(s)+|\ga(s)| \geq 3^{-1}\delta^2 \beta(s) \quad (s=\de \pi t).
\eeqn
In particular if \, $\liminf  c_+(x)/m_+(x)>0$, then  (H) holds.
\end{lemma}
\v2\n
\pf \,  For any positive numbers  $\beta_{\pm}$ and $\de <2$, we have
\beqn\label{gel_ineq}
\de \beta_+ +|\beta_+ -\beta_-| \geq \frac12 \de (\beta_+ +\beta_-).
\eeqn
Under the premise of the first statement of the lemma,  
Lemma \ref{Lem16} applied with $\mu_-$ in place of $\mu$ shows that for $s=\de\pi t$,
  $\alpha_+(s)/s \geq  \frac12 (5/6)\pi\de^2 m_+(1/s)$.  This  combined with Lemma \ref{Lem13}(a)  entails  $\al_+(t) \geq \frac23 \de^2 \beta_+(t)$,  so that by (\ref{gel_ineq})   (\ref{al+ga/be}) follows.
  The second assertion is immediate from  (\ref{al+ga/be})   in view of Lemma \ref{Lem13}(b).  \qed


\begin{lemma}\label{Lem21}
\, For    (H) to hold, it is sufficient that 
\beqn\label{suff_H}
\limsup \frac{x[\eta_+(x)\wedge \eta_-(x)]}{m_+(x)\vee m_-(x)} <\frac14.  
\eeqn

[ $1/4$ can be replaced by a larger number $\leq 1/2$.]
\end{lemma}
\v2\n
\pf \, Fix a constant $0<\de<1/\pi$ arbitrarily and  link the variables $x$ and $s$ by $s=\de \pi/x$.
By (\ref{al+ga/be}) and   similar one for the case
 $c_-(x)\geq   \de m_-(x)$ it follows that  for each  $x$ 
\v2
$(*)$\; 
if either  $c_+(x)\geq \de m_+(x)$   or $c_-(x)\geq   \de m_-(x)$, 
then  $\alpha(s)+|\ga(s)| \geq 3^{-1} \de^2 \beta(s)$.
\v2
Put
\beqn\label{m+-}
\lam(x):=  \frac{\tilde m_+(x)\wedge \tilde m_-(x)}{\tilde m_+(x)\vee \tilde m_-(x)}, \quad \omega(t) 
= \frac{\beta_+(s) \wedge \beta_-(s)}{\beta_+(s) \vee \beta_-(s)}.
\eeqn
Then by Lemma  \ref{Lem13}  again  we have $\omega(s) \leq 4\lam(x)$ that entails
$$|\ga(s)| = \big[1- \omega(s)\big] \big(\beta_+(s)\vee\beta_-(s)\big) \geq  2^{-1} \big[1-4\lam(x)\big]\beta(s).$$ 
If  $c_+(x) < \de m_+(x)$  and $c_-(x) <   \de m_-(x)$ with  $\de>0$ small enough, then in the ratio under the $\limsup$ on the LHS  of   (\ref{suff_H})  replacing both $x\eta_\pm(x)$ and  $m_\pm$  by  $\tilde m_\pm(x)$ 
 reduces  its value not  much so that   (\ref{suff_H}) implies $1-4\lam(x) >\ep$ for some $\ep>0$ for all sufficiently large $x$. This combined with  $(*)$ shows
(\ref{suff_H}). \qed
\v2

The following lemma is crucial (see also \ref{eqL3.2})) in order to handle  the oscillating part of  the integrals defining $\alpha_\pm(t)$ and/or $\beta_{\pm}(t)$.
\begin{lemma}\label{Lem22} \, Let $0<t<s  \leq\pi$. Then 
\[
|\al(t)-\al(s)| \vee |\beta(t)-\beta(s)|  \leq \left\{ \begin{array}{ll}
 |s-t|c(\pi/t) + \pi t c(1/t), \\[2mm]
  \big[(2\sqrt \pi)\sqrt{(s-t)/t} \, \big] tm(1/t).
  \end{array}\right.
\]
\end{lemma}
If  (H) holds, then by Lemma \ref{Lem17}  the second bound entails that for some constant $C$  depending only on $\de_H$,
 $$|f^\circ(t) - f^{\circ} (s)| \leq C\sqrt{(s-t)/t}\,  f^\circ (t).$$
\begin{proof}   
By  definition  $\beta(t)-\beta(s) =   \int_0^\infty \mu(z)(\cos sz - \cos  tz)  dz$. For any positive constant $r$,
$$ \bigg|\int_0^{r/t} \mu(z)\big(\cos sz - \cos  tz\big)  dz\bigg|\leq |s-t|\int_0^{r/t}z\mu(z)dz = (s-t)c(r/t), 
$$
and
$$ \bigg|\int_{r/t}^\infty \mu(z)\big(\cos sz - \cos  tz\big)  dz\bigg|\leq \bigg(\frac{\pi}{s}+\frac{\pi}{t}\bigg)\mu(r/t)\leq \frac{\pi}{r^2} tc(r/t).$$
Obviously we have the corresponding bound  for $\alpha(t) -\alpha(s)$.
Taking $r=1$ yields the first bound of the lemma. For $r>1$,  we have  $c(r/t) \leq m(r/t) \leq rm(1/t)$ so that
$$|\al(t)-\al(s)| \vee |\beta(t)-\beta(s)|  \leq \big[ r|s-t| + \pi r^{-1}t\big] m(1/t).$$
Thus taking $r=  \sqrt{\pi t/ |s-t|}$ we obtain the second inequality. 
\end{proof}

\section{Proofs of  Theorems \ref{th:1} to \ref{th:1_3}.}
 The main content of this section consists of the proof of Theorem \ref{th:1}.  Theorem \ref{th:1_2} is verified after it: the part (i) of  Theorem \ref{th:1_2} is virtually proved in the preceding section, whereas   the part (ii) is essentially the corollary of the proof of Theorem \ref{th:1}. The proof of Theorem \ref{th:1_3}, which partly uses Theorem  \ref{th:1_2}(ii),  is given at the end of the section.  

\v2
{\bf 3.1.} {\sc Proof of Theorem \ref{th:1}.} 
\v2
  By virtue of the right-hand inequality of (\ref{m1}b) $f^\circ(t) \geq \frac{1}{9}f(t)$ and for the present purpose it suffices to bound the integral in (\ref{eqL2.1}) from below by a positive multiple of $x/m(x)$. 
   For a lower bound  of it we take  the contribution  from the interval $\pi/2x< t<1$.   We also employ the lower bound 
   $\al(t) \geq \int_0^{2\pi/t} \mu(z) \sin tz\, dz$ and write down the resulting inequality as follows: with  the constant $B= (9\pi)^{-1}$
      \beq
 \frac{\bar a(x)}{B} &\geq&  \int_{\pi/2x}^1 \frac{f(t) \al(t)}{t} (1-\cos xt) dt \\
 &\geq&  \int_{\pi/2x}^1 \frac{f(t)}{t} (1-\cos xt) dt\int_0^{2\pi/t} \mu(z) \sin tz\, dz\\
&=& K_{ I}(x)+K_{I\!I}(x) + K_{I\!I\!I}(x),
 \eeq
 where 
 \[
 K_I(x)= \int_{\pi/2x}^1 f(t)\frac{dt}{t}\int_0^{\pi/2t} \mu(z) \sin tz \, dz,
 \]
 \[
  K_{I\!I}(x) = \int_{\pi/2x}^1 f(t)\frac{dt}{t}\int_{\pi/2t}^{2\pi/t} \mu(z) \sin tz \, dz
  \]
 and
\[
K_{I\!I\!I}(x) = \int_{\pi/2x}^1 f(t)(-\cos xt)\frac{dt}{t}\int_0^{2\pi/t} \mu(z) \sin tz \, dz.
\]

\begin{lemma}\label{Lem23}
\[
K_I(x) \geq \frac{5}{3\pi} \cdot \frac{x}{m(x)} - \frac{1}{m(1)}.
\]
\end{lemma}
\begin{proof}  \, By $\sin 1 \geq  5/6$ it follows that  $\int_0^{\pi/2t} \mu(z) \sin tz \, dz > \int_0^{1/t} \mu(z) \sin tz \, dz \geq \frac56 tc(1/t)$, and hence
\[
K_I(x)> \frac56 \int_{\pi/2x}^1 \frac{c(1/t)}{t^2 m^2(1/t)}dt \geq \frac56\int_{1}^{2x/\pi}\frac{c(z)}{m^2(z)}dz \geq\frac{5}{3\pi} \cdot \frac{x}{m(2x/\pi)} - \frac{5}{6m(1)},
\]
implying the inequality of the lemma because of the monotonicity of $m$.
\end{proof}

\begin{lemma}\label{Lem24}
$K_{I\!I\!I}(x) \geq -2\pi f(1/2)/x$ for all sufficiently large  $x$.
\end{lemma}
\begin{proof}  \,   Put $g(t)= t^{-1}\int_0^{2\pi/t} \mu(z) \sin tz \, dz$. One sees that  $0\leq g(1)<1$.  We claim  that  $g$ is  decreasing. Observe that 
$$\frac{d}{dt}t^{-1}\int_0^{2\pi/t} \sin tz \, dz = t^{-2}\int_0^{2\pi/t} (tz\cos tz -\sin tz )dz=0$$
 and 
\[
g'(t) =  \frac1{t^2} \int_0^{2\pi/t}\mu(z)(tz\cos tz  -\sin tz)dz,
\]
and that $u\cos u  - \sin u  <0$ for $0<u<\pi$ and the integrand  of the last integral  has unique zero  in the open interval $(0,2\pi/t)$. Then the monotonicity of $\mu$ leads to 
$
g'(t) <   0  
$
as claimed. Now $f$ being decreasing, it follows that  $K_{I\!I\!I}(x) \geq -\int_{(2n+\frac12)\pi/x}^1f(t)g(t)dt$ for any integer $n$ such that $(2n+\frac12)\pi/x \leq 1$.
Since one can  choose $n$ so that $0\leq 1-(2n+\frac12) \pi/x \leq 2\pi/x$,   the inequality of the lemma
follows. \end{proof} 

\begin{lemma}\label{Lem25}
\[
K_{I\!I }(x) \geq -\,\frac{4}{3\pi} \cdot \frac{x}{m(x)}.
\]
\end{lemma}
\begin{proof}  \,  
Since $\mu$ is non-increasing, we have  $ \int_{\pi/2t}^{2\pi/t} \mu(z) \sin tz \, dz \geq \int_{3\pi/2t}^{2\pi/t} \mu(z) \sin tz \, dz$, so that
\[
K_{I\!I}(x) \geq  \int_{\pi/2x}^1 f(t)\frac{dt}{t} \int_{3\pi/2t}^{2\pi/t} \mu(z) \sin tz \, dz.
\]
We wish to make integration by $t$ first. Observe that the region of the double integral is 
 included in 
\[
\{3\pi/2 \leq z\leq 4x; 3\pi/2z < t<2\pi/z\},
\]
where the integrand of the inner integral is negative.  Hence
\[
K_{I\!I}(x) \geq \int_{3\pi/2}^{4x}\mu(z)   dz \int_{3\pi/2z}^{2\pi/z}  f(t)\frac{ \sin tz}{t}dt.
\]
Put
 \[
 \lambda= 3\pi/2.
 \]
 Then, since $f$ is decreasing, the RHS is further bounded below by 
\[
\int_{\lam}^{4x}\mu(z)  f(\lambda/z) dz \int_{3\pi/2z}^{2\pi/z}  \frac{ \sin tz}{t}dt.
\]
The inner integral being equal to $\int_{3\pi/2}^{2\pi} \sin u du/u $ which is larger than $- 2/3\pi=-1/\lambda$,  after  a change of variable we obtain
\[
K_{I\!I}(x) \geq -  \int_{1}^{4x/\lambda}\mu(\lambda z)  f(1/z) dz\geq -  \int_{1}^{x}\mu(\lambda z)  f(1/z) dz.
\]
Recall $ f(1/z)  =z^2/m^2(z)$.
Since $\int_x^\infty \mu(\lambda z)dz =\lambda^{-1} \eta(\lambda x)$ and,
 by integration by parts,
\beq
 - \int_{1}^{x}\mu(\lambda z)  f(1/z) dz  &=&  \frac1{\lambda}\bigg[ \frac{\eta(\lambda z)z^2}{m^2(z)} \bigg]_{z=1}^x - \frac2{\lambda} \int_{1}^x \frac{z \eta(\lambda z) c(z)}{m^3(z)}dz \\
 &\geq&  -\,\frac2{\lambda}   \int_{1}^x \frac{z \eta( z) c(z)}{m^3(z)}dz 
 - \frac{\eta(1)}{\lambda m^2(1)}.
 \eeq
Noting $z\eta(z) = m(z)- c(z)$ we have
\[
\frac{z \eta(z) c(z)}{m^3(z)} = \frac{c(z)}{m^2(z)}- 
\frac{c^2(z)}{m^3(z)}.
\]
Since $m(1) > \eta(1)$, we conclude
\[
K_{I\!I}(x) \geq  -\frac2{\lambda}\int_{1}^x \frac{c(z)}{m^2(z)}dz - \frac{1}{\lambda m(1)} = -\frac{4}{3\pi}  \frac{x}{m(x)} +\frac{1}{\lam m(1)},
\]
hence the inequality of the lemma.  
\end{proof} 

\textit{ Proof of Theorem \ref{th:1}. }   Combining Lemmas \ref{Lem23} to \ref{Lem25}   we obtain
\beqn\label{prT1}
\frac{\bar a(x)}B\geq  \int_{\pi/2 x}^{1} f(t)\al(t)\frac{1-\cos xt}{t}dt \geq \frac1{3\pi} \cdot \frac{x}{m(x)} -   \frac1{m(1)} +O(1/x),
\eeqn
 showing Theorem \ref{th:1}.




\v2
{\bf 3.2.} {\sc Proof of Theorem \ref{th:1_2}.} 
\v2
 The first part (i) of Theorem \ref{th:1_2} is obtained by combining Lemmas \ref{Lem14} and \ref{Lem17}. As for (ii) 
 its premise implies  that for any $\ep>0$ there exists  $\de>0$ such that $f^\circ(t) >\ep^{-1}f(t)$ for $0<t<\de$, which concludes the result in view of the second inequality of (\ref{prT1}), for   $\int_{\de}^1 f(t)\al(t)(1-\cos xt) dt/t$ is bounded for each $\de>0$.
\v2

{\bf 3.3.} {\sc Proof of Theorem \ref{th:1_3}.}

\begin{lemma}\label{Lem26} \, Suppose that there exists $p:=\lim m_+(x)/m(x)$. Then  as $t\downarrow 0$
$$ \beta_+(t) =  p  \beta(t) + o\big(t m(1/t)\big)  \quad \mbox{and} \quad \alpha_+(t) =  p \, \al(t) + o\big(t m(1/t)\big),$$

\end{lemma}
\begin{proof}  \, Take $\tilde\mu(x)\geq 0$ a non-decreasing summable function $\tilde\mu(x)$ $x\geq 0$,   let $\tilde m$, $\tilde \eta$, $\tilde\alpha$ and $\tilde\beta$ be the corresponding functions and $\tilde m/m \to1$.  It suffices to show that as $t\downarrow 0$
\beqn\label{be-be}
\beta(t)-\tilde\beta(t) =o(tm(1/t)) \quad \mbox{and} \quad  \al(t)-\tilde \al(t) =o(tm(1/t)). 
\eeqn
 Pick a positive number $M$ such that   $\sin M =1$. 
Then
$$\int_0^{M/t} (\mu- \tilde \mu) (y)(1-\cos ty) dy = -(\eta -\tilde \eta)(M/t)  +t\int_0^{M/t} (\eta - \tilde\eta)(y)\sin ty\, dy, $$
while, 
on writing  $x=1/t$,
$$\int_{M/t}^\infty \mu (y) \cos ty\, dy\leq \pi x\mu(Mx)\leq \frac{2 \pi}{ M^2x}c(Mx)\leq \frac{2 \pi}{ M x}m(x) = \frac{2\pi}{M}tm(1/t).$$
Since $\int_{M/t}^\infty (\mu -\tilde\mu) (y) dy = (\eta -\tilde\eta_-)(M/t)$, these together yield   
$$\beta(t)-\tilde\beta(t)= t\int_0^{M/t} (\eta -\tilde\eta)(y)\sin ty\, dy + tm(1/t)\times O(1/M)$$
 and 
integrating by parts the last integral 
leads to
$$\frac{\beta(t)-\tilde\beta(t)}{t} = (m- \tilde m)(M/t) - t\int_0^{M/t}(m-\tilde m)(y)\cos ty\,dy + m(1/t)\times O(1/M).$$
 If $\tilde m/m \to1$,  the first two terms on the right side  is  $o(m(1/t))$  for each $M$ fixed and  we conclude the first relation of (\ref{be-be}) since $M$ can be made arbitrarily  large.  
 
 The second one is verified in the same way but taking $M$ large with  $\sin M=0$.
  \end{proof}

\begin{lemma}\label{Lem27}
  Suppose that  $\lim |\ga(t)|/tm(1/t) =0$. If $c/m\to 0$,  then $\bar a(x) m(x)/x$ diverges to infinity,  and if $\liminf c(x)/m(x) =0$, then $\limsup \bar a(x) m(x)/x =\infty$.
    \end{lemma}
\begin{proof}  \,  The first half follows from Theorem  \ref{th:1_2}(ii) and Lemma \ref{Lem12}. For the proof of the second one 
 we apply the trivial lower bound
\beqn\label{LB_a}\pi\bar a(\pi x_n) =\int_0^\pi \frac{\alpha(t)(1-\cos \pi x_nt)}{[\alpha^2(t)+\ga^2(t)]t}dt 
\geq \int_{1/2x_n}^{1/x_n}\frac{\alpha(t) dt}{[\alpha^2(t)+\ga^2(t)]t}.
\eeqn
valid for any sequence $x_n\in \mathbb{Z}/\pi$.
Suppose that   $\limsup c(x)/m(x) >0$ in addition to $\liminf c(x)/m(x) =0$, so that we can take a constant $\ep>0$ and an unbounded sequence  $x_n$ as in  Lemma \ref{Lem19}, according to which 
 $$c(x_n)= \ep m(x_n) \;\; \mbox{and}\;\; \alpha(t)/t \geq 4^{-1} \ep^2 m(1/t) \quad  \mbox{if} \;\; 1/2x_n <t\leq 1/x_n.$$ 
Because of  the second assumption of the  lemma the second relation above leads to 
\beqn\label{ga/al}
|\ga(t)|=o(m(1/t)t) = o(\alpha(t)) \qquad (n\to\infty),
\eeqn
while, combined with (\ref{lam/de}) and Lemma \ref{Lem12},  the first one  shows that for $1/{2x_n}\leq t \leq  1/{x_n}$,
$$\alpha(t)/t \leq \pi^2c(1/t) \leq \pi^2 c(2x_n) \leq 4 \pi^2  c(x_n) =4 \pi^2 \ep m(x_n).$$  
These together yields
$$\frac{\alpha(t) }{\alpha^2(t)+\ga^2(t)} \sim \frac1{\alpha(t)}\geq \frac1{t}\cdot\frac{1}{4\pi^2\ep m(x_n)},$$
and substitution into   (\ref{LB_a})  shows that for all $n$ large enough
 $$\pi\bar a(\pi x_n) \geq \frac{1}{5\pi^2\ep m(x_n)}\int_{1/2x_n}^{1/x_n}\frac{dt}{t^2} =  \frac{x_n}{5\pi^2\ep m(x_n)}.$$ 
Thus $\bar a(x) x/m(x)$ is unbounded, $\ep$ being made arbitrarily small.  
\end{proof}

\begin{remark}\label{Rem28}  Suppose $m_+/m\to 1/2$. Then for the second assertion of Lemma \ref{Lem27} the condition $\liminf c/m = 0$  can be replaced by 
 \beqn\label{m+/c+}
 \liminf {c_+(x)}/{m_+(x)} =0.
 \eeqn
Indeed, we have  $\alpha_+(t)\sim \alpha_-(t)$ owing to Lemma \ref{Lem26}, and  by the first half of Lemma \ref{Lem27} we have only to consider the case
$ \limsup {c_+(x)}/{m_+(x)}>0$.  Then, for $1/2x_n\leq t\leq1/x_n$, the same arguments  verify $m_+(1/t)t= O(\alpha_+(t))$  so that we have (\ref{ga/al}), hence  $\alpha/[\alpha^2+\ga^2]\sim 1/2\alpha_+$. For the rest we can follow the proof above.
\end{remark}
\v2

\v2
\begin{remark} \label{Rem29}  Let  $p\neq 1/2$ in Lemma \ref{Lem26}.  Then 
\v2
(i)\quad $|\ga(t)| \asymp \beta(t)$  so that (H) holds  owing to Lemma \ref{Lem13}(ii);
\v2
(ii)\quad if  $c/m\to0$, or equivalently $\eta$ is s.v., then Theorem \ref{thm7} is applicable: indeed, the slow variation of $\eta$  entails $m(x)\sim x\eta(x)$ and $m_+(x) \sim p x\eta(x)$, hence  $\eta_+(x)\sim p\eta(x)$ by the monotone density  theorem and $A(x)/x\mu(x) \sim (1-2p)\eta(x)/x\mu(x) \to \infty$ or $\to-\infty$ according as  $\frac12 -p$ is positive or negative, so that $F$ is r.s.  
\end{remark}

{\it Proof of Theorem \ref{th:1_3}.} The assertion (i)  is the same as Lemma \ref{Lem26}; the necessity  of the condition  $\limsup x\eta(x)/m(x)<1$  asserted in (ii) is immediate from Lemma \ref{Lem27}. 
The necessity of the conditions  $\limsup x\eta_\pm(x)/m(x)<1$ are verified  in Remark \ref{Rem28}. \qed

 

 \v2
\section{Proof of Proposition \ref{Prop5}} 
  Suppose (H) to hold.  Since then $\bar a(x)\asymp x/m(x)$ by Theorems \ref{th:1} and \ref{th:1_2},   we may suppose  $R/2<x<R$ and 
 on writing  $\de = (R-x)/x$ the assertion to be shown can be  rephrased as
  \beqn\label{x/R}
 | \bar a(R)-\bar a(x)| \leq C \de^{1/4}[x/m(x)] \qquad (R/2<x< R).
  \eeqn
   Letting for  $M>1$ and $R/2 <x \leq R$ we  make the decomposition $\bar a(x)= u^M(x)+v_M(x)$ where
  $$u^M(x)= \int_0^{M/x} \frac{\alpha(t)f^\circ(t)}{t}(1- \cos xt) dt
  \quad\mbox{and}\quad v_M(x) = \int_{M/x}^\pi \frac{\alpha(t)f^\circ(t)}{t}(1 - \cos xt) dt.$$
By the inequality  $|\cos xt -\cos Rt| \leq |Rt-xt|$  it is then easy to show 
   \beqn\label{u/u}
|u^M(x)- u^M(R)|  
< C M\de [x/m(x)]
\eeqn
for a (universal) constant $C$  (as we shall see shortly), whereas to obtain a similar estimate for $v_M$ we cannot help exploit  the oscillation of $\cos xt$.      
To  the latter purpose   one may  seek  some appropriate smoothness of $\alpha(t)g(t)$,  which however is a property difficult to verify   because of   the intractable part  $\int_{1/t}^\infty \mu(y)\sin ty \,dy$ involved in the integral defining $\alpha(t)$.   In order to circumvent it,
 for each positive integer $n$ we bring in  the function
$$\alpha_{n}(t) := \int_{0}^{n\pi/t} \mu(y)\sin ty\, dy$$
and make use of  the inequalities $\alpha_{2n}(t)<\alpha(t) <\alpha_{2n+1}(t)$. 
If $v_M(x) \leq  v_M(R)$, then
\begin{eqnarray}
\label{1_prP1}
0\leq v_M(R) - v_M(x) &\leq&   
\int_{M/x}^\pi \big[\alpha_{2n+1}(t)-\alpha_{2n}(t)\big]\frac{f^\circ(t)(1-\cos Rt)}{t} dt \nonumber \\
&&+
\int_{M/x}^\pi \frac{\alpha_{2n}(t)f^\circ(t)}{t}(\cos Rt-\cos xt)dt \nonumber \\
&&+ \int_{M/R}^{M/x} \frac{\alpha_{2n+1}(t)f^\circ(t)}{t}(1-\cos Rt)dt  \nonumber \\
&=& I_n + I\!I_n + I\!I\!I_n \quad \mbox{(say)};
\end{eqnarray}
 and if $v_M(x)  >  v_M(R)$,  we have an analogous inequality. We consider  the first case only, the other one being similar.  By    $0\leq \alpha_n(t)\leq \alpha(t)\leq  \pi^2tc(1/t)$  and by Lemma \ref{Lem17} that may read $f^\circ(t) \leq C_1/[m^2(1/t)t^2]$  under  (H) we
 see that for any $n$ 
 \beqn\label{2_prP1}
 \int_{M/R}^{M/x} \frac{\alpha_{n}(t)f^\circ(t)}{t}dt \leq C \int_{x/M}^{R/M}\frac{c(y)dy}{m^2(y)} \leq C\frac{(R-x)/M}{m(x/M)}\leq  C\frac{\de x}{m(x)}, 
 \eeqn
 so that  $I\!I\!I_n $ admits a bound small enough for the present purpose. 
 
  Now (\ref{u/u}) becomes ready to be verified.  By $|\cos xt -\cos Rt| \leq (R-x) t$ it follows  that    
  \[\bigg|\int_0^{1/x} \frac{\alpha(t)f^\circ(t)}{t}(\cos xt -\cos Rt)dt\bigg| \leq C(R-x)\int_{x}^\infty \frac{c(y)dy}{m^2(y)y}
 \leq \frac{C\de x}{m(x)}, 
 \]
 where $\int_x^\infty c(y)dy/[m^2(y)y] \leq \int_x^\infty dy/[y^2m(y)]$ is used for the last inequality; in a similar way one deduces that
  the integral over $[1/x,M/x]$ is  dominated in absolute value by
\[
  C(R-x)\int_{x/M}^x \frac{c(y)dy}{m^2(y)y}\leq CM\de\int_{x/M}^x \frac{c(y)dy}{m^2(y)} \leq \frac{CM\de x}{m(x)}, 
\]
  so that 
\[
\bigg|\int_0^{M/x} \frac{\alpha(t)f^\circ(t)}{t}(\cos Rt- \cos xt) dt\bigg| \leq \frac{C'M\de x}{m(x)},
\]
which  together with  (\ref{2_prP1})  concludes  (\ref{u/u}).

 \begin{lemma}\label{Lem30}\,  Under (H),  $|I_n| \leq C[x/m(x)]/n$  whenever  $M>2n\pi$. 
 \end{lemma}
\begin{proof}   The integrand of the integral defining $I_n$ is less than 
$2\big[\alpha_{2n+1}(t)-\alpha_{2n}(t)\big]f^\circ(t)/t $ ($\geq 0$). Noting $\alpha_{2n+1}(t)-\alpha_{2n}(t)= \int_{2n}^{(2n+1)\pi x/M} \mu(y)\sin yt\, dy $ 
  and interchanging   the order of integrations we infer that
 $$I_n \leq  2\int_{2n}^{(2n+1)\pi x/M} \mu(y)dy\int_{2n\pi /y}^{[(2n+1)\pi/y]\wedge \pi} \frac{f^\circ(t)}{t} \sin yt \, dt.$$
Let (H) be satisfied. Then  by Lemma \ref{Lem17} again  
   in the range of the inner integral where $y/2n\pi \geq 1/t \geq y/(2n+1)\pi$ we have
  $$f^\circ(t)/t \leq C_2 (y/n)^3/m^2(y/n) \leq C_2 n^{-1} y^3/m^2(y),$$
whereas the integral of  $\sin yt$ over $0<t<\pi/y$ equals $2/y$.  Thus  
$$I_n \leq \frac{4C_2}{n} \int_{2n}^{(2n+1)\pi x/M} \mu(y)\frac{y^2}{m^2(y)} dy.$$
By $\mu(y)y^2\leq 2c(y)$ 
$$\int_0^z \frac{\mu(y)y^2}{m^2(y)} dy \leq 2 \int_0^z \frac{c(y)}{m^2(y)}dy = \frac{2z}{m(z)}.$$
Hence we obtain  the bound of the lemma in view of  the monotonicity of $x/m(x)$.  
\end{proof}
 
 \begin{lemma}\label{Lem31}\, Under (H) it holds that whenever  $\frac12 R<x \leq R$ and $1<M<x$,
 $$\bigg|\int_{M/x}^\pi \frac{\alpha_{2n}(t) f^\circ(t)}{t}\cos xt \,dt\bigg| \leq C\bigg[\frac1{\sqrt M}+ \frac{n^2}{M}\bigg]\frac{x}{m(x)}.$$
 \end{lemma}
 \begin{proof}  \,  Put $g(t)=f^\circ(t)/t$.
 Since $\alpha_n'(t) = \int_0^{n\pi/t} y\mu(y)\cos ty\,dy$, we have $|\alpha_n'(t)| \leq c(n\pi/t)
\leq \pi^2n^2c(1/t)$,  by which together with Lemma \ref{Lem22} we infer that if $ s <t <2s$, 
\beq
|\alpha_{2n}(t)g(t)- \alpha_{2n}(s)g(s)| &\leq& \alpha_{2n}(t)|g(t)- g(s)| +|\alpha_{2n}(t)- \alpha_{2n}(s) |g(s)\\
&\leq& C_1 \alpha_{2n}(t) g(t) \sqrt{|t-s|/s}  +C_2  n^2c(1/t) g(t)|t-s| \\
&\leq& C\big [\sqrt{|t-s|/s} + n^2 |t-s|/s\big] c(1/t)g(t)t.
\eeq
Put $s_k = (M+ 2\pi k)/x $ for $k=0, 1, 2,\ldots.$ Then  
$|\alpha_{2n}(t)g(t)- \alpha_{2n}(s_k)g(s_k)| \leq C\big[1/\sqrt M + n^2/M\big] c(1/t)g(t)t$ for $s_{k-1}\leq  t \leq s_k$ and if $N=\lfloor x/2 -M/2\pi \rfloor$, then by  $\int_{s_{k-1}}^{s_k}\cos xt dt =0$
\beq
\bigg|\int_{M/x}^{s_N} \frac{\alpha_{2n}(t)f^\circ(t)}{t}\cos xt \,dt\bigg| &=& \bigg|\sum_{k=1}^N  \int_{s_{k-1}}^{s_k} \big[ \alpha_{2n}(t)g(t) -\alpha_{2n}(s_k) g(s_k)\big] \cos xt \,dt\bigg|\\
&\leq& C\big[1/\sqrt M + n^2/M\big]  \sum_{k=1}^N  \int_{s_{k-1}}^{s_k}  c(1/t)g(t)t dt\\
&\leq& C'\big[1/\sqrt M + n^2/M\big][x/m(x)],
\eeq
where the last inequality is due to the assumption (H).  Since  $0\leq \pi - s_N =O(1/x) $, this gives  the bound of the lemma. 
\end{proof}

{\it Proof of  Proposition \ref{Prop5}.} \,Noting   $M^{-1/2} +n^2 M^{-1} +n^{-1}\leq 3M^{-1/3}$ for $n=M^{1/3}$  by Lemmas \ref{Lem30} and  \ref{Lem31}, 
we obtain    $I_n+I\!I_n \leq  C'M^{-1/3}[x/m(x)]$.   Hence  taking   $M =
 \de^{-3/4}$ and recalling (\ref{2_prP1}) lead to $|v_M(x)-v_M(R)| \leq C''\de^{1/4}[x/m(x)]$, which together with  (\ref{u/u})  yields  the required bound (\ref{x/R}). 
 \qed


\section{Proof of   Theorem \ref{Thm6}.}

 If (H)  holds, then by (\ref{eqL2.6}) of Lemma \ref{Lem17}  $\Re (1-e^{ixt})/[1-\psi(t)]$ is integrable over $|t|<\pi$,   which ensures 
\beqn\label{a_rep}
a(x)= \frac1{2\pi}\int_{-\pi}^\pi \Re\frac{1- e^{ixt}}{1-\psi(t)}dt
\eeqn
since  $2\pi\sum_{n=0}^\infty s^n [p^n(0) -p^n(-x)]= \int_{-\pi}^\pi \Re\big[(1-e^{ixy})/(1-s\psi(t))]dt \to a(x)$ ($s\uparrow 1$) by virtue of Abel's lemma, and  it  follows that
\beqn\label{a/f}
a(x) = \frac1{\pi }\int_0^\pi \frac{\al(t)(1-\cos xt) - \gamma(t)\sin xt }{[\al^2(t) +\gamma^2(t)] t} dt. 
\eeqn
  Recalling  $\gamma(t) = \beta_+(t) -\beta_-(t)$, we put
\[
b_\pm(x) =\frac1{\pi}\int_0^\pi \frac{\beta_\pm(t)\sin xt}{[\al^2(t) +\gamma^2(t)]t}dt
\]
so that
\beqn\label{a/b}
 a(x) = \bar a(x) + b_-(x)-b_+(x).
 \eeqn

Choosing a positive integer $N$ such that $E[ \, |X|; |X| > N]\leq P[\, |X| \leq N]$,
define  a function $p_*(x)$ on $\mathbb{Z}$ by  $p_*(1)= E[\,|X|; |X|>N]$, $p_*(0) = P[\,|X|\leq N]- p_*(1)$ and  
\[
p_*(k) =\left\{ \begin{array}{ll} p(k)+p(-k) \quad  &\mbox{if \, $k < -N$},\\
0 &\mbox{if \,  $-N\leq k\leq -1$ or    $k\geq 2$},
\end{array}\right.
\]
where $p(k) =P[X=k]$.  Then  $p_*$ is a probability distribution on $\mathbb{Z}$  with zero mean.
Denote the corresponding functions by $a_*, b_{*\pm}, \al_*$, $\al_{*\pm}$,  etc. 
Since  $p_*(z) =0$ for  $z\geq2$ and $\sigma_*^2=\infty$, we have   for $x>0$,  $a_*(-x)= 0$ so that $a_*(x)=a_*(x)/2 + b_{*-}(x) - b_{*+}(x)$, hence 
\[
\bar a_*(x)=2^{-1} a_*(x) =  b_{*-}(x)-b_{*+}(x).
\]
We shall show that if $m_+(x)/m(x)\to 0$, then  
\beqn\label{*1}
\bar a(x) \sim \bar a_*(x),\quad  |b_+(x)|+|b_{*+}(x)| = o(\bar a(x))
\eeqn
and 
\beqn\label{*2}
b_{*-}(x) = b_-(x) + o(\bar a(x)).
\eeqn
These together  with (\ref{a/b})  yield 
\beqn\label{b/a}
b_-(x)= \bar a_*(x)\{1+o(1)\} = \bar a(x)\{1+o(1)\},
\eeqn
and hence
$a(x) = \bar a(x)\{1+o(1)\} +b_-(x) \sim 2\bar a(x)$, which shows $a(-x)/ a(x)\to 0$.  

The rest of
this section is devoted to the proof of (\ref{*1}) and (\ref{*2}). 
It is easy to see   
\[
\al_*(t) =\al(t)+O(t)
\]
\[
\beta_{*-}(t) =\beta_-(t) +\beta_+(t) +O(t^2)
\]
\[
\beta_{*+}(t)= p_*(1)\int_0^1(1-\cos tx)dx =O(t^2)
\]
(as $t\downarrow 0$). Let   $D(t)$, $t>0$ denote the difference 
\beq
D(t) :=f^\circ(t)- f^\circ_*(t) &=& \frac{1}{\al^2(t)+ \gamma^2(t)} - \frac{1}{\al_*^2(t)+ \gamma_*^2(t)}\\
&=& \{(\al^2_{*} -\al^2)(t)+ (\gamma_{*}^2 -\gamma^2)(t)\}f^\circ(t)f^\circ_*(t).
\eeq
Observe  that $(\gamma_*+\gamma)(t) =-2\beta_-(t)+ o(t^2)$, $(\gamma_*-\gamma)(t) =-2\beta_+(t)+ o(t^2)$ and that 
\beqn\label{*D}
(\al^2_{*}- \al^2)(t) = 2\al(t)\times O(t)\quad\mbox{and}\quad   (\gamma^2_{*} - \gamma^2)(t) = 4\beta_-(t)\beta_+(t)+  o(t^2).
\eeqn

Now we suppose $m_+(x)/m(x)\to 0$. Then  by (\ref{m1})   it follows that $f(t)\asymp f^\circ(t)\asymp f_*^\circ(t)$ and $\beta_-(t) \beta_+(t)/[\al^2(t)+\beta^2(t)] \to 0$, and hence that $D(t) = o(f(t))$, which implies that
$\bar a(x)\sim \bar a_*(x)$, the first relation of (\ref{*1}). The proofs of the second relation in (\ref{*1}) and of 
(\ref{*2}) are somewhat involved since we need to take advantage of the oscillating  nature of the integrals defining $\beta_{\pm}(t)$.   First we dispose of the non-oscillatory  parts of these integrals. 

By (\ref{m1}b) applied to  $\al_\pm+ \beta_\pm$ in place of $\al +\beta$ it follows  that if $m_+(x)/m(x)\to 0$, then $\lim_{t\downarrow 0} \frac{\al_+(t) + \beta_+(t)}{\al_-(t)+\beta_-(t)} =0$ (the converse is also true), which  entails 
\beqn\label{m2}\al(t) -\ga(t) \sim \al(t)+\beta(t) \asymp m(1/t)t,\eeqn

\begin{lemma}\label{Lem32} \, If $m_+(x)/m(x) \to 0$, then 
\[
\int_0^{1/x} f(t)\beta_+(t)dt = o(1/m(x)).
\]
\end{lemma}
\begin{proof}  \, By  (\ref{m1}) and  a change of variable  the assertion of the lemma is the same as
\[
\int_x^\infty \frac{\tilde m_+(y)}{ym^2(y)}dy = o(1/m(x)).
\]
Putting $g(x)= \int_x^\infty dy/y^2m(y)$ and integrating by parts we have
\[
\int_x^\infty \frac{\tilde m_+(y)}{ym^2(y)}dy = -\bigg[g(y)\frac{y\tilde m_+}{m}\bigg]_{y=x}^\infty +\int_x^\infty g(y)\bigg(\frac{y\tilde m_+}{m}\bigg)'dy
\]
as well as
\[
g(y) = \frac1{ym(y)} -\int_y^\infty \frac{\eta(u)}{um^2(u)}du. 
\]
On observing
\[
\bigg(\frac{y\tilde m_+}{m}\bigg)' =\frac{2y\eta_+}{m} -\frac{y\tilde m_+\eta }{m^2}
\leq \frac{2y\eta_+}{m}
\]
substitution leads to
\[
\int_x^\infty \frac{\tilde m(y)}{ym^2(y)}dy  \leq \frac{\tilde m_+(x)}{m^2(x)} +\int_x^\infty \frac{2\eta_+(y)}{m^2(y)}dy.
\]
The first term on the RHS is $o(1/m(x))$ owing to the assumption of the lemma since $\tilde m_+ \leq m_+$.
On the other hand on integrating by parts again
\beq
 \int_x^\infty \frac{ \eta_+(y)}{m^2(y)}dy = 
 -\frac{m_+(x)}{m^2(x)} + \int_x^\infty \frac{m_+(y)\eta(y)}{m^3(y)}dy&\leq& \sup_{y\geq x}\frac{m_+(y)}{m(y)}\int_x^\infty \frac{\eta(y)}{m^2(y)}dy\\
 & =& o(1/m(x)).
 \eeq
The proof is complete.
 \end{proof} 
 
\begin{lemma}\label{Lem33} \,  If $m_+(x)/m(x)\to0$, 
\[
\int_1^x\frac{c_+(y)}{m^2(y)}dx =o\bigg(\frac{x}{m(x)}\bigg).
\]
\end{lemma}
\begin{proof}   \, The assertion of the lemma follows from the following identity for primitive functions 
\beqn\label{EQ}
\int\frac{c_+(x)}{m^2(x)}dx =2\int \frac{c(x)}{m^2(x)}\cdot \frac{m_+(x)}{m(x)}dx -\frac{xm_+(x)}{m^2(x)}.
\eeqn
This identity may be verified  by differentiation as well as   derived by integration by parts, the latter  giving
\[
\int\frac{c_+(x)}{m^2(x)}dx = \frac{x}{m_+(x)}\cdot\frac{m_+^2(x)}{m^2(x)} - 2\int \frac{x}{m_+(x)}\cdot\frac{(\eta_+c -c_+ \eta)m_+(x)}{m^3}dx,
\]
from which we deduce  (\ref{EQ}) by an easy algebraic manipulation. 
\end{proof} 
\vskip2mm


\vskip2mm
From Lemma  \ref{Lem22}   it follows that   for $x\geq 4$,
\beqn\label{eqL3.2}
|\al(t)-\al(s)| \vee |\beta(t)-\beta(s)|  \leq  \kappa c(1/t)t\quad\mbox{if \;\; $t\geq \pi/x$ and $s=t+\pi/x$}
\eeqn 
(with $\kappa< 2\pi^2$). We shall apply Lemma \ref{Lem22} only  in this form in the sequel.
 \vskip2mm

\textit{ Proof of (\ref{*1}).} We prove $b_+(x) = o(\bar a(x))$ only, $b_{*+}$ being dealt with in the same way. In view of  Theorem \ref{th:1} and Lemma \ref{Lem32} it suffices to show that
\beqn\label{pf_3.1}\int_{\pi/x}^\pi \frac{f^\circ(t)\beta_+(t)}{t}\sin xt \, dt =o\bigg(\frac{x}{m(x)}\bigg) \quad\mbox{where}\quad  f^\circ(t):=\frac{1}{\al^2(t)+\gamma^2(t)}. 
\eeqn
We make the decomposition 
\beq
2\int_{\pi/x}^\pi \frac{(f^\circ\beta_+)(t)}{t}\sin xt \, dt  &=& \int_{\pi/x}^\pi\frac{ (f^\circ\beta_+)(t)}{t}\sin xt \, dt  - \int_0^{\pi-\frac\pi x} \frac{(f^\circ\beta_+)(t+ \frac\pi x)}{t+\pi/x}\sin xt \, dt\\
&=& I(x) + I\!I(x) + I\!I\!I(x) + r(x)
\eeq
where
\[
I(x) =\int_{\pi/x}^\pi \frac{f^\circ(t) -  f^\circ(t+ \frac\pi x)}{t}\beta_+(t)\sin xt \, dt,
\]
\[
I\!I(x) =\int_{\pi/x}^\pi  f^\circ(t+ \pi/x)
\frac{\beta_+(t)- \beta_+(t+ \frac\pi x)}{t}\sin xt \, dt,
\]
\[
I\!I\!I(x)= \int_{\pi/x}^\pi  f^\circ(t+ \pi/x)
\beta_+(t+\pi/x) \Big(\frac1{t}-\frac{1}{t+\pi/x}\bigg)\sin xt \, dt
\]
and 
\[
r(x) = - \int_{0}^{\pi/x}\frac{ (f^\circ\beta_+)(t+\frac\pi x)}{t+ \pi /x}\sin xt \, dt  +  \int_{\pi-\pi/x}^\pi \frac{ (f^\circ\beta_+)(t+\pi /x)}{t+\pi/ x} \sin xt \, dt. 
\]

From  (\ref{eqL3.2}) (applied not only with $\mu$ but  with $\mu_\pm$ in place of $\mu$) we obtain
\[
\big|\beta_+(t+\pi/x)-\beta_+(t)\big| \leq \kappa c_+(1/t)t
\]
and
\[
\big |f^\circ(t+\pi/x)-f^\circ(t)\big| \leq C_1 c(1/t)  [f(t)]^{3/2}t
 \]
for $t>\pi/x$.
From the last inequality  together with $f^{3/2}(t)t = 1/t^2m^3(t)$, $\beta_+(t)\leq C_3\tilde m_+(1/t) t$ and $\tilde m_+(x)\leq m_+(x) =o(m(x))$ we infer that
\[
|I(x)| \leq C\int_{1/2}^{x/2}\frac{\tilde m_+(y)}{m(y)}\cdot\frac{c(y)}{m^2(y)}dy =o\bigg(\frac{x}{m(x)}\bigg).
\]
Similarly
\[
|I\!I(x)| \leq  C\int_{1/2}^{x/2}\frac{c_+(y)}{m^2(y)}dy  =o\bigg(\frac{x}{m(x)}\bigg)
\]
and
\[
|I\!I\!I(x)| \leq \frac{C}{x}\int_{1/\pi}^{x/\pi}\frac{\tilde m_+(y)y}{m^2(y)}dy =o\bigg(\frac{x}{m(x)}\bigg),
\]
where the equalities follow from Lemma \ref{Lem33} and the monotonicity of $y/m(y)$ in the bounds of  $|I\!I(x)|$ and $|I\!I\!I(x)|$, respectively. Finally
\[
|r(x)| \leq C\int_{x/2\pi}^{x/\pi} \frac{\tilde m_+(y)}{m^2(y)}dy + O(1/x)=o\bigg(\frac{x}{m(x)}\bigg).
\]
Thus we have verified  (\ref{pf_3.1})  and accordingly   (\ref{*1}).  \qed

\vskip2mm
\textit{ Proof of (\ref{*2}).} Recalling $D(t) = f^\circ(t)-f_*^\circ(t)$ we have
\beq
\pi \big[b_-(x) - b_{*-}(x)\big] &=& \int_0^\pi D(t)\frac{\beta_-(t)}{t}\sin xt\,dt +  \int_0^\pi f_*^\circ(t)\frac{\beta_-(t)- \beta_{*-}(t)}{t}\sin xt\,dt\\
&=& J +K \quad \mbox{(say)}.
\eeq
Suppose $m_+/m\to 0$. Since $\beta_{*-}(t)-\beta_-(t)= \beta_+(t)+O(t^2)$ and $f_*^\circ$ is essentially of the same 
regularity as $f^\circ$, the proof of (\ref{pf_3.1}) and Lemma \ref{Lem32} applies to $K$  on the RHS above to yield    $K=o(x/m(x))$. As for $J$ we first observe that  in view of (\ref{*D})
\beqn\label{*J}
|D(t)|\leq C[f(t)]^{3/2}(t+\beta_+(t))
\eeqn
 so that the integral defining $J$ restricted to $[0,\pi/x]$ is $o(x/m(x))$ in view of Lemma \ref{Lem32}. It remains to show that
\beqn\label{**J}
 \int_{\pi/x}^\pi D(t)\frac{\beta_-(t)}{t}\sin xt\,dt =o\bigg(\frac{x}{m(x)}\bigg).
 \eeqn

We decompose
$D(t) = D_1(t)+ D_2(t)$ where
\[
D_1(t) = \big[(\al^2_{*} -\al^2)(t)+ (\gamma_{*}^2 -\gamma^2)(t) -4(\beta_-\beta_+)(t)\big] f^\circ(t)f^\circ_*(t)
\]
and
\[
D_2(t) = 4(\beta_-\beta_+)(t)f^\circ(t)f^\circ_*(t).
\]
By (\ref{*D}) $|D_1(t)\beta_-(t)/t| \leq C_1\big(\al(t)+t\big) [f(t)]^{3/2} \leq C_2 c(1/t)/[t^2m^3(1/t)] $ and 
\[
\bigg| \int_{\pi/x}^\pi D_1(t)\frac{\beta_-(t)}{t}\sin xt\,dt\bigg| \leq C_2\int_{1/\pi}^\infty \frac{c(y)}{m^3(y)}dy = o(x/m(x))
\]
as is easily verified.
For the integral involving $D_2$ we proceed as in the proof of (\ref{pf_3.1}). To this end it suffices to evaluate the integrals corresponding to $I(x)$ and $I\!I(x)$,  namely 
\[
J_I(x) := \int_{\pi/x}^\pi \frac{D_2(t) -  D_2(t+ \pi/x)}{t}\beta_-(t)\sin xt \, dt,
\]
and
\[
J_{I\!I}(x) :=\int_{\pi/x}^\pi  D_2(t+ \pi/x)
\frac{\beta_-(t)- \beta_-(t+\pi/x)}{t}\sin xt \, dt,
\]
the other integrals being easily dealt with  as before.  By (\ref{eqL3.2})
the integrand for $J_{I\!I}(x)$ is dominated in absolute value by a constant multiple of 
$[f(t)]^{3/2}\beta_+(t)c(1/t) $ from which it follows immediately that  $J_{I\!I}(x) =o(x/m(x))$.
For the evaluation of $J_I(x)$, 
observe that
\[
\big|(\beta_-\beta_+)(t) - (\beta_-\beta_+)(t+\pi/x)\big| \leq C_1\big\{ \beta_+(t)c(1/t)t + \beta_-(t)c_+(1/t)t\big\}
\]
and
\[
\big|(f^\circ f)(t) - (f^\circ f)(t+\pi/x) \big|\beta_-(t) \leq C_1 c(1/t) t[f(t)]^{2}
\]
so that 
\beq
\frac{|D_2(t) -  D_2(t+ \pi/x)|\beta_-(t)}{t} &\leq&
 C\big\{  \beta_+(t)c(1/t) + \beta_-(t)c_+(1/t) \big\}[f(t)]^{3/2}\\
&\leq &C'\frac{\tilde m_+(1/t)c(1/t)}{t^2m^3{1/t)} }+ C'\frac{c_+(1/t)}{t^2 m^2(1/t)}.
\eeq
The integral of the first term of the last member  is immediately evaluated and that of the second  by Lemma \ref{Lem33}, showing
\[
\bigg| \int_{\pi/x}^\pi D_2(t)\frac{\beta_-(t)}{t}\sin xt\,dt\bigg| \leq C_2\int_{1/\pi}^\infty \bigg[\frac{\tilde m(y)c(y)}{m^3(y)}+ \frac{c_+(y)}{m^2(y)}\bigg] dy = o\bigg(\frac{x}{m(x)}\bigg).
\]
The proof of (\ref{*2}) is complete. \qed

\section{Relatively stable distributions}

In this section,  $F$ will be recurrent (as elsewhere in this paper)  but allowed of the case $E|X|=\infty$, and relatively stable, namely
\v2 
 \beqn\label{RS}  |A(x)|/x\mu(x) \to\infty.
\eeqn
Let  $M(x)$ be a function of $x\geq x_0$ ($x_0\geq 0$ is a constant such that  $\inf_{x\geq x_0}|A(x)|>0$) defined by
\beqn\label{F_M}
M_\pm(x) = \int_{x_0}^x\frac{\mu_\pm(y)}{A^2(y)}dy \quad\mbox{and}\quad M(x)=M_+(x)+M_-(x).
\eeqn
Theorem \ref{thm7} readily follows from the following proposition as is discussed right after it.
\begin{proposition}\label{Prop34} Suppose that $F$ is recurrent and (\ref{PRS}) holds. Then as $x\to\infty$
\v2
{\rm (i)}  \;\,\; $ a(x)-a(-x) \sim 1/A(x);$ 
\v2
{\rm (ii)} \;\;  $2\bar a(x) \sim M(x);$ 
\v2
{\rm (iii)} \, $ a(-x)/a(x)\to 1$ if and only if $M_-(x)/M_+(x) \to 1;$ and 
\v2
{\rm (iv)} \, if $EX=0$ and  $m_-/m_+ \to 1$, then  $\bar a(x) m(x)/x \to \infty$.
\end{proposition}

 Suppose  (\ref{PRS}) to hold, i.e., (\ref{RS}) to hold with $A(x) >0$ for all sufficiently large $x$. Put $K(x)= \mu_+(x)-\mu_-(x)$.   Then 
$\log [A(x)/A(x_0)] =\int_{x_0}^x \ep(t)dt/t$ with $\ep(t)= tK(t)/A(t)\to 0$,   hence  $A(x)$ is s.v. 
By   $(1/A)'(x) =- K(x)/A^2(x)$, we have
\beqn\label{1/A}
\frac1{A(x)} -\frac1{A(x_0)} =  \int_{x_0}^x \frac{-K(y)}{A^2(y)}dy = M_-(x)- M_+(x),
\eeqn
in particular $M_-(x)\geq M_+(x) \vee[1/A(x)]- 1/A(x_0)$.
Because of the recurrence of the r.w.,  we have $M_-(x) \to\infty$ (see (P1) below), and  from (i) and (ii) one infers that
$$a(x) \sim M_-(x) \quad \mbox{and}\quad a(-x) = M_+(x) + o(M_-(x));$$
 it also follows that  $x\mu(x)/A^2(x) = o(1)\times 1/A(x)= o(M_-(x))$, so that both $M_-$ and $M$ are s.v.  Thus Theorem \ref{thm7}  follows from Proposition \ref{Prop34}. 


 The following results  are obtained under (\ref{PRS})  in \cite{Urenw}:

{\it If  (\ref{RS}) hold, then}
\v2

(P1)\quad {\it  $F$ is recurrent if and  only if $\int_{x_0}^\infty \mu(x)dx/A^2(x) =\infty$.}
\v2

(P2) \quad $\alpha(\theta) =o(\gamma(\theta))  \quad\mbox{and}\quad  \gamma(\theta) 
\sim - A(1/\theta)\quad (\theta\downarrow 0)$.
\v2
(P3) \quad there exists a  constant $C$ such that  for any $0<\ep< 1$ and for all $x$ large enough, 
\beqn\label{RS2}
\bigg|\int_{1/\ep x}^\pi \frac{f^\circ(t)\alpha(t)\cos\, xt} {t}dt\bigg| +  \bigg| \int_{1/\ep x}^\pi \frac{f^\circ(t)\gamma(t)\sin\,xt }{t}dt\bigg| \leq \frac{C\ep}{A(x)},
\eeqn
 where $f^\circ(t)=1/[\alpha^2(t) +\ga^2(t)]$ as before (see Lemmas 6, 4 and 8  of \cite {Urenw} for (P1), (P2) and (P3), respectively). When $F$ is transient,   an asymptotic form of the Green function of the r.w.
is computed in \cite{Urenw}.

\v2
{\it Proof of {\rm (i)}.}\; By (\ref{a/f})
$$a(x)-a(-x)= \frac2{\pi} \int_0^\pi \frac{-\gamma(t)f^\circ(t)\sin xt}{t}dt.$$
By (P2) 
\beqn\label{ga/A}
f^\circ(t) \sim 1/A^2(1/t) \quad \mbox{and}\quad -\ga(t)f^\circ(t)\sim 1/A(1/t)
\eeqn
($t\downarrow 0$).    Let  $0<\ep<1$. For the integral over $t\geq 1/\ep x$ we have the bound $C\ep/A(x)$ because of (P3), while
the integral restricted to  $t <\ep/x$ is dominated in absolute value a constant multiple of 
$x\int_0^{\ep/x}dt/A(1/t) \sim \ep/A(x)$. These reduce our task to showing that for each $\ep$
\beqn\label{gA}
\lim_{x\to\infty} \int_{\ep/x}^{1/\ep x}\frac{-\gamma(t)f^\circ(t)\sin xt}{t}dt\cdot A(x) = \int_\ep^{1/\ep} \frac{\sin t}{t}dt,
\eeqn
for the integral on the RHS converges to $\pi/2$ as $\ep\downarrow 0$.
On writing  $g(t)$ for $-\gamma(t)f^\circ(t)$  
the change of the variable  $t =w/x$ transforms the  function  under the above limit to
$$ \int_{\ep}^{1/\ep}[g(w/x)A(x)] \frac{\sin w}{w}dw.$$ 
By the slow variation of  $A$ together with (\ref{ga/A}) it follows that  $g(w/x)A(x) \to 1$ as $x\to\infty$ uniformly for $\ep\leq w\leq 1/\ep$. Thus we have (\ref{gA}). \qed

\v2

{\it Proof of {\rm (ii)}.}\; 
Given a constant $0<\ep<1$, we decompose the integral in the representation of $\bar a$ given in  (\ref{**a}) as follows:
$$\pi\bar a(x) = I(x) - II(x) + I\!I\!I(x),$$
where
\beq
&&I (x)= \int_{1/\ep x}^\pi \frac{f^\circ(t)\alpha(t)}{t}dt, \qquad II (x)=  \int_{1/\ep x}^\pi \frac{f^\circ(t)\alpha(t)\cos xt}{t}dt,\\
&&\qquad\mbox{and}  \qquad I\!I\!I (x) = \int_0^{1/\ep x} \frac{f^\circ(t)\alpha(t)(1-\cos xt)}{t}dt.
\eeq
By (\ref{PRS}) it follows that
\beqn\label{c/A}
c(x) =\int_0^x y\mu(y)dy =o(xA(x)).
\eeqn
Using $f^\circ(t)\sim 1/A^2(1/t)$ (see (\ref{ga/A}))  as well as  $\alpha(t) \leq C_1tc(1/t)$, we then deduce that
$$I\!I\!I (x) \leq  C_1x^2 \int_0^{1/\ep x} \frac{c(1/t)t^2}{A^2(1/t)}dt =x^2 \int_{\ep x}^\infty \frac{dy}{y^3A(y)}dy \times o(1) = o\bigg(\frac1{A(x)}\bigg).$$
 By (P3) we also have 
 $$|II(x)| \leq C\ep/A(x).$$
  Hence for the proof of (2) it suffices to verify 
 $I(x)\sim \frac12 \pi M(x)$. Since  $\alpha(t)$ is positive,  this follows if we can show
 \beqn\label{claim}
 \tilde I(x):=\int_{1/\ep x}^{b} \frac{\alpha(t)}{A^2(1/t)t}dt \sim \frac{\pi}{2}M(x)
 \eeqn
 for some/any  positive constant $b\leq (1/x_0)\wedge \pi$.
We decompose the integral defining $\alpha(t)$ by splitting its range to have 
 \beq
 \alpha(t) &=& \bigg(\int_0^{\ep/t} +\int_{\ep/t}^{1/\ep t} + \int_{1/\ep t}^\pi\bigg) \mu(y)\sin ty\,dy\\
&=&   \alpha_1(t) +\alpha_{2}(t)+ \alpha_{3}(t) \qquad\mbox{(say)},
\eeq
and accordingly 
$$\tilde I(x) = \int_{1/\ep x}^b \big[\alpha_1(t) +\alpha_2(t)+\alpha_3(t)\big] \frac{dt}{A^2(1/t)t}.$$

Obviously  $\alpha_1(t) = \int_0^{\ep/t}\mu(y)\sin ty \,dt \leq t c(\ep/t)$, and integrating by parts leads to 
\begin{eqnarray}
\int_{1/\ep x}^b \alpha_1(t)  \frac{dt}{A^2(1/t)t} &\leq& \int_{1/b}^{\ep x} \frac{c(\ep y)}{A^2(y)y^2}dy
\nonumber\\
&=& -\frac{c(\ep^2 x)}{\ep xA^2(\ep x)}\{1+o(1)\} +\int_{1/b}^{\ep x} \frac{\ep^2 \mu(\ep y)dy}{A^2(y)}\{1+o(1)\} \nonumber\\
& \leq&  \ep M(x)\{1+o(1)\}.
\label{M1}
\end{eqnarray}

Noting  $\alpha_{2}(t) = \int_{\ep/t}^{1/\ep t}\mu(y)\sin ty\, dy =t^{-1}\int_{\ep}^{1/\ep}
\mu(w/t)\sin w\,dw$, we infer 
$$
\int_{1/\ep x}^b \alpha_2(t)  \frac{dt}{A^2(1/t)t} 
= \int_{\ep}^{1/\ep} \sin w\,dw \int_{1/\ep x}^b  \frac{\mu(w/t)dt}{A^2(1/t)t^2} = \int_{\ep}^{1/\ep} \sin w\,dw \int^{\ep x}_{1/b}  \frac{\mu(wy)dy}{A^2(y)}. 
$$
Changing the variable  $y=z/w$ transforms  the last repeated integral into
$$\int_{\ep}^{1/\ep} \frac{\sin w\,dw}{w} \int^{\ep wx}_{w/b}  \frac{\mu(z)dz}{A^2(wz)}.$$ 
Because of the slow variation of $A$, the inner integral is asymptotically equivalent to $M(\ep wx)$ as $x\to\infty$ uniformly for $\ep \leq w\leq 1/\ep$.
Since $M(x)$ is s.v., we can conclude that for each  $\ep$, as $x\to\infty$
\beqn\label{M2}
\int_{1/\ep x}^b \alpha_2(t)  \frac{dt}{A^2(1/t)t} \sim \bigg[\int_{\ep}^{1/\ep} \frac{\sin w}{w}dw\bigg]  M(x). 
\eeqn

For $t$ small enough, $1/A^2(1/t)t$ is decreasing and we choose  $\de>0$ so that this is  true for $t\leq \de$. Then by interchanging the order of integration 
\beq
\int_{1/\ep x}^b \alpha_3(t)  \frac{dt}{A^2(1/t)t} &=& \int_{1/\ep x}^\de \frac{dt}{A^2(1/t)t} \int_{1/\ep t}^\infty \mu(y)\sin ty\, dy + C_{\de,b}\\
&=&
\int_{1/\ep b}^{x} \mu(y) dy \int_{1/\ep y}^\de \frac{\sin ty\,  dt}{A^2(1/t)t}  + \int_{x}^{\infty} \mu(y) dy \int_{1/\ep x}^\de \frac{\sin ty\,  dt}{A^2(1/t)t} +  C_{\de,b}. 
\eeq
and  the inner integrals of the first and second repeated integrals are bound by a constant multiple of $\ep/A^2(\ep y)$ and $\ep x/[yA^2(\ep x)]$, respectively.  Observe that 
 $$\int_x^\infty \frac{\mu(y)}{y}dy \leq \int_x^\infty \frac{A(y)}{y^2}dy\times o(1) = o\big(A(x)/x\big).$$
Then, it is easy to see that
$$\int_{1/\ep x}^b \alpha_3(t)  \frac{dt}{A^2(1/t)t} \leq C \ep M(x) +  C_{\de,b}.$$
Combined with (\ref{M1}) and (\ref{M2}) this shows (\ref{claim}). The proof of (2)  is complete. \qed

\v2

{\it Proof of {\rm (iii)} and {\rm (iv)}.}\; (iii) follows from (i),  (ii) and (\ref{1/A}) since $M_-(x)\to \infty$.

Let $EX=0$.  If $m_+/m_-\to1$, then by Theorem \ref{th:1_3}(i) $\ga(t)=o(tm(1/t))$, which together with  (P2) shows (iv) owing to Theorem \ref{th:1_2}(ii).



\section{Applications }     

Let $S_n^x = x+ X_1+\cdots +X_n$ be a r.w. on $\mathbb{Z}$ started at $x$, where $X_1, X_2, X_3,\ldots $ are independent and have the same distribution as $X$. For $B\subset \mathbb{Z}$ let  $\sigma^x_B$ denote the first hitting time of $B$ by $S^x$ after time 0 so that we have always 
$\sigma^x_B\geq 1$  with probability one (w.p.1).  We write $S^x_{\sigma_B}$ or sometimes  $S^x_{\sigma B}$ for $S^x_{\sigma^x_B}$ and $\sigma_y^x$ for  $\sigma_{\{y\}}^x$ to simplify the notation. Let $g_B(x,y) =  \sum_{n=0}^\infty P[ S^x_n =y, \sigma^x_B>n]$, the Green function of the r.w. killed as it hits $B$. Our definition of $g_B$ is not standard:  $g_B(x,y) = \delta_{x,y} + E[g_B(S_1^x,y); S_1^x\notin B]$ not only for $x\notin B$ but for $x\in B$, while $g_B(x,y) =\delta_{x,y}$ for all $x\in \mathbb{Z}, y\in B$. (Here $\delta_{x,y}$ designates  Kronecker's delta kernel.)  The function $g_B(\cdot, y)$ restricted on $B$ equals the hitting distribution of $B$  by  $\hat S^y= y-X_1 -X_2- \cdots -X_n$, the dual r.w. started at  $y$, in particular
\[
g_{\{0\}}(0,y) =1\quad (y\in \mathbb{Z}).
\]
(In \cite{S} the Green function of the r.w. killed on hitting $0$ is defined by 
$g(x,y) = g_{\{0\}}(x,y) - \delta_{0,x}$ 
so that $g(0,\cdot)=g(\cdot,0)=0$; the above identity is the same as $\sum_x p(x)g(x,y) =1- \de_{0,y}$.) The potential function  $a$ bears   relevance to $g_{\{0\}}$  through the identity 
\[
 g_{\{0\}}(x,y) = a^\dagger(x)+ a(-y) -a(x-y)
 \]
(cf. \cite[p.328]{S}), which entails
\beqn\label{eq:L6.1} 
P[\sigma^x_y<\sigma^x_0] =\frac{a^\dagger(x)+a(-y)-a(x-y)}{2\bar a^\dagger(y)}\quad\;\;  (y\neq x).
\eeqn
Here we put 
$$a^\dagger(x) =\delta_{x,0} +a(x)\quad \mbox{and}\quad \bar a^\dagger(x) = (a^\dagger(x) +a(-x))/2.$$
 When $x=y$,  it follows that 
$P[\sigma^{x}_{0}\leq \sigma^{x}_x] = P[\sigma^0_{-x}\leq \sigma^0_0] = 1/2\bar a^\dagger(x)$, so that
$$P[\sigma^x_x<\sigma^x_0] = 1- 1/2\bar a^\dagger(x).$$
 If the r.w. is left-continuous (i.e.,  $P[ X\leq -2]=0$), then $a(x)=x/\sigma^2$ for $x>0$; analogously  $a(x)=-x/\sigma^2$ for $x<0$ for right-continuous r.w.s;  $a(0)=0$, there exists $\lim_{x\to\infty}a(x) \leq \infty$  and $\inf_{x>0} a(x)>0$  except for the  left-continuous r.w.s \cite{Uladd}. 
 
Recall that the  condition $\lim_{x\to\infty}m_+(x)/m(x)=0$ is simply written as
$ m_+/m\to 0$;
we also write  $\eta_+/\eta_- \to \infty$, $m_+\asymp m$ etc. to indicate  corresponding conditions.
We shall use the letter $R$ to  denote a (large) positive integer without exception.  We suppose 
\beqn\label{N_RC}
P[X\geq 2] >0
\eeqn
so as to ensure  $a(-x)>0$ for $x\geq 1$,  the  right-continuous r.w.s being not interesting for the discussion in this section.  
  Recall also that  $Z$ (resp.  $\hat Z (<0)$) stands for  the strictly ascending (resp. descending)  ladder height of the r.w. $S^0_{\cdot}$.
  

The rest of this section is divided into the six subsections: 
\vskip1mm

7.1. Some  asymptotic estimates of $P[\sigma^x_R<\sigma^x_0]$;\,  

7.2. Relative stability and overshoots; 

  7.3. Asymptotic form of  $P[\sigma^x_{[R,\infty)} <\sigma^x_{(-\infty,0]}] $; 
  
  7.4. Proof of Proposition \ref{Prop5}; 

7.5. Comparison between  $\sigma^x_R$ and $\sigma^x_{[R,\infty)}$; 

7.6. Escape into $(-\infty,-Q]\cup[R,\infty)$. 
\vskip1mm\n

\noindent
The last two subsections  can be  read independently of the subsections 7.3 and 7.4.

\vskip3mm
\textbf{ 7.1.}  {\sc Some  asymptotic estimates of $P[\sigma^x_R<\sigma^x_0]$.}
\vskip2mm

 The potential function  satisfies the functional equation
\beqn \label{eq2.6}
\sum_{y=-\infty}^\infty p(y-x)a(y) =a^\dagger(x),
\eeqn  
(cf. \cite[p.352]{S}), which restricted on $x\neq 0$ states that $a$ is harmonic there so that  the process $M_n:=a(S^x_{\sigma_0\wedge n})$ is a martingale for each $x\neq0$ and by the  optional sampling theorem  
\beqn\label{eq2.-1}
a(x)\geq E[a(S^x_{\sigma_0\wedge \sigma_y})]=a(y)P[\sigma^x_y<\sigma^x_0] \qquad  (x\neq 0).
\eeqn

\begin{lemma}\label{Lem35}  For all  $x, y\in \mathbb{Z}$,
\beqn\label{L6.1}
-\, \frac{a(y)}{a(-y)} a(x)\leq a(x+y)-a(y) \leq a(x) \quad \mbox{if} \;\; a(-y) \neq 0.
\eeqn
\end{lemma}
\v2

  This is Lemma 3.2 of \cite{Uladd}. The right-hand inequality of (\ref{L6.1}), stating the  sub-additivity of $a$ that is given in \cite{S},  is the same as $g_{\{0\}}(x,-y)\geq 0$. The left-hand one, which  seems much less familiar, will   play a significant role in the sequel. Here we repeat the proof of it.   Comparing (\ref{eq:L6.1}) and (\ref{eq2.-1}) (with variables suitably chosen) we have
\beqn\label{prL6.1}
\frac{a(x) + a(y) - a(x+y)}{a(y)+a(-y)} \leq \frac{a(x)}{a(-y)} \quad (a(-y)\neq 0),
\eeqn
which, after rearrangement,  becomes the left-hand inequality of (\ref{L6.1}). 
 
 \begin{remark} \label{Rem36} 
 (i)  The left-hand inequality of (\ref{L6.1}) may  yield useful  upper  as well as lower bounds of the middle term.  Here we write down such ones in a form used later: if $a(R)a(R-x)\neq 0$,
\beqn\label{eL6.2}
  -\frac{a(x-R)}{a(R-x)}a(-x)\leq  g_{\{0\}}(x,R) -  a^\dagger(x) =a(-R)- a(x-R)  \leq \frac{a(-R)a(x)}{a(R)}.
\eeqn  
Both inequalities are deduced from the left-hand inequality of (\ref{L6.1}):
  the lower bound  follows by replacing $y$ and $x$ with $x-R$ and $-x$, respectively and   the upper bound by replacing $y$ with $-R$. Since $P[\sigma_R^x<\sigma_0^x]= g_{\{0\}}(x,R)/2\bar a(R)$, 
\beqn\label{eqL6.20}
  -\frac{a(x-R)}{a(R-x)}\cdot\frac{a(-x)}{2\bar a(R)}\leq  P[\sigma_R^x<\sigma_0^x]- \frac{a^\dagger(x)}{2\bar a(R)}  \leq  \frac{a(-R)a(x)}{2a(R)\bar a(R)}.
\eeqn
(\ref{eL6.2}) is quite efficient in case $a(-R)/a(R)\to0$  and will be used later.
 

(ii)  By (\ref{eq:L6.1}) and the sub-additivity of $a$
\[
P[\sigma^x_0<\sigma^x_y] = \frac{a(x-y) +a(y)- a^\dagger(x)}{2\bar a(y)}  \leq \frac{\bar a(y-x)}{\bar a(y)} \quad \;\; (y\neq 0, x).
\]
 \end{remark}
\vskip2mm

 
  \begin{lemma}\label{Lem37}\,\,Suppose ${\displaystyle \lim_{z\to\infty}}  a(-z)/a(z) = 0$. Then
  
    {\rm (i)}   uniformly for $x\leq -R$,  as $R\to\infty$
  \[
  \frac{a(x)-a(x+R)}{a(R)} \, \longrightarrow \, 0\quad\mbox{and}\quad 
 P[\sigma^x_{-R}<\sigma^x_0] \,\longrightarrow \,1.
 \]

{\rm (ii)}  uniformly for $-M <  x<R$ with any fixed $M>0$, as $R\to\infty$
\[
\frac{a(-R)-a(x-R)}{a^\dagger(x)}\, \longrightarrow \, 0 \quad\mbox{and} \quad P[\sigma^x_R<\sigma^x_0] \,= \,\frac{a^\dagger(x)}{a(R)}\{1+o(1)\}.
\]
\end{lemma}

\begin{proof}  
Suppose $\lim_{z\to\infty} a(-z)/a(z) = 0$. This excludes the possibility of the left-continuity of the r.w. so that $a^\dagger(x)>0$ for  all $x$ (because of (\ref{N_RC}))  and 
 (ii) follows immediately from (\ref{eL6.2}) and  (\ref{eqL6.20}).  (i) is deduced from (\ref{L6.1})  as above (substitute $x+R$ and $-R$ for $x$ and $y$ respectively for the lower bound; use sub-additivity of $a$  for the upper bound).  
  \end{proof} 
 
By virtue of  Proposition \ref{Prop5} and Theorem \ref{Thm6}, 
 Lemma \ref{Lem37}(ii) entails the following. 
 
\begin{corollary}\label{Cor38}\,  Suppose $m_+/m \to 0$. 
 Then $P[\sigma_R^{x}<\sigma_0^{x}]  \to 1$ as $x/R\uparrow1$. 
\end{corollary}

  By Remark \ref{Rem36}(ii)  we know that   $P[\sigma_R^{R-y}<\sigma_0^{R-y}] \geq 1-\bar a(y)/\bar a(R)$, which estimate  is better   than the one given above in most  cases, but does not generally  imply  the consequence in Corollary \ref{Cor38} since $ \bar a(y)/\bar a(R)$ may possibly approach unity even if $y/R\downarrow 0$ (cf. Lemma \ref{Lem45}(i)). 
  By the same token  $P[\sigma_R^{x}<\sigma_0^{x}]$ may approach zero in case  $x/R \downarrow 1$  (under $m_+/m\to0$).

\begin{lemma}\label{Lem39}  Suppose $\de:= \limsup_{y\to\infty} a(-y)/a(y) <1$. Then  
$$\lim_{x\to \infty} \inf_{y\geq x} \frac{a(-y)}{a(-x)} \geq 1-\de.$$
In particular if $\de=0$ then  $a(-x)$ is asymptotically increasing with  $x$ in the sense that there exists an increasing function $f(x)$ such that $a(-x) =f(x)\{1+o(1)\}$ ($x\to\infty$)..
\end{lemma}
\pf\, For any  $\de'\in (\de,1)$ choose $N>0$  such that $a(-z)/a(z)<\de'$ for all $z\geq N$ and  let $N\leq x\leq y-N$ and  $z=y-x$. If  $a(-x)\leq a(-z)$, then on using (\ref{L6.1})
\beqn\label{aaa}
a(-y)-a(-z)= a(-x-z) -a(-z) \geq -\frac{a(-z)}{a(z)}a(-x) \geq -\frac{a(-z)}{a(z)}a(-z),
\eeqn
which by a simple rearrangement leads to
   $$a(-y) \geq [1- a(-z)/a(z)]a(-z) \geq (1-\de')a(-z)\geq (1-\de')a(-x).$$
   In case $a(-x)\geq a(-z)$, interchanging the roles of  $x$ and $z$ in (\ref{aaa}), we have $a(-y) \geq  (1-\de')a(-x)$ for $x>N$. Since $\lim_{x\to\infty}  a(-y)/a(-x)=1$ uniformly for  $x \leq y<x+N$ and $\de' -\de$ can be
   made arbitrarily  small we conclude the inequality of the lemma. \qed

    \vskip2mm
\textbf{ 7.2.} {\sc Relative stability of $Z$ and overshoots.}
\vskip2mm
     When  $EZ<\infty$,  according to  a standard  renewal theory  the law of the overshoot 
     \beqn\label{OVS}
 Z_R:=    S^0_{\sigma_{[R,\infty)}} -R
  \eeqn
      itself converges weakly as $R\to\infty$   to a proper probability distribution  \cite[(XI.3.10)]{F}; it also follows that $\lim_{n\to\infty}Z_n/n=0$ w.p.1, that entails   $\lim Z_R/R=0$ w.p.1.
On the other hand  in  case $EZ=\infty$, 
according to  Kesten \cite{K4} (or \cite[Section 4]{K}),
$\limsup_{n\to\infty}Z_n/[Z_1+\cdots Z_{n-1}] =\infty$ w.p.1,
where $Z_k$ are i.i.d.\;copies of $Z$,  or equivalently  $\limsup Z_R/R 
= \infty$ w.p.1.  It in particular follows that  
\[
\lim_{R\to\infty} Z_R/R =0\quad \mbox{w.p.1 \quad if and only if} \quad EZ<\infty.
\]  
In this subsection  we are interested in the convergence in probability which may be sometimes more significant than the w.p.1\;convergence  and holds true under a much weaker condition. In this respect the following result due to 
Rogozin \cite{R} is relevant:
\beqn\label{rgzn}
\begin{array}{ll} \mbox{\textit{$Z $ is r.s. if and only if  $Z_R/R \stackrel{P}\longrightarrow 0\,;$  and for this}}\\
\mbox{\textit{to be the case  it is sufficient that $X$ is positively r.s. }}
\end{array}
\eeqn

By combining Theorem \ref{Thm6} and a known criterion for positive relative  stability of $X$ (cf.  \cite[Eq(1.15)]{KM0}) we obtain a  reasonably fine sufficient condition  for $Z$ to be r.s.
For  condition (C.I) in the  following result we do not assume $E|X|<\infty$ nor the recurrence of  the r.w. $S_.$. Recall $A(x) = \int_0^x \{\mu_+(y)-\mu_-(y)\}dy$ as defined in (\ref{A}).
\begin{proposition}\label{Prop40}\, For $Z$ to be relative stable each of the following conditions  is sufficient.
\vskip2mm
{\rm (C.I)}\quad   ${\displaystyle \lim_{x\to\infty} \frac{A(x)}{x\mu(x)} = \infty},$
\vskip2mm
{\rm (C.II)}\quad  $EX=0$ and  ${\displaystyle \lim_{x\to\infty} \frac{x\eta_+(x)}{m(x)}=0}$.
\end{proposition}
\begin{proof}  \,  
Condition (C.I) is equivalent to the positive relative stability of $X$ as is already mentioned and 
 hence it is a sufficient condition for relative stability of $Z$ in view of  (\ref{rgzn}). 
As for (C.II),  expressing 
$
P[ Z_R >\ep R] $ as the infinite series 
\[
\sum_{w\geq 1}g_{ [R,\infty)}(0,R-w) P[X > \ep R+w], 
\]
one  observes first that $g_{ [R,\infty)}(0, R-w) < g_{\{0\}}(-R, -w)\leq g_{\{0\}}(-R, -R) =2\bar a(R)$.  Note
that  (C.II) entails   (\ref{H0}) and accordingly  (H) holds so that  $\bar a(x)\asymp x/m(x)$ 
 by Theorems \ref{th:1} and  \ref{th:1_2}. Hence  for any $\ep>0$
\[
P[ Z_R >\ep R]  \leq 2\bar a(R)\sum_{w>0}P[X > \ep R+w] \asymp \frac{R\eta_+(\ep R)}{m(R)}\leq \frac{R\eta_+(\ep R)}{ m(\ep R)}.
 \]
Thus  (C.II) implies $Z_R/R\, \stackrel{P}\to \,0$, concluding the proof   in view of (\ref{rgzn}) again.  
 \end{proof} 
  
 \vskip2mm
\begin{remark}\label{Rem41}\, (a)  Condition (C.I) is stronger than (c$'$) (given in Remark \ref{Rem11}) which by (\ref{A/KM2}) is necessary and sufficient  in order that   $P[S_n>0] \to 1$, 
so that  (C.I) is  only of relevance  in such a case.   If $\mu_-/\mu$ is bounded away from zero, then (C.I) obviously follows from $(c')$.

(b)\, Condition (C.II)  is satisfied if $m_+/m \to 0$. The converse is  of course not true---(C.II) may be fulfilled even if $m_+/m \to 1$---and it  seems hard to find  any simpler substitute for (C.II).  Under the restriction
$ m_+(x) \asymp m(x)$, however,    (C.II) holds  if and only if  $m_+$ is s.v. (which is the case if  $x^2\mu_+(x) \asymp L(x)$ with a s.v. $L$).  

(c) \, If $F$ belongs to the domain of attraction of a stable law  and   Spitzer's condition  holds, 
then that either (C.I) or (C.II)  holds is  also necessary for $Z$ to be r.s.
 as will be discussed in Section 8.1.2. 
 \end{remark}

\vskip2mm
The following result, used in the next subsection, concerns an overshoot estimate for the r.w. conditioned on avoiding the origin.  Let  $\bar a^\dagger(x) := \frac12\big[a^\dagger(x)+a(-x)\big]$.

\begin{lemma}\label{Lem42}\, {\rm (i)}
Let $\delta$ be a positive number.  Then  for $z>0$
and  $x<R$ satisfying 
$P\big[\sigma^x_{[R,\infty)}<\sigma^x_0\big] \geq \delta \bar a^\dagger(x)/ \bar a(R) $,
 \beqn\label{P6.211}
 P\big[S^x_{\sigma{[R,\infty)}}> R +z \,\big|\, \sigma^x_{[R,\infty)}<\sigma^x_0\big] \leq 2\delta^{-1}\bar a(R)\eta_+(z).
\eeqn

{\rm (ii)}  If $m_+/m \to  0$, then for each $\ep>0$,  uniformly for $0\leq x<R$,
 \beqn\label{eq:P6.20}
 P\big[S^x_{\sigma{[R,\infty)}}> R+\ep R \,\big|\, \sigma^x_{[R,\infty)}<\sigma^x_0\big] \to 0\quad (R\to\infty).
 \eeqn
 \end{lemma}
\v2
 
The condition  $P[\sigma^x_{[R,\infty)}<\sigma^x_0] \geq \delta \bar a^\dagger(x)/ \bar a(R) $
is  obviously weaker than $P[\sigma^x_{R}<\sigma^x_0] \geq \delta \bar a^\dagger(x)/ \bar a(R) $, but only the latter,  easier to check and valid for any $x$  fixed, will be used in our applications of (i).
\begin{proof}  \,  Suppose    $\bar a(x)\asymp x/m(x) \; (x\geq1)$ and  put
\[
r(z)= P\big[ S^x_{\sigma[R,\infty)} >R+z, \sigma^x_{[R,\infty)}<\sigma^x_0\big].
\]
 Plainly  $g_{\{0\}\cup [R,\infty)}\leq g_{\{0\}}$  and
 $g_{\{0\}}(x,z)\leq 2\bar a^\dagger(x)$,   hence
\[
r(z) =  \sum_{w\geq1}g_{\{0\}\cup [R,\infty)}(x,R-w) P[X > z+w] \leq 2\bar a^\dagger(x)\sum_{w\geq1}P[X > z+w].
\]
By 
$\sum _{w\geq1} P[X > z+w] = \eta_+(z+1) $ it therefore  follows that  if  $P[\sigma^x_{[R,\infty)}<\sigma^x_0] \geq \delta \bar a^\dagger(x)/ \bar a(R)$, 
\[
r(z) \leq 2\bar a^\dagger (x)\eta_+(z) \leq P[\sigma^x_{[R,\infty)}<\sigma^x_0] \times 2\delta ^{-1} \bar a(R) \eta_+(z)
\]
and dividing by $P[\sigma^x_{[R,\infty)}<\sigma^x_0]$ we find    (\ref{P6.211}).

Suppose $m_+/m \to  0$. Then  $\bar a(x) \asymp x/m(x)$ owing to Theorem \ref{th:1_2} , and 
 $a(-x)/a(x)\to 0$ ($x\to\infty$) by  Theorem  \ref{Thm6}.  We can accordingly apply  Lemma \ref{Lem37} to see that    $P[\sigma^x_R<\sigma^x_0] = \bar a^\dagger(x)/ \bar a(R) \{1+o(1)\}$ uniformly  for $0\leq x<R$, so that (\ref{P6.211}) obtains on the one hand.  On the other hand for $z>0$, recalling $\eta_+(z)< m_+(z)/z$  we deduce that 
  \beqn\label{AB}
 \bar a(R)\eta_+(z)\leq C\frac{m_+(z)\bar a(R)}{m(z)\bar a(z)},
 \eeqn
of which  the RHS with $z=\ep R$ tends to zero.  Thus  (\ref{eq:P6.20}) follows.
 \end{proof}

 \vskip2mm
 
 We shall need an estimate of overshoots  as the r.w. exits from the half line $(-\infty, -R]$ after its entering  into it.
 Put 
 \beqn\label{tau}
 \tau^x(R)= \inf \big\{n>\sigma^x_{(-\infty,-R]} : S^x_n \notin (-\infty, -R]\big\},
 \eeqn
 the first time when the r.w. exits  from $(-\infty,-R]$ after once entering it.
\begin{lemma}\label{Lem43} \, Suppose $m_+/m\to0$. Then for each constant  $\ep>0$,   uniformly for $x>-R$ satisfying 
$P[\sigma^x_{(-\infty,-R]}<\sigma^x_0] \geq \ep \bar a^\dagger(x)/ \bar a(R) $, as $R\to\infty$
\[
P\big[S^x_{\tau(R)}>-R+\ep R \,\big|\, \sigma^x_{(-\infty,-R]}<\sigma_0^x\big]  \to 0.
\]
\end{lemma}

\begin{proof} \, 
Denoting by ${\cal E}^x$   the event  $\{\sigma^x_{(-\infty,-R]}<\sigma^x_0\}$ we write down 
\beqn\label{eq:L6.41}
P\big[S^x_{\tau(R)}>-R+\ep R\,\big|\, {\cal E}^x ]=\sum_{w\leq -R} P\big[S^x_{\sigma(-\infty,-R]}=w\,\big|\, {\cal E}^x\big]P\big[S^w_{\sigma(-R,\,\infty)}>-R+\ep R\big].
\eeqn
 If  $P[{\cal E}^x] \geq \ep \bar a^\dagger(x)/ \bar a(R) $, by Lemma \ref{Lem42}(i) (applied to  $-S^x_\cdot$) 
\[
 P\big[S^x_{\sigma (-\infty, -R]}< - R -z \,\big|\, {\cal E}^x\big] \leq 2\ep^{-1}m(z)\bar a(R)/z.
 \]
Given    $\delta>0$ (small enough) we  define $\zeta=\zeta(\delta, R)$ $(>R)$ by the equation
\[
\frac{m(\zeta)R}{ m(R)\zeta} =\delta
\]
(uniquely determined since  $x/m(x)$ is increasing), so that by  $2\bar a(R) <CR/m(R)$
\[
 P\big[S^x_{\sigma (-\infty, -R]}< - R - \zeta \,\big|\, {\cal E}^x \big] \leq2\ep^{-1} m(\zeta)\bar a(R)/\zeta \leq   (C\ep^{-1}) \delta.
 \]
 For $-R -\zeta \leq  w \leq -R$,
\beq
P\big[S^{w}_{\sigma(-R,\infty)} >-R +\ep R\big] &=&  \sum_{y > -R+ \ep R}\sum_{z \leq -R} g_{(-R, \, \infty)}(w,z)p(y-z) \\
&=& \sum_{y > \ep R}\sum_{z \leq 0} g_{[1,\infty)}(w+R,z)p(y-z)\\
&\leq& C_1\bar a(\zeta) \eta_+(\ep R)\\
&\leq& \delta^{-1}C' R\eta_+(\ep R)/m(R),
\eeq
where the first inequality follows from  $g_{[1,\infty)}(w+R,z) \leq g_{\{1\}}(w+R,z)\leq \bar a(w+R-1)$ and the second from  $\bar a(\zeta) \leq C_2 \zeta/m(\zeta)$ and the definition of $\zeta$.  
Now, returning to  (\ref{eq:L6.41}) we 
apply the bounds derived above  to see 
\[
P\big[S^x_{\tau(R)}>-R+\ep R \,\big|\, \sigma^x_{(-\infty,-R] }<\sigma_0^x\big] \leq  (C\ep^{-1}) \delta +\delta^{-1}C' R\eta_+(\ep R)/m(R),
\]
Since  $R\eta_+(\ep R)/m(R) \to 0$ and  $\delta$ may be arbitrarily small, this concludes the proof.
\end{proof}
\vskip2mm

If $\tau$ is a stopping time of the r.w.  $S^0$ and ${\cal E}$ is an event depending only on  $\{S^0_n, n\geq \tau\}$, then by strong Markov property  $P[{\cal E},  \sigma^0_0 < \tau]= P({\cal E}) P[\sigma^0_0<\tau]$ so that   ${\cal E}$ is stochastically independent of $ \{\sigma^0_0<\tau\}$ and hence of $\{ \sigma^0_0\geq \tau\}$. Since $\{\tau^0(R)\leq \sigma^0_0\}$ coincides with $\{\sigma^0_{-\infty,-R]}<\sigma_0^0\}$, it follows that
\beqn\label{stch_ind}
P\big[S^0_{\tau(R)}= x\big] = P\big[S^0_{\tau(R)}=x\, \big|\, \sigma^0_{(-\infty, -R]} < \sigma^0_0 \big]
\eeqn 
so that Lemma \ref{Lem43} yields the following 
\begin{corollary}\label{Cor44} \, Suppose $m_+/m\to0$. Then for any $\ep>0$, as $R\to\infty$
\[
P[S^0_{\tau(R)}>-R+\ep R ]  \to 0.
\]
\end{corollary}


\v2
\textbf{ 7.3.} {\sc Asymptotic form of  $P[\sigma^x_{[R,\infty)} <\sigma^x_{(-\infty,0]}] $ under $m_+/m\to0$.}
\v2
Put  
\[
k:=\sup_{x\geq 1}\frac{a(-x)}{a(x)},
\] 
with the understanding that $k=\infty$ if   the r.w. is left-continuous.
 The following bounds that play a crucial  role in this section is  taken from   \cite[Lemma 3.4]{Uladd}:  
\beqn\label{sb6.3}
0\leq g_{\{0\}}(x,y) - g_{(-\infty,0]}(x,y) \leq 
(1+k) a(-y) \qquad (x, y\in \mathbb{Z}).
\eeqn
For $0\leq x\leq y$ we shall apply  (\ref{sb6.3})    in the slightly weaker form
\beqn\label{RL6.2}
-a(-x) \leq  a^\dagger(x) - g_{(-\infty,0]}(x,y) \leq (1+k)a(-y)+ka(-x),
\eeqn
  where the sub-additivity  $a(-y)-a(x-y)\leq a(-x)$ is used for the lower bound and the inequalities $a(x-y) -a(-y)\leq [a(x-y)/a(y-x)]a(-x)\leq ka(-x)$ for the upper bound.



\v2 
Put  $u_{\rm as}(0)=1$,  $u_{\rm as}(x)=\sum_{n=0}^\infty P[Z_1+\cdots +Z_n =x]$ ($x\geq1$), where $Z_1, Z_2,\ldots $ are i.i.d. copies of $Z$, and
 \beqn\label{ell_pm}
  \ell_+(x)= \int_0^x P[Z>t]dt, \quad \ell_-(x)= \int_0^x P[-\hat Z>t]dt.
  \eeqn
We know that   if $Z$ is r.s. then  $\ell_+$ is s.v. at infinity and  it accordingly follows (see \cite{Urenw} or Appendix (B)) that 
\beqn\label{u/ell}
u_{\rm as}(x) \sim 1/\ell_+(x).
\eeqn

In the rest of this subsection we shall suppose
\beqn\label{a/a/Z}
 Z \; \;\mbox{is r.s. \quad and}\quad  \frac{a(-x)}{a(x)} \to 0 \quad (x\to\infty).
\eeqn
By Theorem \ref{Thm6} and Proposition \ref{Prop40} this condition holds if $m_+/m\to\infty$.
The first half of (\ref{a/a/Z}) ensures that   $\ell_+$ is s.v. and (\ref{u/ell}) holds, while the second one  implies that $a(-x)$ is asymptotically increasing according to Lemma \ref{Lem39}, hence   by  (\ref{RL6.2})    as $x\to \infty $
\beqn\label{g/a_0} 
g_{(-\infty,0]}(x,y)/a(x) \to 1 \quad \mbox{uniformly for}  \;\;x\leq y \leq Mx
\eeqn
for each $M>1$ (in fact even with $M\to\infty$ such that  $a(-Mx)/a(x)\to0$).
Recall that $V_{\rm ds}(x)$ is the renewal function of the weakly descending ladder height process.


\begin{lemma}\label{Lem45}  
Suppose (\ref{a/a/Z}) to hold. Then $\ell_+$ varies slowly, and the following hold
\v2
{\rm (i)}  \quad $a(x) \sim V_{{\rm ds}}(x)/\ell_+(x)$\quad as\quad $x\to\infty$; 
 \v2  
{\rm (ii)}\quad  uniformly for $1\leq x\leq y$,  as $y\to\infty$ 
\beqn\label{eqL6.6}
g_{(-\infty,0]}(x,y) \sim V_{\rm ds}(x-1)/\ell_+(y).
\eeqn
\end{lemma}

\v2
\n
\pf Let   $Z$ be r.s. so that $\ell_+$   varies slowly. It is known \cite[Proposition 19.3]{S} that  for $1\leq x\leq y$
\beqn\label{g/vu}
 g_{(-\infty,0]}(x,y) =\sum_{k=1}^x v_{\rm ds}(x-k)u_{\rm as}(y-k),
 \eeqn
where $v_{\rm ds}(x) =V_{\rm ds}(x) - V_{\rm ds}(x-1)$ 
($x\geq 0$) and $v_{\rm ds}(0) =V_{\rm ds}(0)$
and owing to  (\ref{u/ell}) we see
\beqn\label{g/2x}
g_{(-\infty,0]}(x,2x) \sim  V_{\rm ds}(x-1)/\ell_+(x).
\eeqn
If $a(-x)/a(x) \to 0$, by  (\ref{g/a_0}) we also have   $g_{(-\infty,0]}(x,2x)  \sim a(x)$, whence the equivalence relation of (i)  follows.  (ii) follows from (\ref{g/vu}) and the slow variation of $u_{\rm as}$ for $1\leq x<y/2$,
and from (\ref{g/a_0})  in conjunction  with (i)  for $y/2\leq x\leq y$.
\qed 
\v2
Define a function  $f_r$ on $\mathbb{Z}$ by $f_r(x) = \sum_{y=1}^\infty V_\ds(y-1)p(y-x)$  for $x\leq0$ and  
$$f_r(x)=V_{\rm ds}(x-1), \quad x\geq1.$$
Then $f_r(x)=P[\hat Z < x]$ ($x\leq 0$)  (cf. \cite[Eq(2.3)]{Uladd})  and
$f_r(x)$ is a harmonic function of the r.w. killed as it enters $(-\infty, 0]$, in the sense that 
$\sum_{y\geq1} f(y)p(y-x)= f_r(x)$ $  (x\geq1)$. For each $x\in \mathbb{Z},$ $M^x_n := f_r (S^x_{n}) {\bf 1}(n< \sigma^x_{(-\infty,0]})$ is a martingale, so that by optional stopping theorem
$$f_r(x) \geq \liminf_{n\to\infty} E\big[M^x_{n\wedge \sigma_{[R,\infty)}}\big]\geq f_r(R)P\big[\sigma^x_{[R,\infty)} <\sigma^x_{(-\infty,0] }\big]. $$
Hence
\beqn\label{Mn}
f_r(x)/f_r(R) \geq P\big[\sigma^x_{[R,\infty)} <\sigma^x_{(-\infty,0] }\big] \geq P\big[\sigma^x_R <\sigma^x_{(-\infty,0] }\big].
\eeqn
The last probability equals $g_{(-\infty,0]}(x,R)/ g_{(-\infty,0]}(R,R)$ which is  asymptotic to $f_r(x)/f_r(R)$ because of  Lemma \ref{Lem45}(ii) (note that (\ref{eqL6.6}) extends to $x\leq 0$). 
This leads to the following

\begin{proposition}\label{Prop46} If (\ref{a/a/Z}) holds, then uniformly for $ x \leq R$, as $R\to\infty$
\beqn\label{eqP6.20}
 P\big[\sigma^x_R <\sigma^x_{(-\infty,0] }\big] \sim P\big[\sigma^x_{[R,\infty)} <\sigma^x_{(-\infty,0]}\big] \sim  f_r(x)/f_r(R),
\eeqn
and for $x=0$, in particular,  $P\big[\sigma^0_{[R,\infty)} <\sigma^0_{(-\infty,0]}\big] \sim 1/f_r(R)$.
\end{proposition}


\begin{remark}\label{Rem47} Let  $m_+/m\to0$. By Proposition \ref{Prop5} and  Lemma \ref{Lem45}(i)  it follows that as $x/R\to1$, $V_{\rm ds}(x)/ V_{\rm ds}(R) \to1$ and hence 
 $P[\sigma^x_{[R,\infty)} <\sigma^x_{(-\infty,0]}]\to 1$. [Note that if $U_{\rm as}$ is s.v., then as a dual result we may have  $P[\sigma^x_{[R,\infty)} <\sigma^x_{(-\infty,0]}]\to 0$ even when $x/R\to 1$.] It   holds that for $1\leq x\leq R$, 
$$
 f_r(x)/f_r(R) \sim  a(x)/a(R)
 $$ 
 as $R\to\infty$
 along with  $a(-R)/a(x) \to 0$ (possible even when $x/R\to0$).
 Hence the asymptotics of each probability in (\ref{eqP6.20})  agrees with that of $P[\sigma^x_{R} < \sigma^x_0]$  at least under $a(-R)/a(x) \to 0$ because of Lemma \ref{Lem37}, but does not if  $\ell_+(x)/\ell_+(R)\to 0$.
 

\end{remark}

\begin{remark}\label{Rem48} \;   
 There are  some results concerning  the  two sided exit problem. 
 If  the distribution of $ X$ is symmetric and belongs to the domain of normal attraction of a stable law with exponent $0<\alpha\leq 2$,  the problem is investigated by Kesten \cite{K2}: for $1<\alpha\leq2$ he identifies the limit of $P_x[\sigma_{[R,\infty)}< T]$ as $x\wedge R\to\infty$ so that $x/R\to \lambda\in (0,1)$.  For L\' evy processes with no positive jumps   there are certain definite results (cf. \cite[Section 7.1-2]{Bt}, \cite[Section 9.4]{D_L})). 
\end{remark}



\vskip2mm
\textbf{ 7.4.} {\sc Proof of Theorem \ref{Thm8}}

If  $m_+/m\to 0$, then 
 it follows \cite[Theorem1.1]{Unote} that for (a) in (\ref{a-e}) to hold, i.e.,  $\lim P[S_n>0] \to 1/\alpha$,  it  is necessary and sufficient that  $1\leq \alpha\leq 2$ and 
\beqn\label{Unote1}
\begin{array}{ll}
\int_{-x}^0t^2dF(t) \sim 2L(x) \quad &\mbox{if}\quad  \alpha =2,\\[1mm]
F(-x)  \sim (\alpha-1)(2-\alpha)x^{-\alpha}L(x)  \quad &\mbox{if}\quad 1<\alpha < 2,\\[1mm]
\int_{-\infty}^{-x} (-t)dF(t)  \sim L(x)  \quad &\mbox{if}\quad  \alpha =1
\end{array}
\eeqn
with a s.v. function  $L(x)$ at infinity. By  standard arguments (\ref{Unote1}) 
is equivalent to 
(b): $m_-(x) \sim  x^{2-\alpha}L(x)  \;\; (1\leq \alpha \leq 2)$. [If  $1<\alpha<2$ this is immediate by the monotone density theorem \cite{BGT}; for $\alpha =1$ observe that  the condition in  (\ref{Unote1}) implies  that $xF(-x)/\int_{-\infty}^{-x} (-t)dF(t) \to 0$ (see e.g., Theorem VIII.9.2 of \cite{F}), and  then that $\eta_-(x)\sim L(x)$---the implications of opposite direction is easy; for $\alpha=2$ see (1) of  Appendix (A).]  Thus (a) and (b) are equivalent.

 In view of Proposition \ref{Prop46} and Lemma \ref{Lem45}(i),  (c) to (e) in (\ref{a-e})  are equivalent to one another provided $m_+/m\to 0$, and for the proof of Theorem \ref{Thm8} it suffices to show the equivalence of (b) and  (d).  
This equivalence is involved in the following result.

\begin{proposition}\label{Prop49} Let $m_+/m\to 0$. Then 
\v2
{\rm (i)} $V_{\rm ds}(x)U_{\rm as}(x)\sim xa(x) \asymp x^2/m(x)$; and
\v2
{\rm (ii)}  in order that  $m_-$   varies regularly  with index $2-\alpha$ $(\in [0,1])$,  it is necessary and sufficient  that  $V_{\rm ds}$ varies regularly with index $\alpha-1$, in which case
it  holds that
 \beqn\label{d/b}
V_{\rm ds}(x)U_{\rm as}(x) \sim C_\alpha x^2/m(x), 
\eeqn
where $C_\alpha = 1/\Ga(\alpha)\Ga(3-\alpha)$.
 \end{proposition}
 

Suppose  $m_+/m\to 0$ and  $m_-$ is regularly varying  with index  $\alpha-2$. From Lemma 6.1 of \cite{Uexit} (see also Remark 6.1 of it)  it follows that   for $1<\alpha\leq 2$, $v_{\rm ds}(x) \sim (\alpha-1)x^{-1}V_{\rm ds}(x)$, hence by (\ref{d/b})
$$v_\ds(x) \sim (\alpha-1)C_\alpha \ell_+(x)/m(x).$$

If $1<\alpha\leq 2$ and $F$ is in the domain of attraction of a stable law, (\ref{d/b}) is shown  by Vatutin-Wachtel \cite{VW}; our proof,  quite different from theirs,   rests on Lemma \ref{Lem45}. For the r.w.s attracted to  stable laws  with exponent $0<\alpha<2$ the precise asymptotic form of $V_{\rm ds}(x)U_{\rm as}(x) $ is given in \cite[Lemma 5.1]{Uexit} except in the case when $\alpha= 2p=1$,  $E|X|=\infty$ and $\rho\in \{0,1\}$.

 
  

  The first result  (i) of Proposition \ref{Prop49} follows immediately from Lemma \ref{Lem45}(i)  (since $U_{\rm as}(x)\sim x/\ell_+(x)$).
The proof  of (ii)  will be  based on the identity
  \beqn\label{R_hat_Z}
    P[\hat Z\leq -x] = V_{\rm ds}(0)\sum_{y=0}^\infty u_{\rm as}(y)F(-y-x) \quad (x>0)
  \eeqn
which one can derive from (\ref{g/vu}) (or  the duality lemma that says  $u(y)= g_{(-\infty,0]}(0,y)$, $y\geq 0$ \cite[Section XII.1]{F}).
 Recall  (\ref{u/ell}), i.e.,  $u_{\rm as}(x)\sim 1/\ell_+(x)$ and that $\ell_+$ is s.v. (at least under $m_+/m\to 0$).
 Here it is also noted that for $1\leq \alpha<2$,
  if  either $V_{\rm ds}(x)$ or  $1/P[\hat Z\leq -x]$ varies regularly as $x\to\infty$ with index $\alpha-1$, then  
   \beqn\label{ZV1}
   P[\hat Z\leq -x] V_{\rm ds}(x) \longrightarrow V_{\rm ds}(0)/\Ga(2-\alpha)\Ga(\alpha)
   \eeqn
     (cf. e.g., \cite{BGT}, pp.364-5). 
     \v2
 {\it Proof of the necessity part of Proposition  \ref{Prop49}(ii). } By virtue  of the first inequality of (\ref{Mn})   $V_{\rm ds}$ is s.v.  if and only if  
 $P[\sigma^x_{[R,\infty)} <\sigma^x_{(-\infty,0]}] \to1$ as $R\to\infty$ for  $x\geq R/2$ which is equivalent to  $\lim P[S^0_n>0]=1$ due to Kesten-Maller \cite{KM} as noted in (\ref{A/KM2}). 
 
 For $\alpha=1$, by what is mentioned in the beginning   of this subsection this verifies    the equivalence stated in the first half of the proposition. [Here sufficiency part is also proved, of which we shall give a proof  without resorting to \cite{KM} (cf. Lemma \ref{Lem51}).]

 For $1<\alpha<2$ by (\ref{R_hat_Z}) condition (\ref{Unote1})   implies that
\beqn\label{hatZ/U}
P[\hat Z <-x] \sim (\alpha-1)\kappa\sum_{y=0}^\infty \frac{L(x+y)}{\ell_+(y)(x+y)^\alpha}  \sim \kappa \frac{L(x)}{\ell_+(x)x^{\alpha-1}}
\eeqn
($\kappa=(2-\alpha)V_{\rm ds}(0) $), and hence $V_{\rm ds}$ varies regularly with index $\alpha-1$ because of (\ref{ZV1}).   

Let $\alpha=2$. Then  the slow variation of $m_-$ entails that $F$ is in the domain of attraction of a normal law (cf. Appendix (A)) and we shall see (Proposition \ref{7.1}(i)) that  this entails $2\bar a(x) \sim x/m_(x)$, hence $V_{\rm ds}(x)\sim a(x)\ell_+(x) \sim x\ell_+(x)/m_-(x)$ owing  to Lemma \ref{Lem45}. \qed
\v2
 
In the proof of the sufficiency part of Proposition \ref{Prop49}(ii), we shall be concerned with the condition
\beqn\label{V/L1}
 V_{\rm ds}(x)\sim x^{ \alpha -1}/\hat \ell(x) 
 \eeqn
 with some $\hat \ell $ that is s.v.
at infinity.   The next lemma deals with  the case $1\leq \alpha <2$ when  (\ref{ZV1}) is valid.
 By (\ref{R_hat_Z})  and  (\ref{ZV1})   the above  equivalence relation 
may be rewritten as  
 \beqn\label{uF/L}
\frac{ P[\hat Z\leq -x]}{V_\ds(0)} = \sum_{y=0}^\infty u_{\rm as}(y)F(-x-y) = \frac{x^{1-\alpha}\hat \ell(x)}{ \Ga(\alpha)\Ga(2-\alpha)}\{1+o(1)\}.
   \eeqn

    Put
   $$\ell_\sharp(x) = x^{\alpha-1}\int_x^\infty \frac{F(-t)}{\ell_+(t)}dt.$$

    \begin{lemma}\label{Lem50} If   $Z$ is r.s. and  (\ref{V/L1})  
   holds for $1\leq \alpha<2$, then  
   $$ \hat\ell(x) \sim \Ga(\alpha)\Ga(2-\alpha)\ell_\sharp(x)  \quad\mbox{and}\quad  P[\hat Z\leq -x]/V_{\rm ds}(0) \sim  x^{1-\alpha }\ell_\sharp(x).$$
   \end{lemma}   
   \pf
    For this proof we do not need the condition $E|X|<\infty$  nor the recurrence of the r.w., but need the oscillation of the r.w. so that both $Z$ and $\hat Z$ are well-defined to be proper random variables and Spitzer's representation (\ref{g/vu}) of   $g_{(-\infty,0]}(x,y)$ is valid.
    Let $Z$ be r.s. so that  $\ell_+$ is s.v. and $u_{\rm as}(x)\sim 1/\ell_+(x)$.  Put 
    $$\Sigma(x)=  \int_{1}^\infty \frac{F(-t-x)}{\ell_+(t)}dt \quad\mbox{and}\quad  \De(x) = \frac{\ell_\sharp(x)}{x^{\alpha-1}} = \int_x^\infty \frac{F(-t)}{\ell_+(t)}dt.$$
 In view of  (\ref{uF/L})  it suffices to show the implication
    \beqn\label{G/L}
 \Sigma(x) \;\mbox{varies regularly of index} \; 1-\alpha
   \,\Longrightarrow \, \Sigma(x)  \sim \De(x).
 \eeqn  
 Clearly  $\Sigma(x)\geq F(-2x)\int_{1}^x dt/\ell_+(t) \sim  xF(-2x)/\ell_+(x)$.  Replacing  $x$ by 
  $x/2$ in this inequality and noting that  $\Sigma( x/2)\sim 2^{\alpha-1}\Sigma(x)$  because of the premise of (\ref{G/L}) 
    we obtain  
  \beqn\label{xF/sigma}
 x F(-x) /\ell_+(x) \leq   2^{\alpha}\,\Sigma(x)\{1+o(1)\}.
\eeqn
From the defining expressions of  $\Sigma$ and $\De$ one observes first  that 
 for each $\ep>0$, as $x\to \infty$ 
$$
  \Sigma\big(x)  \leq \frac{\ep xF(-x)}{\ell_+(x)}\{1+o(1)\} + \De((1+\ep)x),
$$
 then, by substituting  (\ref{xF/sigma}) into the RHS, that  
$$\big[1- 2^\al \ep\big]\Sigma(x)\{1+o(1)\} \leq \De\big((1+\ep)x\big)  \leq \De(x) \leq 
\Sigma(x)\{1+o(1)\}.$$ 
Since  $\ep$ may be arbitrarily small, we can conclude  $\De(x)/\Sigma(x) \to 1$ as desired.  
\qed
 
  
\begin{lemma}\label{Lem51}  {\rm (i)} \, Suppose that (\ref{V/L1}) holds with $\alpha=1$ and $\hat \ell =\ell_\sharp$ and that $Z$ is r.s. Then both $\eta_-$ and $\eta$ are s.v. and 
\beqn\label{eqL6.8}
\ell_+(t) \ell_\sharp(x) = A(t) + o(\eta_+(x)) \;\; \mbox{as}\;\; x\to \infty.
\eeqn
{\rm (ii)} Suppose that  $x\eta_+(x)/m_-(x)\to 0$  and  (\ref{V/L1}) holds with $\alpha=1$. Then  $Z$ is r.s., $\eta_-$ is s.v. 
and $\ell_+(t)\ell_\sharp(t) \sim A(t) \sim \eta_-(t)$.
\end{lemma}
\v2\noindent
\pf  \, (i) is the second half of  Lemma 4.2 in \cite{Uexit}.
As for proof of (ii), suppose $x\eta_+(x)/m_-(x)\to 0$. Then    $Z$ is r.s. by Proposition \ref{Prop40}.    Let  $V_{\rm ds}(x)\sim 1/\hat \ell(x)$. Then   by Lemma \ref{Lem50} 
 $ \hat \ell(x) \sim \ell_\sharp(x)$  
 and (ii) follows from (i).
 \qed

 \v2
 {\it Completion of the proof of  Proposition \ref{Prop49}. }\;
 We  show that if  $V_{\rm ds}(x)$ varies regularly, then (\ref{Unote1}) and (\ref{d/b}) hold.
 As in (\ref{V/L1}) 
 let 
 $$V_{\rm ds}(x)\sim x^{\alpha-1}/\hat \ell(x)$$
  with a s.v.  $\hat \ell$ and $1\leq \alpha\leq 2$.

 \v2

 {\sc Case\,  $1\leq\alpha <2$}. Let  $C_\alpha' =1/[\Ga(\alpha)\Ga(2-\alpha)]=(2-\alpha)C_\alpha$. 
  By (\ref{R_hat_Z}), (\ref{ZV1}) and Lemma \ref{Lem50}
 \beqn\label{V/F/ell}
\frac{C_\al'\hat\ell(x)}{x^{\alpha-1}} \sim \frac{C_\al'}{V_{\rm ds}(x)} \sim  \frac{P[\hat  Z <-x] }{V_\ds(0)}\sim \frac{\ell_\sharp(x)}{\, x^{\alpha-1}}
 = \int_x^\infty \frac{F(-t)}{\ell_+(t)}dt.
 \eeqn

Let $1<\alpha < 2$. Then  
by the monotone density theorem---applicable since   $F(-t)/\ell_+(t)$ is decreasing---the above relation   leads to $F(-x) \sim \big[(\alpha-1)C_\alpha'\big] x^{-\alpha}\hat \ell(x)\ell_+(x)$ which is equivalent to (\ref{Unote1}) with $L(x)= C_\al \hat \ell(x)\ell_+(x)$ and shows
 (\ref{d/b}) again (that has been virtually seen at (\ref{hatZ/U})).

 If  $\alpha =1$. Then $U_{\rm as}(x)\sim 1/ \ell_+(x)$ and,  by (\ref{V/F/ell}) and Lemma  \ref{Lem51}(ii),  $1/V_{\rm ds}(x)\sim \eta_-(x)/\ell_+(x)$  which  is equivalent to  
 (\ref{Unote1}) and implies    (\ref{d/b}) as one can readily check.
       \qed




\v2
   {\sc   Case\, $\alpha =2$}.  We have 
    \beqn\label{uF_L}
\frac{\ell_-(x)}{V_{\rm ds}(0)} =\frac1{V_\ds(0)}\int_0^xP[-\hat Z>t]dt 
=  \int_0^x  dt   \sum_{y=0}^\infty u_{\rm as}(y)F(-y-t) 
   \eeqn
and, instead of 
 (\ref{ZV1}),   $\ell_-(x)V_{\rm ds}(x)/V_{\rm ds}(0)\sim x$ \cite[Eq(8.6.6)]{BGT},  so that 
      \beqn\label{x/V/L}
   \ell_-(x)/V_{\rm ds}(0)\sim x/V_{\rm ds}(x) \sim \hat \ell(x).
   \eeqn


 To complete the proof it suffices to show  that $m_-$ is  s.v.,  or equivalently,  $x\eta_-(x)/ m(x)\to 0$ (because of the assumption  $m_-\sim m$), because
     the slow  variation of $m$ entails that of $\int_{-x}^x y^2 dF(y)$ and hence   (\ref{d/b}) (see the proof of the necessity part).
 We shall verify the latter condition and to this end  we shall 
 apply the inequalities
 \beqn\label{m/ell}
 1/C \leq \ell_-(x) \ell_+(x)/ m(x) \leq C   \qquad (x>1)
 \eeqn
 (for  some positive constant  $C$), which  follow from (\ref{x/V/L}) and Lemma \ref{Lem45}(i), the latter entailing  $V_{\rm ds}(x)/\ell_+(x)\asymp x/m(x)$ in view of Theorems \ref{th:1} and  \ref{th:1_2}. 
  
  Using (\ref{uF_L}) and $u_{\rm as}(x)\sim 1/\ell_+(x)$ as well as the monotonicity of 
    $\ell_+$ and $F$  one observes that for each $M>1$ as $x\to\infty$
  \beq
  \frac{\ell_-(x)-\ell_-(x/2)}{V_{\rm ds}(0)} &\geq & \int_{x/2}^{x} dt\int_0^{(M-1)x} \frac{F(-y-t)}{\ell_+(y)\vee 1}\{1+o(1)\}dy\\ 
  &\geq& \int_{x/2}^{x} dt \int_0^{(M-1)x}\frac{F(-y-t)}{\ell_+(y+x)}dy \{1+o(1)\}
   \\
  &\geq&\frac{x/2}{\ell_+(x)}\big[\eta_-(x) - \eta_-(Mx)\big]\{1+o(1)\},
  \eeq
 which implies
$$x\eta_-(x)\leq x\eta_-(Mx) + \ell_-(x)\ell_+(x)\times o(1),$$
for  by (\ref{x/V/L})  $\ell_-$ is s.v. 
Combining this with  (\ref{m/ell}) one  infers that
$$   x\eta_-(x)  \leq \frac{m(Mx)}{M}\{1+o(1)\} \leq \frac{C [\ell_-\ell_+](x)}{M}\{1+o(1)\} \leq \frac{C^2m(x)}{M} \{1+o(1)\}$$
and  concludes that $x\eta_-(x)/m(x) \to 0$, hence  the desired slow variation of $m$. 
\qed
\vskip3mm

\textbf{ 7.5.} {\sc  Comparison between  $\sigma^x_R$ and $\sigma^x_{[R,\infty)}$ and one-sided escape from zero.
}
\vskip2mm
The results obtained in this subsection  are thought of as examining to what extent the results which are simple for the right-continuous r.w.s to those satisfying  $m_+/m \to 0$. 
The following one gives the asymptotic form in terms of $a$ of the  probability  of one-sided escape of the r.w. killed as it hits  $0$. 
\begin{proposition}\label{Prop52}\,{\rm (i)}  If $ \lim_{z\to\infty}a(-z)/a(z) = 0$, then  as $R\to\infty$
\beqn\label{eqP6.4}
P[\sigma^x_{(-\infty, -R]}<\sigma^x_0]\,\sim\, P[\sigma^x_{-R}<\sigma^x_0]  \quad\mbox{uniformly for $\;x\in \mathbb{Z}$.
}
\eeqn

{\rm (ii)}\,  Suppose $ m_+/m \to 0$. Then  $\lim_{z\to\infty} a(-z)/a(z) = 0$ and as $R\to\infty$
\beqn\label{eqP6.2}
P[\sigma^x_{[R,\infty)}<\sigma^x_0]\sim P[\sigma^x_R<\sigma^x_0] \sim \frac{a^\dagger(x)}{a(R)}
\eeqn
 uniformly for $x\leq R$ subject to the condition
\beqn\label{ass1}
\frac{a(-R) -a(-R+x)}{a(x)} \, \longrightarrow \, 0 \qquad (x\leq R).
 \eeqn
  \end{proposition}

Condition (\ref{ass1}) holds at least  for $0\leq x<R$ under $m_+/m\to\infty$ due to   (\ref{eL6.2}) which may read
\beqn\label{L6.5ii}
- \frac{a(-R+x)a(-x)}{a(R-x)a(x)} \leq \frac{ a(-R)-a(-R+x)}{a(x)} \leq  \frac{a(-R)}{a(R)}.
\eeqn
(See Remark \ref{rem7.3} at the end of Section 8.1.2 for more about  the relation of (\ref{ass1}) to (\ref{eqP6.2}).) 
\vskip2mm
\begin{proof}   
Let  ${\cal E}_{R}^x$ stand for the event $\{\sigma^x_{[R,\infty)}<\sigma^x_0\}$. Note that for each $x$,
 \beqn\label{Equiv}
P[\sigma^x_{[R,\infty)}<\sigma^x_0]\sim P[\sigma^x_R<\sigma^x_0] \Longleftrightarrow   \lim_{R\to\infty} P\big[ \sigma^x_{R}<\sigma^x_0\,\big|\, {\cal E}_R^x\big] =1.
 \eeqn 

 
 (i) is immediate fro Lemma \ref{Lem37}(i).

For the proof of (ii) let $m_+/m\to 0$.  By (\ref{eq:L6.1}) condition (\ref{ass1}) is then  equivalent to  the second equivalence of (\ref{eqP6.2}) under the supposition.  For the proof of the first,  
take a small constant  $\ep>0$. Then  for any $R$ large enough we can choose $u>0$ so that 
 \[
 u/m(u) = \ep R/m(R);
 \]
put $z=z(R,\ep) =\lfloor u\rfloor$. Since $\bar a(x)\asymp x/m(x)$, this entails
 \[
 \ep C'\leq \bar a(z) / \bar a(R) < \ep C''
 \]
for  some  positive constants $C'$, $C''$. Now define  
  $h_{R,\ep}$ via 
 \[
 P[\sigma^x_R<\sigma^x_0\,|\, {\cal E}_R^x] = h_{R,\ep} +  \sum_{R \leq y\leq R+z} P\big[S^x_{\sigma[R,\infty)} = y \,\big|\, {\cal E}_R^x\big] P[ \tilde \sigma^{y}_R<\sigma^{y}_0]
 \]
 with  $\tilde \sigma_R^y$ defined as in the proof of (i). 
 Then by Lemma \ref{Lem42}(i)  (see also (\ref{AB}))
 \[
  h_{R,\ep}\leq P\big[S^x_{\sigma[R,\infty)} >R+z \,\big|\, {\cal E}_R^x\big] \leq C\frac{m_+(z)}{m(z)}\cdot\frac{\bar a(R)}{\bar a(z)} \leq [C/\ep C'] \frac{m_+(z)}{m(z)},
 \]
 whereas 
$1-P[\tilde\sigma^{y}_R<\sigma^{y}_0]  
 \leq  {\bar a(y-R)}/{\bar a(R)} <\ep C_1 $ for $R\leq y\leq R+z$ (see Remark \ref{Rem36}(ii) for the first inequality).
Since $m_+(z)/m(z)\to 0$ so that $h_{R,\ep}\to 0$, we conclude 
 \[
  \liminf_{R\to\infty}\inf_{0\leq x<R}  P\big[\sigma^x_R<\sigma^x_0\,\big|\, \sigma^x_{[R,\infty)}<\sigma^x_0\big] >1-\ep C_1,
  \]
 hence $ P[\sigma^x_R<\sigma^x_0\,|\, \sigma^x_{[R,\infty)}<\sigma^x_0] \to 1$, for $\ep$ can  be made arbitrarily small.  This verifies  that the first equivalence in (\ref{eqP6.2}) holds uniformly for $0\leq x\leq R$. 
 
Letting $x=0$ in (\ref{eqP6.2}) gives that as $R\to\infty$
\[
P[\sigma^0_{[R,\infty)} <\sigma^0_0] \sim 1/a(R),
\]
which shows the asymptotic monotonicity of $a(x)$,   the probability on the LHS being a monotone function of  $R$.  With the help of the following lemma and the trivial inequality $P[\sigma^x_{[R,\infty)} <\sigma^x_0] \geq P[\sigma^x_R <\sigma^x_0]$ 
 as well as the second equivalence in (\ref{eqP6.2}) this shows 
 that the first one of  (\ref{eqP6.2}) holds   for $x<R$ subject to (\ref{ass1}). 
 \end{proof}


\begin{lemma}\label{Lem53}\,  If $a(x)$ is asymptotically increasing as $x\to\infty$,  then as $R\to\infty$
 $$P[\sigma^x_{[R,\infty)} <\sigma^x_0] \leq  \big[a^\dagger(x)\big/a(R)\big]\{1+o(1)\} 
  \quad \mbox{uniformly for}\;\;  x \in\mathbb{Z}.
  $$
\end{lemma}
\begin{proof}  \,  Since $a(S^x_{n})$ is a martingale (if $x\neq 0$),     the optional stopping theorem with the help of   Fatou's lemma shows
\[
E\big[a(S^x_{\sigma_{[R,\infty)}})\big] \leq a^\dagger(x).
\]
[Here in fact the equality holds  if $E|\hat Z|=\infty$ (cf. \cite[Eq(2.8)]{Uladd}).] The expectation on the LHS is bounded below 
by $a(R)P[\sigma^x_{[R,\infty)}<\sigma^x_0]\{1+o(1)\}$ by the assumed monotonicity of $a(x)$, hence the asserted inequality  follows.
  \end{proof} 


We have seen the asymptotic monotonicity of $a(x)$.  This together with Lemma \ref{Lem39} yields
\begin{corollary}\label{Cor54}\, If $m_+/m \to 0$, 
then both $a(x)$ and $a(-x)$ are asymptotically increasing. 
\end{corollary}



In the next subsection we shall state another corollary---Corollary \ref{Cor59}---of Proposition \ref{Prop52} that concerns the asymptotic distribution of $\sharp \{n<\sigma_{[R,\infty)}: S_n\in I\}$, the number of visits of a finite set $I\subset \mathbb{Z}$ by the r.w. before entering  $[R,\infty)$.


 \vskip2mm
\textbf{ 7.6.} {\sc Escape into $(-\infty,-R]\cup [R,\infty)$.}
\vskip2mm
Let $Q$ as well as $R$ be a positive integer.  Here  we consider the event 
$\sigma^x_{\mathbb{Z} \setminus (-Q,R)}< \sigma^x_0$,  the escape  into $\mathbb{Z} \setminus (-Q,R)=(-\infty, -Q]\cup [R,\infty)$   from the killing at 0. The next result is  essentially a corollary of Proposition \ref{Prop52} and Lemma \ref{Lem53}.
For simplicity we write $\sigma^x_{y, z}$ for $\sigma^x_{\{y,z\}}$ so that 
$$\sigma^x_{y, z} =\sigma^x_y\wedge \sigma^x_z.$$ 
\begin{proposition}\label{Prop55} \, If  $m_+/m\to 0$,  then  uniformly for $x\in (-Q,R)$ subject to  (\ref{ass1}), as $Q\wedge R \to\infty$
  \beqn\label{cons1}
 P[\sigma^x_{\mathbb{Z} \setminus (-Q,R)}  <\sigma_0^x] = P[\sigma^x_{ -Q,R}  <\sigma_0^x]\{1+o(1)\}.
 \eeqn
 \end{proposition}
\begin{proof}  \, Put $\tau^x_-=\sigma^x_{(-\infty, -Q]}$ and $\tau^x_+= \sigma^x_{ [R,\infty)}$. It  suffices to show that
\[
\frac{ P[\tau^x_-\wedge \tau^x_+  <\sigma^x_0 < \sigma^x_{-Q,R} ]}
{P[\tau^x_-\wedge \tau^x_+ <\sigma_0^x]} \to 0.
\]
The numerator of the ratio above is less than
 \[
 P[\tau^x_- <\sigma^x_0 < \sigma^x_{-Q} ]+ P[ \tau^x_+ <\sigma^x_0 < \sigma^x_{R} ] = P[\tau^x_- < \sigma_0^x]\times o(1) + P[\tau^x_+< \sigma_0^x]\times o(1)
 \]
 under (\ref{ass1}), where Proposition \ref{Prop52}---both (i) and (ii)---is applied for the bound of 
  the LHS.  Hence it is  of smaller order of magnitude than the denominator.  
 \end{proof}  
 
 For any  subset $B$ of $\mathbb{R}$ such that $B\cap \mathbb{Z}$ is non-empty,  define 
 \beqn\label{def_H}
 H^x_B(y)= P[S^x_{\sigma_{B} } =y ]\quad \;\; (y\in B),
 \eeqn
  the hitting distribution of $B$ for the r.w. $S^x$.
 Put
  \[
  B(Q,R)=\{-Q,0, R\}.
  \] 
   Then (\ref{cons1}) is rephrased as
 \beqn\label{esc2}
 P[\sigma^x_{\mathbb{Z} \setminus (-Q,R)} < \sigma_0] \sim  1- H^x_{B(Q,R)}(0).
 \eeqn
 By using Theorem 30.2 of Spitzer \cite{S}  one can compute an explicit expression of $H^x_{B(Q,R)}(0)$ in terms of $a(\cdot)$, which though useful  is pretty complicated.  The following result  is derived without using it.

\begin{lemma}\label{Lem56}   For $x\in \mathbb{Z}$
\beq
  P[\sigma^x_{-Q} <\sigma_0^x] 
&\geq&  
1-H^x_{B(Q,R)}(0) - P[\sigma^x_R <\sigma_0^x]P[\sigma^R_{0} <\sigma^R_{-Q}]  \\
&\geq&    P[\sigma^x_{-Q} <\sigma_0^x] P[\sigma^{-Q}_{0} <\sigma^{-Q}_{R}].  
\eeq
\end{lemma}
\begin{proof}  \, Write   $B$ for $B(Q,R)$.  
 Plainly we have
\beqn\label{eP6.3}
1-H^x_B(0) = P[\sigma^x_{-Q,R} <\sigma^x_0] = P[ \sigma^x_{-Q} <\sigma^x_0] + P[\sigma^x_R <\sigma^x_0] - P[\sigma^x_{-Q} \vee\sigma^x_R  < \sigma^x_{0}].
\eeqn
Let ${\cal E}^x$ denote  $\{\sigma^x_R<\sigma^x_0\}\cap\{S^x_{\sigma_R + \,\cdot}$ hits $-Q$  before $0$\},  the event that   $S^x_\cdot$  visits $-Q$ after the first visit of $R$ without visiting  $0$.
 Then
\[
 {\cal E}^x \subset \{\sigma^x_R\vee \sigma^x_{-Q} <\sigma^x_0\}\quad \mbox{and}
\quad \{\sigma^x_{-Q} \vee\sigma^x_R <\sigma^x_0\} \setminus {\cal E}^x \subset \{\sigma^x_{-Q}< \sigma^x_{R} <\sigma^x_0\}.
\]
 By  $P[\sigma^x_{-Q}< \sigma^x_{R} <\sigma^x_0] \leq P[\sigma^x_{-Q} <\sigma^x_0]  P[\sigma^{-Q}_{R} <\sigma^{-Q}_0]$ it therefore follows that
\beqn\label{eP6.31}
0\leq P[\sigma^x_{-Q} \vee\sigma^x_R < \sigma^x_{0}] -P({\cal E}^x) \leq   P[\sigma^x_{-Q} <\sigma^x_0]P[\sigma^{-Q}_{R} <\sigma^{-Q}_0].
\eeqn
The left-hand inequality together with  $P({\cal E}^x) = P[\sigma^x_R  < \sigma^x_{0}] P[\sigma^R_{-Q} < \sigma^R_{0}] $ 
leads to 
\[
P[\sigma^x_R <\sigma^x_0] - P[\sigma^x_{-Q} \vee\sigma^x_R  < \sigma^x_{0}] \leq P[\sigma^x_R <\sigma^x_0] 
- P[{\cal E}^x] = P[\sigma^x_R <\sigma_0^x]P[\sigma^R_{0} <\sigma^R_{-Q}], 
\]
which,  
substituted into (\ref{eP6.3}), yields the first inequality of the lemma after  trite transposition of a term. In a similar way  the second one is deduced by substituting the right-hand inequality of (\ref{eP6.31})  into (\ref{eP6.3}). 
\end{proof}

According to  Lemma 3.10 of \cite{Uladd}  
for any finite set $B\subset \mathbb{Z}$ that contains  $0$ we have the identity
\beqn\label{HB}
1-H^x_B(0) = [1-H^0_B(0)]a^\dagger(x) +\sum_{z\in B\setminus\{0\}}[a(-z)-a(x-z)]H^z_B(0), \quad x\notin B\setminus \{0\}.
  \eeqn
Recalling  $P[\sigma^x_0 <\sigma_x^x] = 1/2\bar a(x) =P[\sigma^0_x <\sigma_0^0] $ ($x\neq0$) we see that  $H^{R}_{B(Q,R)}(0) \vee H^{-Q}_{B(Q,R)}(0) \leq P[\sigma^{0}_{R}<\sigma_{0}^{0}] \vee P[\sigma^{0}_{-Q}<\sigma_{0}^{0}]  
\leq 1- H^{0}_{B(Q,R)}(0)$.  
 Hence for  $B=B(Q,R)$, the second term on the RHS of (\ref{HB}) is negligible in comparison to the first  under the condition  
 \beqn\label{cnstr}
|a(Q) - a(x+Q)|  + |a(-R) -a(x-R)| =o(a^\dagger(x)),
\eeqn
 which  concludes the following result.
\begin{lemma}\label{Lem57}  Uniformly for $-Q<x<R$ subject to (\ref{cnstr}), 
as $Q\wedge R  \to\infty$  
\beqn\label{eP0} 1-H^x_{B(Q,R)}(0) \sim a^\dagger(x)(1-H^0_{B(Q,R)}(0) ).
\eeqn 
\end{lemma}

Condition (\ref{cnstr}) (to be understood to entail $a^\dagger(x)\neq0$)---always satisfied for each $x$  (fixed)  with $a^\dagger(x)\neq0$---is  necessary and sufficient for   the following  condition  to hold: 
\beqn\label{eq:6.3}
P[\sigma^x_{-Q} <\sigma^x_0]\sim a^\dagger(x)/2\bar a(Q)\quad\mbox{and} \quad  P[\sigma^x_{R} <\sigma^x_0] \sim  a^\dagger(x)/2\bar a(R).
\eeqn
\begin{proposition}\label{Prop58} Suppose  $m_+/m \to 0$. Then  
uniformly for $-Q < x<R$ subject to  condition (\ref{cnstr}), 
as $Q\wedge R \to\infty$  
\beqn\label{eP} 
 P[\sigma^x_{\mathbb{Z} \setminus (-Q,R)} < \sigma^x_0] \sim a^\dagger(x)(1-H^0_{B(Q,R)}(0) )  \sim a^\dagger(x)\frac{a(Q+R)}{a(Q)a(R)}. 
 \eeqn
 \end{proposition}


\begin{proof} \,On  taking $x=0$ in  Lemma \ref{Lem56} its second   inequality yields  
\beq
1-H^0_{B(Q,R)}(0) &\geq & \frac{a(Q+R) +a(-Q) -a(R)}{4\bar a(R)\bar a(Q)} + \frac{a(-Q-R) +a(R) -a(-Q)}{4\bar a(Q)\bar a(R) }\\
& =& \frac{a(Q+R) +a(-Q-R)}{4\bar a(Q)\bar a(R)} \\
&=& \frac{a(Q+R)}{a(Q)a(R)}\{1+o(1)\}, 
\eeq
and the first  inequality  gives  the upper bound  
\[
1-H^0_{B(Q,R)}(0) \leq  \frac{a(Q+R) +a(-Q)+ a(-R)}{4\bar a(Q)\bar a(R)} =  \frac{a(Q+R)}{a(Q)a(R)}\{1+o(1)\},
\]
showing the second relation of (\ref{eP}).
Proposition \ref{Prop55} in conjunction  with Lemma \ref{Lem57}  verifies the first relation of (\ref{eP}).  
 \end{proof}

 It is shown by Kesten  \cite[Lemma 1]{K1} that    for any finite subset $I\subset \mathbb{Z}$ and $t>0$
\beqn\label{KL1}
\begin{array}{ll}
(a)\;\; P\big[ \sharp \{n:n<\sigma^0_{[R,\infty)}: S^0_n\in I\} \geq (\sharp I) q_R t\big] \to  e^{-t}\quad (R\to\infty), \\[2mm]
(b)\;\; P\big[ \sharp \{n:n<\sigma^0_{\mathbb{Z} \setminus (-Q,R)}: S^0_n\in I\} \geq (\sharp I) q_{Q,R}t\big] \to  e^{-t}\quad (R\wedge Q\to\infty)
\end{array}
\eeqn
with    $q_R = P[\sigma^0_{[R,\infty)} <\sigma^0_0]$ in (a) and $q_{Q,R} =P[\sigma^0_{\mathbb{Z} \setminus (-Q,R)} <\sigma^0_0]$ in (b)
 (valid  for any recurrent r.w. that is irreducible on $\mathbb{Z}$).  Here $\sharp$ designates the  cardinality of a  set.  
By Propositions \ref{Prop52}(ii), \ref{Prop55} and \ref{Prop58}  we have
\begin{corollary}\label{Cor59}\,  Suppose $m_+/m \to 0$.  Then (\ref{KL1}) holds with 
$q_R=1/a(R)$ in (a) and  $q_{Q,R}= a(Q+R)/a(Q)a(R)$ in (b).
\end{corollary}

\vskip2mm



\begin{proposition} \label{Prop60}\, Suppose $m_+/m\to 0$. Then uniformly 
for $-Q<x<R$ subject to  (\ref{cnstr}),  as $Q\wedge R\to\infty$
\vskip2mm
 {\rm (i)} \,    $H^x_{B(Q,R)}(R) \sim a^\dagger(x)/a(R)$ and 
\beqn\label{eqP2}
P\big[\sigma^x_{R}< \sigma^x_{-Q} \,\big|\, \sigma^x_{-Q,R} < \sigma^x_0\big]
  \sim  a(Q)/a(Q+R);
  \eeqn
  
{\rm (ii)}\;  if $\, M^{-1} < Q/R \leq \limsup Q/R < M$ for some $M>1$ in addition,

\[
P\big[\sigma^x_{[R,\infty)} <  \sigma^x_{(-\infty, -Q]}  \,\big|\,\sigma^x_{\mathbb{Z} \setminus (-Q,R)}<\sigma^x_0\big]\sim a(Q)/a(Q+R).
\]
\end{proposition}
\begin{proof}
 By decomposing
 \beqn\label{oo}
 \{\sigma^x_{R} <\sigma^x_0\} = \{\sigma^x_{-Q}< \sigma^x_{R} <\sigma^x_0\} 
 +\{\sigma^x_{R}< \sigma^x_{-Q} \wedge \sigma^x_0\}
 \eeqn
 ($+$ designates the disjoint union) 
 it  follows that 
\[
H^x_{B(Q,R)}(R) =P[\sigma^x_{R}< \sigma^x_{-Q} \wedge \sigma^x_0] = P[\sigma^x_{R}<\sigma^x_0]-  P[\sigma^x_{-Q}< \sigma^x_{R} <\sigma^x_0].
\]
By the Strong Markov property  the last probability is at most
$ P[\sigma^x_{-Q} <\sigma^x_0]P[\sigma^{-Q}_{R} <\sigma^{-Q}_0]$.
By Lemma \ref{Lem35}  (see (\ref{prL6.1})) 
$P[\sigma^{-Q}_{R} <\sigma^{-Q}_0] \leq a(-Q)/a(R)$,
while owing to (\ref{cnstr}) we have $P[\sigma^x_{-Q} <\sigma^x_0] \sim a^\dagger(x)/a(Q)$.
  Combining these together verifies that 
\[
P[\sigma^x_{-Q}< \sigma^x_{R} <\sigma^x_0] \leq \frac{a^\dagger(x)a(-Q)}{a(R)a(Q)}\{1+o(1)\}.
\]
Hence   $H^x_{B(Q,R)}(R) \sim a^\dagger(x)/a(R)$.
   Now  noting  
   $P[\sigma^x_R<\sigma^x_{-Q}, \sigma_{-Q,R}<\sigma^x_0] = H^x_{B(Q,R)}(R)$ one can readily deduce (i)   from  (\ref{eP0}) and (\ref{eP}).
 
For the proof of (ii) let $\tau^x(Q)$ be the first time  $S^x_{\cdot}$ exits from $(-\infty,-Q]$   after its entering this half line (see (\ref{tau})) and ${\cal E}^x$ denote the event $\{\sigma^{x}_{(-\infty,-Q]} <\sigma^{x}_0\}$. Then 
\beq
P[\sigma^{x}_{(-\infty,-Q]} <\sigma^{x}_{[R,\infty)}<\sigma^{x}_0] 
\leq  \sum_{-Q<y<0} P[S^x_{\tau(Q)}=y, \, {\cal E}^x] P[\sigma^y_{[R,\infty)} <\sigma^y_0] +P({\cal E}^x)\times o(1),
\eeq
where $o(1)$ is due to Lemma \ref{Lem43} (see also (\ref{eq:6.3})) applied with $-Q$ in place of $-R$.

We claim that if $ Q/R < M$,  
$P[\sigma^y_{[R,\infty)}<\sigma_0^y] \to 0$ uniformly for  $-Q<y<0$, which combined with  the bound above   yields
\beqn\label{eq:6.48}
P[\sigma^{x}_{(-\infty,-Q]} <\sigma^{x}_{[R,\infty)}<\sigma^{x}_0] = P({\cal E}^x)\times o(1).
\eeqn
Since for any $\ep>0$, $P[S^y_{\sigma[0,\infty)} \geq \ep Q]\to 0$ uniformly for $-Q<y<0$ owing to Proposition \ref{Prop40}, it follows that under $Q/R < M$,
\[
P[\sigma^y_{[R,\infty)}<\sigma_0^y] \leq \sum_{0<z<R} P[S^y_{\sigma[0,\infty)} =z]P[\sigma^z_{[R, \infty)} <\sigma_0^z] +o(1) \leq P[\sigma^y_R<\sigma^y_0] +o(1) \to 0,  
\]
 where Proposition \ref{Prop52}(ii) is used for the second inequality. Thus the claim is verified.

For the rest we can proceed as in the proof of (i) above with an obvious   analogue  of (\ref{oo}). 
Under (\ref{cnstr}),  by Proposition \ref{Prop52}(ii) 
$P[\sigma_{[R,\infty)}^x <\sigma_0^x] \sim a^\dagger(x)/a(R)$   and by 
Proposition \ref{Prop52}(i) $P({\cal E}^x) \sim a^\dagger(x)/a(Q)$. Using (\ref{eP}) together with  \eqref{eq:6.48} we now obtain the asserted result as in  the same way as above. The proof of Proposition \ref{Prop60} is finished.
\end{proof}

\section{Examples}
Here we give two examples of different nature. The first one concerns  the case when $F$  belongs to the domain of attraction of a stable law:  we compute the exact asymptotic form of $a(x)$  and describe behaviour of the (one-sided) overshoot and the product $V_\ds(x)U_{\rm as}(X)$.
  The second one exhibits how   $a(x)$ and/or  $c(x)/m(x)$ can behave in  irregular ways.
 \vskip2mm
 
\textbf{ 8.1.} {\sc  Distributions in domains of attraction.}
\vskip2mm

In this subsection we suppose (in addition to (\ref{X0}))  that $X$ belongs to the domain of attraction of a stable law with exponent $1\leq \al\leq 2$, or equivalently
\beqn \label{7.1}
\begin{array}{ll}
{\rm (a)} \;\; \int_{-x}^x y^2dF(y) \sim L(x)\qquad &\mbox{if}\quad \al=2\\[2mm]
{\rm (b)} \;\;  {\displaystyle \mu(x) \sim  \frac{2-\alpha}{\alpha} x^{-\al}L(x)\quad\mbox{and}\quad \frac{\mu_+(x)}{\mu(x)} \to p} \qquad &\mbox{if}\quad 1\leq \al<2
\end{array}
\eeqn
as $x\to\infty$. Here and in the sequel  $0\leq p\leq 1$, and $L$  is  positive,  s.v. at infinity and 
 chosen so that  $L$ is differentiable and $L'(x) = o(L(x)/x)$. 
Let $Y$ be the limiting stable variable whose  characteristic function $\Psi(t) = Ee^{itY}= e^{-\Phi(t)}$  is given  by
 \[
 \Phi(t) =\left\{ \begin{array}{ll}  C_\Phi |t|^{\al}\{1 - i({\rm sgn}\, t) (p - q)\tan {\textstyle \frac12}  \al\pi\}  \qquad &\mbox{if}\quad 1<\al\leq 2,  \\[2mm]
C_\Phi |t|\{{\textstyle \frac12}\pi + i({\rm sgn}\, t) (p-q) \log |t|\}  \qquad &\mbox{if}\quad \al=1,
\end{array}\right.
\]
where $q=1-p$, ${\rm sgn} \, t= t/|t|$  and  $C_\Phi$ are some positive constants that depend on the scaling factors (cf. \cite[(XVII.3.18-19)]{F}). It is also supposed that $\int_1^\infty L(x) x^{-1}dx <\infty$ if $\al=1$ and $\lim L(x)=\infty$ if $\al=2$, so that $E|X|<\infty$ (so as to conform to   $EX=0$)  and $E X^2=\infty$ unless the contrary is stated explicitly (in some cases we allow  $E|X|=\infty$ as in Proposition \ref{Prop61}(iv), Remark \ref{Rem62}(i) and the subsection 8.1.2). 
\noindent\vskip2mm
\textbf{8.1.1.} {\sc Asymptotics of $a(x)$.}

 Put
$L^*(x) = \int^\infty_x y^{-1}L(y)dy.$
Then  
\beqn\label{eta1}
\eta(x)=\left\{ \begin{array}{ll} o(c(x)/x) \quad&\mbox{if} \quad \al=2,\\[1mm]
(\al-1)^{-1}\mu(x)\{1+o(1)\} \quad&\mbox{if} \quad 1 < \al<2,\\[1mm]
L^*(x) \{1+o(1)\} \quad&\mbox{if} \quad  \al=1
\end{array}\right.
\eeqn
and
\beqn\label{c1}
c(x)= \al^{-1}x^{2-\al} L(x)\{1+o(1)\} \quad \quad (1\leq \al\leq 2)
\eeqn
 accordingly 
\[
m(x) \sim \left\{ \begin{array}{ll} 
x^{2-\al}L(x)/[\al (\al-1)] \quad&\mbox{if} \quad 1 < \al\leq 2,\\[1mm]
xL^*(x) \quad&\mbox{if} \quad  \al=1.
\end{array}\right.
\]
The derivation  is straightforward.  
If $\al=2$, condition (\ref{7.1}a) is equivalent to  $x^2\mu(x) =o(L(x))$ as well as  to  $c(x) \sim L(x)/2$ (Appendix (A)),
which together  show  
\beqn\label{a=2}
\tilde m(x)=x\eta(x) + \tilde c(x)= o(m(x)) \quad \mbox{if}\quad \al=2.
\eeqn

Asymptotics of $\al(t)$ and $\beta(t)$  as $t\to0$ are given as follows:
\[
\al(t) \sim \left\{\begin{array}{ll}  tL(1/t)/2  \quad&\mbox{if} \quad \al=2,\\[1mm]
\kappa_\al' t^{\al-1} L(1/t) \quad&\mbox{if} \quad 1< \al<2, \\[1mm]
\frac12 \pi L(1/t) \quad&\mbox{if} \quad \al=1,
\end{array} \right.
\]
and
\[
\beta(t) = \left\{\begin{array}{ll} o(\al(t))    \quad&\mbox{if} \quad \al=2,\\[1mm]
\kappa_\al'' t^{\al-1} L(1/t)\{1+o(1)\}  \quad&\mbox{if} \quad 1< \al<2, \\[1mm]
 L^*(1/t)\{1+o(1)\}  \quad&\mbox{if} \quad \al=1,
\end{array} \right.
\]
where $\kappa_\al'=(\al-2) \Gamma(-\al)\cos{\textstyle \frac12} \pi \al$ and $\kappa''_\al= (2-\al)\Gamma(-\al)\sin{\textstyle \frac12} \pi \al$; in particular for $1<\alpha <2$,
\beqn\label{CF}
1- E e^{it X}= t\alpha(t) \sim  (\kappa_\alpha' + i (2p-1)\kappa_\alpha'')t^\alpha L(1/t) \quad (t\downarrow 0).
\eeqn
 For verification see  \cite{Pt},  \cite[Theorems 4.3.1-2]{BGT} if $1\leq \al <2$.  The estimate in case   $\al=2$ is deduced from (\ref{eta1}) and (\ref{c1}). Indeed,  uniformly for  $\ep>0$
 \[
 \al(t) = \int_{0}^{\ep/t} \mu(x) \sin tx\, dx + O(\eta(\ep/t)) =tc(\ep/t)  \{1+O(\ep^2) +o(1)\},
 \]
 so that $\al(t) \sim tc(1/\ep t)\sim  tL(1/t)/2$; as for $\beta(t)$ use  
(\ref{m1}a) and (\ref{a=2}). 

  In case $\al=1$,
 we shall need the following second-order estimate:
\beqn\label{beta}
\beta(t) = \eta(1/t)  + \gamma L(1/t))\{1+o(1)\},
\eeqn
where $\gamma$ is Euler's constant. The proof, given below, is similar---rather simpler---to that 
of Proposition \ref{Prop34}(i).
It holds  that
\beqn\label{C1_C2}
\int_{1/t}^\infty \mu(y)\cos ty\,dy \sim C_1 L(1/t)\;\; \mbox{and} \;\; 
 \int_0^{1/t}\mu(y)(1-\cos ty)dy \sim C_2 L(1/t), 
 \eeqn
 where $C_1 =\int_1^\infty \cos s\,ds/s$, $C_2=\int_0^1 (1-\cos s)ds/s$.
 Indeed, by monotonicity of $\mu$  we infer that for each $M>1$, 
 $\big|\int_{M/t}^\infty \mu(y)\cos ty\,dy\big|\leq \pi\mu(M/t)/t \sim \pi L(1/t)/M$,
 while  
 $$\int_{1/t}^{M/t} \mu(y)\cos ty\,dy = \int_1^M L(z/t){\cos z} dz/z \sim C_1L(1/t)$$
  as $t\downarrow 0$ and $M\to\infty$ in this order.
One,  therefore,  obtains the first formula of (\ref{C1_C2}).  One can deduce the second one  similarly. 
  Since $C_1-C_2 =\gamma$, (\ref{beta}) now follows immediately.

The next proposition is valid also in  case  $E|X|=\infty$, while the r.w. $S$  is assumed to be {\it recurrent}. This remark is relevant only for the last assertion  (iv) of it, the r.w. being  recurrent under $E|X|=\infty$ only if   $\alpha=2p =1$   (see (\ref{Rec_tst}) for the recurrence criterion).

\begin{proposition}\label{Prop61}  Suppose that   (\ref{7.1}) is satisfied.  Then as $x\to\infty$
\vskip2mm
{\rm (i)}\; except in  case  $\al=2p=1$, 
\beqn\label{i}
\bar a(x)\sim  \kappa_\al^{-1} x/m(x), 
\eeqn
 where $\kappa_\al = 2\big \{\cos^2 \frac12 \pi \al + (p-q)^2\sin^2 \frac12 \pi\al \big\}\Ga(\al)\Ga(3-\al)$ $\big[\kappa_\al =0 \Leftrightarrow \al=2p=1\big];$
 \vskip2mm
 {\rm (ii)}\; if   $1\leq \al<2$ and $p\neq 1/2$ in case $\alpha=1$, 
\beqn\label{a+/-}
\left\{\begin {array}{ll}  a(-x) \sim 2p\bar a(x),\\
   a(x) \sim 2q\bar a(x),  \end{array} \right.
 \eeqn
 where the sign \lq \,$\sim$' is interpreted in  the obvious way if $pq=0;$ and
  \vskip2mm
 {\rm (iii)}\;  if  $\al=2$,   $\mu(x)x^2$ is s.v.    and  $\mu_+(x)/\mu(x) \to p$ ($x\to \infty$, $0\leq p\leq1$), then (\ref{a+/-}) holds$\,;$
 
 \v2
 {\rm (iv)}\;  if $\alpha =2p=1$ and  there exists $r:=  \lim P[S_n>0]$, then 
 \beqn\label{iv}
 \bar a(x)  \sim \left\{\begin{array} {ll} {\displaystyle \frac{2\sin^2 r\pi}{\pi^2} \int_1^x\frac{dy}{yL(y)} } \quad &\mbox{if}\quad  0< r <1,\\[4mm]
 {\displaystyle \frac12 \int_{x_0}^x \frac{\mu(y)}{A^2(y)} dy} \quad &\mbox{if}\quad r =  \;\mbox{$0$ or $1$};
 \end{array} \right.
 \eeqn
 $a(x)\sim a(-x)$;  and
 $$ x/m(x) =o(\bar a(x)) \quad \mbox{in case} \;\;  E|X|<\infty. $$ 
 
 [$A(x)$ is given in (\ref{A}) and $x_0$ in (c$'$) of Remark \ref{Rem11}; see Remark \ref{Rem62}(ii) when Spitzer's condition is not assumed..]
\end{proposition}

If $\alpha =1$, $p\neq q$ and $EX=0$, then $\int_1^x \mu(y)dy/A^2(y) \sim x/(q-p)^2 m(x)$ so that  the second expression on the RHS of (\ref{iv}) is natural extension of   that on the RHS of (\ref{i}).
\begin{proof}  \, (i) and (ii) are given in \cite[Lemma 3.3]{Belk}  (cf. also \cite[Lemma 3.1]{Uattrc} for $1<\alpha<2$) except for the case $\al=1$, and implied  by  Proposition  \ref{Prop34} for $\alpha=1$. 



\v2
{\sc proof of} (iii). Recall that $a(x) = \bar a(x) + b_-(x) - b_+(x)$, where
$$b_\pm(x)= \frac1\pi \int_0^\pi \frac{\beta_\pm(t)} {[\alpha^2(t)+ \gamma^2(t)]}\cdot \frac{\sin xt}{t} dt$$
 (see (\ref{a/b})). 
 
 Let  $\al=2$ and   $L^\circ(x):= x^2\mu(x)$.   Then by the assumption of (iii) $\tilde L(x):=x\eta_+(x)\sim pL^\circ(x)$ is s.v. and $\tilde L'(x) = \eta_+(x) -x\mu_+(x) =o(\tilde L(x)/x)$ ($x\to\infty$), which leads to
 \beqn\label{stand}
 \beta_+(t) = t\int_0^\infty \eta_+(y)\sin ty \,dy= t\int_0^\infty \tilde L(y) \frac{\sin ty}{y}  \,dy \sim \frac12\pi p L^\circ(1/t)t
 \eeqn
 by a standard way  (cf. \cite[Theorem V.2.7]{Z}).    Thus as before we see
 $$b_+(x) = \frac p{2}\int_0^\pi \frac{L^\circ(1/t)\{1+o(1)\}}{[t L^\sharp(1/t)]^2}\sin x t\, dt \quad \bigg(L^\sharp(x) :=\int_1^x \frac{L^\circ(u)}{u}du\bigg). $$
 For $\ep>0$, on using  $|\sin xt|\leq 1$, the contribution from $t>\ep/x$ to the above integral 
 is dominated in absolute value by a constant multiple of 
  $$\int_{\ep/x}^\pi \frac{L^\circ(1/t)}{t^2 [L^\sharp(1/t)]^2}dt
   = \int_{1/\pi}^{x/\ep} \frac{L^\circ(y)}{[L^\sharp(y)]^2}dy\sim \frac{x}{\ep} \frac {L^\circ(x)}{[L^\sharp(x)]^2} =o\bigg(  \frac{x}{L^\sharp(x)}\bigg).$$
while the remaining integral  may be written as
 $$x\int_{0}^{\ep/x} \frac{L^\circ(1/t)\{1+o_\ep(1)\}}{[t L^\sharp(1/t)]^2} tdt  =x\int^\infty_{x/\ep} \frac{L^\circ (y)dy}{[L^\sharp(y)]^2y} \{1+o_\ep(1)\} = \frac{x}{L^\sharp(x)}\{1+o_\ep(1)\},
  $$
  Thus $b_+(x)\sim \frac12 px/L^\sharp(x)$. In the same way   $b_-(x)\sim \frac12q x/L^\sharp(x)$ and by (i) $\bar a(x) \sim \frac12 x/L^\sharp(x)$.  Consequently  $a(x)\sim qx/m(x) \sim 2q \bar a(x)$.




\v2
{\sc proof of} (iv). Here   $E|X|$ may be infinite. Note that $\alpha(t) \sim \frac\pi2 L(1/t)$ remains valid under $E|X|=\infty$.   In case $r=1$,  the assertion follows from  Proposition  \ref{Prop34} as a special case.  By duality the same is true in case $r=0$.

Let   $B(x)$ be an increasing positive  function of $x\geq0$ such that  $B(x)/x \sim L(B(x))$  as $x\to\infty$ (the both sides   are  necessarily s.v.).  Then  the laws of $S_n/B(n)$ constitute a tight family    if and only if $\rho_n$ is bounded away from $0$ and $1$.  The convergence
  $\rho_n \to r \in (0,1)$ is equivalent to
\beqn\label{m/Bn}
\qquad \lam_n:=nE[\sin \{X/B(n)\}] \to \lam   \quad (-\infty < \lam<\infty).
\eeqn
Together with  $\alpha(t) \sim \frac12 \pi L(1/t)$ the latter condition yields
\beqn\label{Esin}
-\gamma(t) = t^{-1}E[\sin X t]  \sim \lam  L(1/|t|)\qquad (t\to 0),
\eeqn
so that $1-\psi(t) \sim \big(\frac\pi2|t| -i\lam t\big) L(1/|t|)$ ($t\to 0$),  the limiting stable law 
is characterized by 
$$E[e^{itY}] = \lim E[e^{it(S_n/B(n) -\lam_n)} ]= e^{-\frac\pi2 |t|}$$
and the values of $r$ and $\lam$ are related  by
 \beqn\label{rho/Y}
  r = P[Y+\lam>0].
 \eeqn
The law of $Y$ has the density given  by
$\frac12\big[x^2+ \frac14\pi^2\big]^{-1}$, and  
 solving (\ref{rho/Y})  for  $\lam$, one obtains 
$$\lam= \frac\pi{2} \tan \big[\pi (r- \textstyle{\frac12})\big].$$

By (\ref{Esin})
$$\frac{\alpha(t)}{\alpha^2(t)+\ga^2(t)} \sim \frac{2/\pi}{(1+\tan^2[\pi(r - \textstyle{\frac12})])L(1/t)} = \frac{2\sin^2 r\pi} {\pi L(1/t)}. $$
and as before we deduce
$$\bar a(x) \sim \frac{2 \sin^2 r\pi}{\pi^2} \int_1^x \frac{dy}{yL(y)}.$$


 For the proof of $a(x) \sim a(-x)$, which is equivalent to $a(x)-a(-x)=o(\bar a(x))$, it suffices to show that
 \beqn\label{a-a}
  a(x) -a(-x)=\frac2{\pi}\int_0^\pi \frac{-\ga(t)\sin xt}{[\alpha^2(t)+\ga^2(t)]t}dt = o\bigg(\Big[\frac1{L}\Big]_*(x)\bigg),
  \eeqn
  where $[1/L]_*(x)= \int_1^x [yL(y)]^{-1}dy$.
Denote the above integral restricted on $[0,1/x]$ and $[1/x,\pi]$ by $J_{<1/x}$ and $J_{>1/x}$.  By  (\ref{Esin}) as well as $\alpha \sim \frac\pi2 L(1/t)$  it follows that  
$$J_{<1/x}  \sim Cx\int_0^{1/x} \frac{dt}{L(1/t)} \sim \frac2{\pi L(x)}, $$
hence $J_{<1/x}= o([1/L]_*(x))$.
 For the evaluation of $J_{>1/x}$, recall we have chosen $L$ so that  $L'(x)/L(x)=o(1/x)$.  
 Let  $\alpha_0(t) = \frac1{2}\pi L(1/t)$ and define $\tilde L$ and  $\tilde L_0$ by
$$\tilde L(1/t) = \frac{-\ga(t)}{\alpha^2(t) +\ga^2(t)} \quad \mbox{and}\quad
 \tilde L_0(1/t) = \frac{\lambda L(1/t)}{\alpha_0^2(t) +\lambda^2L^2(1/t)}.$$
Then  $\int_{1/x}^\pi \tilde L_0(1/t) \sin xt\, dt/t = O(\tilde L_0(x))= O(1/L(x))$, whereas  by (\ref{Esin})

  $$ \tilde L(1/t)  -\tilde L_0(1/t) = o\big(1/L(1/t)\big),$$
showing  $\big|\int_{1/x}^\pi \big[ \tilde L(1/t)  -\tilde L_0(1/t)] \sin xt \, dt/t\big| \leq \int_1^x | \tilde L(y)  -\tilde L_0(y)|dy/y = o\big([1/L]_*(x)\big)$. It therefore follows that  $J_{>1/x} = o\big([1/L]_*(x)\big)$. Thus 
(\ref{a-a}) is verified.

In case $E|X|<\infty$, by  (\ref{iv}) as well as $m(x)/x \sim L^*(x)$   we can readily see that  $x/m(x) = o(\bar a(x))$ (see Remark \ref{Rem62}(ii) below). 
The proof of Proposition \ref{Prop61} is complete.
\end{proof}


\vskip2mm
\begin{remark}\label{Rem62}   (i)  Let $\alpha=1$  in (\ref{7.1}) and $E|X|=\infty$. Then  $F$ is recurrent if and only if
\beqn\label{Rec_tst}
\int_{1}^\infty \frac{\mu(y)}{\big(L(y)\vee |A(y)|\big)^2}dy =\infty
\eeqn
(cf. \cite[Section 4.3]{Urenw}); in particular if $p\neq 1/2$, then $F$ is transient, for we have  $|A(y)| \sim |p-q|\int_0^y\mu(t)dt \to\infty$ which implies $\int^\infty  \mu(x)dx\big/A^2(x) <\infty$.

(ii)  Let $\al=2p=1$  in (\ref{7.1}).    From the above proof of (iv) one has for $t>0$ $|\ga(t)|\leq C\big[L(1/t) \vee A(1/t)\big]$. It therefore follows  without assuming Spitzer's condition that
$$\bar a(x) \asymp \int_{1}^x \frac{\mu(y)}{\big(L(y)\vee |A(y)|\big)^2}dy;$$ 
If $EX=0$, then  $L(y)\vee |A(y)| =o\big(L^*(y)\big)$, which shows 
 that $\bar a(x)/[x/m(x)]\to\infty$, for $x/m(x)\sim 1/L^*(x)\sim \int_1^x \mu(y)dy/[L^*(y)]^2$. This is deduced also from Theorem \ref{th:1_2}(ii).


\end{remark}
\v2 
\noindent
\begin{remark} \label{Rem63}  If $\al=2$ and the assumption of (iii) of Proposition \ref{Prop61} is satisfied, then equality (\ref{cons1}) (for $x$ fixed) holds, which however fails  if
 $1\leq \al<2$, $0<p<1$  and Spitzer's condition (\ref{SpC}) is fulfilled.     
  This is  verified  by combining   (\ref{eta1}) and  (\ref{c1}).  
   If either $\al\neq 1$  or  $\al=1$ with $p=1/2$ and  $0<r<1$, the verification is  easy and omitted.
 Let $\al=1$ with $r=1$ (including the case $0<p<1/2)$, the case $r=0$ being similar. Since  the probability of $S^0_\cdot$
 exiting $[-R,R]$ through $R$ then tends to 1 (cf. \cite[Section 7.3]{D_L}), it suffices to show that  $P[\sigma^0_{R}<\sigma^0_{-R}]\sim q$. Let $\lam_\pm =P[\sigma^0_{\pm R}<\sigma^0_{\mp R}\vee \sigma^0_0]$. Using Proposition \ref{Prop61}
 one  observes that $P[\sigma^R_{-R}<\sigma^{R}_0]\to 1-p$ and  $P[\sigma^{-R}_{R}<\sigma^{-R}_0]\to p$ so that 
 \[
 P[\sigma^0_{-R}<\sigma^{0}_0]=P[\sigma^0_{R}<\sigma^{0}_0] 
\sim \lam_-+ q\lam_+ \sim  \lam_++ p\lam_-,
\] 
showing $\lam_-/\lam_+ \sim p/q$. Hence $P[\sigma^0_{R}<\sigma^0_{-R}] = \lam_+/(\lam_-+\lam_+) \sim q$, and 
solving  the linear equations   we obtain    $\lam_+ \sim [q/(1-pq)]/2\bar a(R)$.
 \end{remark}
 \v2

 
Suppose $m_+(x)/m(x) \to 0$. Then   $E(-\hat Z)=\infty$ under which
 it holds  that
 \beqn\label{a_H}
 a(-x)= \sum_{y=1}^\infty H_{[0,\infty)}^{-x}(y)a(y), \quad x>0
 \eeqn
 \cite[Corollary 1] {Uladd}.  Using this identity we are going to derive  the asymptotic form of $a(-x)$ as $x\to\infty$ when $\mu_+(x)$ varies regularly at infinity. Recall that 
 $V_{{\rm ds}}(x)$ denotes the renewal function for the weakly descending ladder height process of the r.w.  As in Section 7.3,  let $U_{{\rm as}}(x)$  be  the renewal function for the strictly  ascending one and put
 $$\ell_+(x) = \int_0^x P[Z> t]dt \;\; \mbox{for} \; 1\leq \al \leq 2;\mbox{and}\quad \ell_-(x) = \int_0^x P[-\hat Z> t]dt \;\; \mbox{for} \;  \al = 2.$$
 We know  that $\ell_+(x)$ and $\ell_-(x)$ are   s.v. as $x\to\infty$  and so   is $P[-\hat Z>x]$ for $\al =1$  and that
 \beqn\label{U/V}
 U_{{\rm as}}(x) \sim \frac{x}{\ell_+(x)} \quad\mbox{and}\quad
 V_{{\rm ds}}(x) \sim \left\{
 \begin{array}{ll} 
 c_0^{-1}x/\ell_-(x) \quad & \alpha =2,\\[1mm]
 \kappa_\al' x^{\al-1}\ell_+(x)/L(x) \;\; & 1<\al <2,\\[1mm]
 c_0^{-1}/P[-\hat Z>x] \quad & \alpha =1,
 \end{array} \right. 
 \eeqn
where $ \kappa_\al' =-\al[(2-\al)\pi]^{-1}\sin \al \pi$  and  $c_0= e^{-\sum p^k(0)/k}$ (provided   (\ref{7.1}) holds): see  \cite[Theorems 2, 3 and 9]{R}, \cite[Lemma 8.8]{Uattrc}  (for the slow variation of  $P[-\hat Z>x]$ in case $\alpha=1$ see the last mentioned result (in the dual form) in (iii) of the next subsection).  By  Proposition \ref{Prop61}(i)  
$$
a(x) \sim \big[2\al (\al-1) \kappa_\al^{-1}\big] x^{\al-1}/L(x) \quad (x\to\infty) \quad \mbox{ for}\quad 1<\al\leq 2
$$
 and comparing  the asymptotic formulae of  Lemma \ref{Lem45}(i)  and  (\ref{U/V})  we find 
\beqn\label{2ell/L}
2\ell_-(x)\ell_+(x) \sim c_0^{-1}L(x) \qquad (\al =2).
\eeqn 

 \begin{proposition}\label{Prop64}\,  Suppose that  (\ref{7.1})  holds  and  $m_+(x)/m(x) \to 0$ and that either
 $\mu_+(x)$ is regularly varying at infinity with index  $-\beta$ or  $\sum_{x=1}^\infty \mu_+(x) [a(x)]^2<\infty$. Then  
 $$
a(-x) \sim \left\{ 
\begin{array} {ll}
{\displaystyle U_{{\rm as}}(x) \sum_{z=x}^\infty  \frac{\mu_+(z) V_{{\rm ds}}(z)a(z)}{z} }\quad &\mbox{if}\quad \al=\beta=2,\\[4mm]
{\displaystyle  C_{\al,\beta}\frac{U_{{\rm as}}(x)}{x}\sum_{z=1}^x  \mu_+(z)V_{{\rm ds}}(z)a(z) }\quad &\mbox{otherwise}
\end{array}\right.
$$
as $x\to\infty$,  where 
$C_{\al, \beta} =[\Ga(\al)]^2\Ga(\beta-2\al+2)/\Ga(\beta).$ 
 
 [Explicit expressions of the RHS  are given  in the proof: see (\ref{asmp1}) to (\ref{asmp4}). By Theorem \ref{Thm6} and Proposition \ref{Prop61}(i) it follows  that  $a(x)\sim [\Ga(\alpha)\Ga(3-\alpha)]^{-1} x/m(x)$, provided  $m_+/m\to 0$.]
 \end{proposition}
 \begin{proof}  \, Put  $u_{{\rm as}}(x)= U_{{\rm as}}(x)- U_{{\rm as}}(x-1)$, $x>0$ and $u_{{\rm as}}(0) =U_{{\rm as}}(0) =1$. Then $G(x_1,x_2) := u_{{\rm as}}(x_2-x_1)$ is the Green function of the strictly increasing ladder process killed on its exiting the half line $(-\infty, 0]$ and
  by the last exit decomposition we obtain
  \beqn\label{LED}
  H^{-x}_{[0,\infty)}(y) = \sum_{k=1}^x u_{{\rm as}}(x-k)P[Z=y+k] \quad (x\geq 1, y\geq 0)
  \eeqn
(see (\ref{def_H}) for $H^{-x}_{[0,\infty)}$). Suppose the conditions of the proposition to hold  and let  $\mu_+(x) \sim L_+(x)/x^\beta$  with   $\beta \geq \alpha$ and $L_+$ slowly  varying at infinity.  Then on  summing by parts 
 \beq
 P[Z> y] = \sum_{k=-\infty}^0 g_{[1,\infty)}(0,k)\mu_+(y-k)
 &=& \sum_{z=0}^\infty v_{{\rm ds}}(z) \mu_+(y+z)\\
 &\sim& \beta \sum_{z=0}^\infty V_{{\rm ds}}(z) \frac{\mu_+(y+z)}{y+z+1}\\
 &\sim&  C_0 V_{{\rm ds}}(y)\mu_+(y) \qquad (y\to\infty), 
 \eeq
where $v_{{\rm ds}}(z) =  V_{{\rm ds}}(z) -V_{{\rm ds}}(z-1)$ and 
$$C_0= \beta \int_0^\infty t^{\alpha-1} (1+t)^{-\beta-1}dt =\frac{\Ga(\al)\Ga(\beta-\al+1)}{\Ga(\beta)}.$$ 

 If  $\beta> 2\al-1$, then  $\sum \mu_+(x) [a(x)]^2 <\infty$ and  $a(-x)$ converges to a constant that is positive if the walk is not right-continuous (cf. \cite[Theorem 2]{Uladd}). Hence we may consider only  the case  $\al\leq \beta\leq 2\al-1$. 
 
  Let $\al>1$. Recall  $a$ varies regularly with index  $\al-1$ at infinity. Then  substituting the above equivalence  into (\ref{LED}), returning to 
 (\ref{a_H}) and performing   summation by parts  lead to 
 \begin{eqnarray}
 a(-x) &\sim & (\alpha-1) \sum_{y=1}^\infty \sum_{k=0}^x u_{{\rm as}}(x-k)P[Z\geq y+k] \frac{a(y)}{y}  \nonumber\\
 &\sim& (\alpha-1) C_0 \sum_{y=1}^\infty \sum_{k=0}^x 
 u_{{\rm as}}(x-k)V_{{\rm ds}}(y+k) \mu_+(y+k) \frac{a(y)}{y}. 
 \label{711}
 \end{eqnarray}
Owing to  (\ref{u/ell}), namely 
$u_{{\rm as}}(x) \sim 1/ \ell_+(x)$, 
 one can replace  $ u_{{\rm as}}(x-k)$ by  $1/\ell_+(x)$ in (\ref{711}), the inner sum 
over  $(1-\ep)x <k\leq x$ being negligible as $x\to \infty$ and $\ep\to 0$. After  changing the variables by $z = y+k$  the last double sum restricted to $y+k\leq x$  then becomes asymptotically equivalent to
\beqn\label{k/j}
\frac1{\ell_+(x)}\sum_{z=1}^xV_{{\rm ds}}(z)\mu_+(z) \sum_{k=0}^{z-1} \, \frac{a(z-k)}{z-k}
\sim \frac1{(\al-1)\ell_+(x)}\sum_{z=1}^xV_{{\rm ds}}(z)\mu_+(x) a(z).
\eeqn

If $\beta =2\al-1 >1$ (entailing $\lim \ell_+(x) = EZ<\infty$), then $C_0=C_{\al,\beta}$,  the last sum in (\ref{k/j}) varies slowly and the remaining part of the double sum on the right side of (\ref{711}) is negligible,  showing 
\beqn\label{asmp1}
a(-x) \sim \frac{C_0}{\ell_+(x)} \sum_{z=1}^xV_{{\rm ds}}(z)\mu_+(z) a(z) \sim \frac{C_{\al,\beta}}{\ell_+(x)} \sum_{z=1}^x \frac{L_+(z)}{\ell_*(z)z},
\eeqn
where 
$$\ell_*(z) = \left\{
\begin{array}{ll} c_0\ell_-(z)L(z)/2  \quad  &\al=2,   \\[1mm]
 (\kappa_\al/2\al(\al-1)\kappa_\al')[L(z)]^2/\ell_+(z) \quad  & 1<\al<2.
 \end{array}\right.
$$  

 Let  $\al\leq \beta <2\al-1$. If  $|\al-2|+ |\beta-2|\neq 0$, then the outer sum in (\ref{711}) over $y>M x$ as well as that over $y<x/M$ becomes negligibly small  as $M$ becomes large and one can easily
infer that 
\beqn\label{asmp2}
a(-x) \sim   \frac{(\alpha-1)C_1C_0L_+(x) x^{2\al-1-\beta}}{\ell_+(x)\ell_*(x)(2\alpha-1-\beta)} 
\sim \frac{C_{\al,\beta}}{\ell_+(x)} \sum_{z=1}^xV_{{\rm ds}}(z)\mu_+(z) a(z),
\eeqn
where $C_1 =(2\alpha-1-\beta) \int_0^1 ds \int_0^\infty  (s+t)^{-\beta+\al-1}t^{\al-2}dt=\int_0^\infty  (1+t)^{-\beta+\al-1}t^{\al-2}dt$ and  $C_{\al,\beta}$ is identified by $(\al -1)C_1C_0= C_0 \Ga(\al)\Ga(\beta-2\al+2)/\Ga(\beta-\al+1) = C_{\al,\beta}$. 

Let $\al =\beta=2$.  Then 
$a(x)\sim 2x/L(x)$ and uniformly for $0\leq k\leq x$
$$\sum_{y=x}^\infty V_{{\rm ds}}(y+k)\mu_+(y+k)\frac{a(y)}{y} = \frac1{c_0} \sum_{y=x}^{\infty} \frac{ L_+(y+k)\{1+o(1)\}}{(y+k)\ell_-(y+k) L(y)/2}\sim \sum_{y=x}^\infty  \frac{ L_+(y)}{\ell_*(y)y}$$
while on  recalling  (\ref{k/j}) 
$\sum_{k=1}^x\sum_{y=1}^x V_{{\rm ds}}(y+k)\mu_+(y+k){a(y)}/{y} \leq  C''' xL_+(x)/\ell_*(x)$.  Wth these two bounds we deduce from   (\ref{711})  that  
\beqn\label{asmp3}
a(-x) \sim 
\frac{C_0x}{\ell_+(x)} \sum_{y=x}^\infty V_{{\rm ds}}(y)\mu_+(y)\frac{a(y)}{y}\sim \frac{x}{\ell_+(x)} \sum_{y=x}^\infty \frac{L_+(y)}{\ell_*(y)y}
\eeqn
(since  $(\al-1)C_0=C_0=1$).  

The relations (\ref{asmp1}) to (\ref{asmp3}) together show those of  Proposition \ref{Prop64} in case $1<\al\leq 2$ since $U_{{\rm as}}(x) \sim x/\ell_+(x)$.
 
  It remains to deal with  the case $\al=\beta =1$ 
when  in place of (\ref{711}) we have
  $$a(-x)\sim  \sum_{y=1}^\infty \sum_{k=0}^x\frac{ u_{{\rm as}}(x-k)\mu_+(y+k)}{c_0P[-\hat Z>y+k]} \cdot \frac{d}{dy}\frac1{L^*(y)}.$$
 Note $P[-\hat Z>x]$ is s.v. and  $\frac{d}{dy}[1/L^*(y)]= L(y)/y[L^*(y)]^2 $. Then    one sees that 
the above double sum restricted to $y+k\leq  x$ is asymptotically equivalent to
\beqn\label{asmp4}
\frac{1/c_0}{\ell_+(x)} \sum_{z=1}^x \frac{L_+(z)}{ zP[-\hat Z>z]L^*(z)}, \eeqn
hence s.v., while  the outer sum over $y> x$ is negligible.  It therefore follows that the above formula  represents the asymptotic form of  $a(-x)$ and  may be written  alternatively   as 
  $x^{-1}U_{{\rm as}}(x) \sum_{y=1}^x V_{{\rm ds}}(y)\mu_+(y)a(y)$ as required.
\end{proof} 

\begin{remark}\label{rem7.30} 
 Let   $\alpha=q=1$.  Then $P[-\hat Z> x] \sim \int_x^\infty \big[ F(-t)/\ell_+(t) \big]dt\big/V_{\rm ds}(0)$ according to Lemma \ref{Lem50}. If  $EZ<\infty$ in addition,  by evaluating  the formula in (\ref{asmp4}) one can find that  $a(-x)\sim \int_1^x  (1-F(t))[c_0 L^*(t)]^{-2}dt$. 
\end{remark}
\begin{remark}\label{rem7.3} Let the assumption of Proposition \ref{Prop64} be satisfied and $M$ be an arbitrarily given number $>1$.  If $\al =\beta=1$ or  $2\al-1\leq \beta$, then $a(-z)$ is s.v. as $z\to\infty$ so that  (\ref{ass1}) holds
  uniformly for $-M R< x<0$.  
 Below we show the following.
  \v2

  (i) \; For $\al =2$, $ P[\sigma^{x}_R <\sigma^x_0]/ P[\sigma^{x}_{[R,\infty)} <\sigma^x_0] \to 1$\; as $R\to\infty$ uniformly for  $|x| <MR$, although  (\ref{ass1}) may fail even if $0<-x <\!< R$. [In case $\beta=2$, in particular, $P[\sigma^{x}_{[R,\infty)} <\sigma^{x}_0] =o\big(a(x)/a(R)\big)$ for $ -MR<x< -R/M$.]

  \v2
   (ii) \;If $\alpha\leq \beta <2\al -1<3$, then  $\de < P[\sigma^{x}_R <\sigma^{x}_0]/ P[\sigma^{x}_{[R,\infty)} <\sigma^{x}_0] < 1-\de$  for  $ -MR<x< -R/M$ for some 
  $\de=\de_M>0$.
  \v2\noindent

For the proof we first deduce  that if $\beta>1$, then  for any  $\ep>0$ there exists $k\geq 1$ such that for $-MR<x<0$, 
\beqn \label{R7.31}
P\big[ S^x_{\sigma[0,\infty)} > k R \,\big|\, S^x_{\sigma[0,\infty)} \geq R\big] \leq \ep. 
\eeqn 
To this end use the inequality 
$g_{[0,\infty)}(x,y)\leq g_{\{0\}}(x,y)\leq  g_{\{0\}}(x,x)=2\bar a(x) $ to see that
$H_{[0,\infty)}^{x}(y)
\leq 2\bar a(x) \mu_+(y),$
showing  for  $k\geq 1$
\beqn\label{LB1}
P\big[S^x_{\sigma[0,\infty)} > kR\big] \leq 2\bar a(x)\sum_{y\geq kR} \mu_+(y).
\eeqn
For the lower bound, noting $g_{B}(x,-z)=g_{-B}(z,-x)$ we apply the bound (\ref{sb6.3}) to have
$$g_{[0,\infty)}(x,-z)\geq g_{\{0\}}(z,-x) -a(x)\{1+k_+\} = a(z)-a(z+x)- o(a(|x|)),$$
and hence that for $|x|/2\leq z\leq |x|$,  $g_{(-\infty,0]}(x,-z) \geq a(|x|/2)\{1+o(1)\}$. This leads to
\beqn\label{R7.32}
H^{x}_{[0,\infty)}(y) \geq C_1 a(|x|)\mu_+(y).
\eeqn 
Hence by the   assumed regular variation of $\mu_+$, $P[R < S^x_{ \sigma[0,\infty)}\leq 2R] \geq C_1' \bar a(|x|) \sum_{y\geq R}\mu_+(y)$,
 which together with (\ref{LB1}) yields (\ref{R7.31}). 

If $\al=2$,  then $a(y)-a(-R+y) =a(R) +o(a(y))$ uniformly for $y\geq R$ in view of Proposition \ref{Prop61}(i),  and one easily infers that $P[\sigma^y_R<\sigma^y_0]\to 1$ uniformly for $R\leq y\leq kR$.  Combined with the result for  $x\geq 0$ (given in (\ref{eqP6.2}) as well as in (\ref{R7.31}))  this  shows (i). [See (\ref{asmp3}) for  what is mentioned about  (\ref{ass1}).]

The proof of (ii) is similar. Note that $1<\alpha\leq \beta$ under the premise  of (ii).   The lower bound follows from  (\ref{L6.5ii}) and the first relation of  (\ref{asmp2}).  For the upper bound one observes that under the premise of (ii),  $P[\sigma^y_0<\sigma^y_R]$ is bounded away from zero for $2R\leq y\leq 3R$, which together with (\ref{R7.32}) shows that
$$P\big[S^x_{\sigma[0,\infty)} \geq R, \sigma_0^x<\sigma_R^x\big] \geq c a(|x|)\mu_+(R)R$$
with some $c>0$. On the other hand 
we have  $P[\sigma^{x}_R<\sigma^x_0] \asymp a(-R)/a(R)$ for  $- MR \leq x\leq R/M$ and   $a(-R)/a(R) \asymp a(R) \mu_+(R)R$. These together verify  (ii).
\end{remark}

\vskip2mm
\textbf{8.1.2}\;  {\sc Relative stability.}

We continue to suppose (\ref{7.1})  and  examine the behaviour of  the overshoot
$Z_R = S_{\sigma(R,\infty)} -R$.
We do not assume the condition  $E|X|<\infty$ (automatically valid for $1<\alpha \leq 2$) nor the recurrence of the r.w., but assume that in case $E|X|=\infty$,  $\sigma_{[1, \infty)} <\infty$ w.p.1 or, what is the same thing (according to  \cite{E1}),  $\int_{1}^\infty xdF(x)\big/\big[1+ \int_0^x F(-t)dt\big] =\infty$. 
 We shall observe that under (\ref{7.1}), the sufficient condition of Proposition \ref{Prop40} is also necessary for $Z$ to be r.s. 
In case $1<\al\leq 2$ let $\rho= P[Y>0]$, which $P[S_n>0]$ approaches as $n\to\infty$.
\vskip2mm

(i)\, {\it If   either $1<\alpha <2$ and $p=0$ or $\al=2$, then  $x\eta_+(x) =o(m(x))$ so that   $Z_R/R \stackrel{P}\to 0$  as $R\to \infty$ according to Proposition \ref{Prop40}. }
(In this case, we have  $\al\rho =1$, and  the same result also follows from Theorems 9 and 2 of   \cite{R}.)  
\v2
(ii)  {\it If $1<\al<2$ and $p> 0$, then $0<\al \rho<1$,  which  implies that $P[Z>x]$ is regularly varying of index $\al\rho$ \cite{R} and   the distribution of 
$Z_R/R$ converges weakly to the probability law determined by the  density  $C_{\al\rho}/x^{\al\rho}(1+x), x>0$ (\cite{R}, \cite[Theorem XIV.3]{F}).}
\v2
(iii) {\it Let $\al=1$. If $p=q=1/2$ we also suppose  that  Spitzer's condition is satisfied, 
or what amount to the same thing, there exists
\beqn\label{SpC}
\lim  P[S_n >0] =r.
\eeqn
It then follows that  $r=1$ if and only if (C.I) of Proposition \ref{Prop40}  holds, and if this is the case   $Z$ is r.s. by Proposition \ref{Prop40}
  so that $Z_R/R \, \stackrel{P}\longrightarrow \,  0 $. 
  
   In case $p\neq 1/2$,  by examining (C.I) one easily deduces that (\ref{SpC}) holds with $r\in \{0,1\}$: $r=1$ if either $p < 1/2$ and $EX=0$  or  $p>1/2$  and $E|X|=\infty$;   $r=0$   in the other case  of $p\neq 1/2$. 

   If $p=1/2$, $r$  may take all values from $[0,1]$.  
  If   $0<r<1$ (entailing $p=\frac12$), then the law of  $Z_R/R$ converges  to  a probability law with the  density  $[\pi \sqrt x(1+x)]^{-1}, \, x>0$. [Combined with what is stated above this entails that $Z$ is r.s. if and only if $r=1$.]
  If $r=0$ (i.e., either  $p>1/2$ with $EX=0$ or $p<1/2$ with $E|X|=\infty$ or $p=1/2$ with $r=0$),  then   $P[Z>x]$ is s.v. at infinity and $Z_R/R \, \stackrel{P}\longrightarrow \,  \infty $. }

\v2


Thus in case $1<\al\leq 2$  condition (C.II)  works as a criterion   for the relative stability of $Z$, while  if $\al=1$, $Z$ can be r.s. under  $x\eta_+(x) \asymp m(x)$ (so that (C.II) does not hold)  and condition (C.I) must be employed for the criterion.
\vskip2mm

For  (iii)  some explanations are needed. Let $\al=1$.   What is mentioned about the case   $r=1$  
follows immediately from the result of \cite{KM} mentioned in Remark \ref{Rem11}: under (\ref{7.1}), one sees that condition (c$'$) in it is equivalent to (C.I) so that (C.I) $\Leftrightarrow$ $r=1$ in view of (\ref{A/KM2}). 
If $0<r<1$, $S^0_n/B(n)$ converges in law (see the argument made around (\ref{m/Bn})),  implying  the asserted convergence in law of $Z_R/R$  (cf.  \cite[Section 4]{R}).
If $r=0$,  the result stated in (iii) is shown in \cite{Uexp} (see also (the proof of) Lemma 3.5(i) and Lemma  5.1(ii)  of \cite{Uexit} for another reasoning and related results).

\vskip2mm
\textbf{ 8.2.}  { \sc An example exhibiting  irregular behaviour of $a(x)$.}
\vskip2mm

 We give a recurrent symmetric r.w. such that $E[ | X|^\al/ \log( | X|+2)]<\infty$, and  $P[|X|\geq x] =O(|x|^{-\alpha})$ for some $1<\al<2$ and 
\beqn\label{Ex1}
0< \!\!\,^\forall r<1, \;\; \bar a(r x_n)/\bar a(x_n) \;  \longrightarrow \; \infty \quad \mbox{as} \;\;n\to\infty
\eeqn
for some  sequence $x_n\uparrow \infty$, $x_n\in 2\mathbb{Z}$, which provides  a counter  example
for the fact mentioned right after Theorem \ref{th:1_2}.  
In fact for the law $F$ defined below it holds that  for any $0<\de <  2-\alpha$ there exists  positive constants  $c_*$ such that for all sufficiently large $n$,
\beqn\label{Ex10}
\bar a(x)/\bar a(x_n) \geq  c_*n \quad \mbox{for all integers $x$ satisfying}\;\; 2^{- \delta n} < x/x_n \leq 1-2^{-\de n}.
\eeqn
( $-\bar a$ diverges to $-\infty$ fluctuating  with sharp peaks at the points $x_n$  and very wide basins in between.)


Put
\[
x_n=2^{n^2},\;\; \lambda_n=x_n^{-\al}=2^{-\al n^2}\quad (n=0, 1,2,\ldots),
\]
\[
p(x)=\left\{\begin{array}{ll} A\lambda_n\quad &\mbox{if}\;\; x=\pm x_n\;\; (n=0, 1,2,\ldots),\\
0 \quad&\mbox{otherwise}, 
\end{array}\right.
\]
where $A$ is the  constant chosen so as to make $p(\cdot)$  a probability.   

Denote by $\eta^{(1)}_{n,k,t} $ the value of $1- 2(1-\cos u) /u^2$ at $u=x_{n-k}t$ so that
 uniformly for  $|t|<1/x_{n-1}$ and $k=1, 2,\ldots,  n$,
\[
\frac{ \lambda_{n-k}[1-\cos (x_{n-k} t)]}{1- \eta^{(1)}_{n,k,t}}= \frac{\lambda_{n-k} x_{n-k}^2}{2x^2_n}(x_n t)^2=\frac12  2^{-(2-\al)(2n-k)k}\lambda_n(x_n t)^2
\]
and 
\[
\eta^{(1)}_{n,k,t} = o(1)\quad \mbox{if}\quad k\neq 1\quad\mbox{ and}   \quad 0\leq \eta^{(1)}_{n,1,t} < (x_{n-1}t)^2/12.
\]
 Then for $|t|<1/x_{n-1}$,
 \beqn\label{7.2}
\sum_{x=1}^{x_{n-1}} p(x)(1-\cos x t)=\sum_{k=1}^{n} p(x_{n-k})(1-\cos x_{n-k} t) = A\ep_n\lambda_n (x_n t)^2 \{1-\eta^{(1)}_{n,1,t} +o(1)\}
 \eeqn
 with 
 \[
 \ep_n =2^{1-\al}2^{-2(2-\al)n},
 \]
for $p(x_{n-1})(1-\cos x_{n-1}t)$,  the last term of the series,   is dominant over the rest.
 On the other hand 
 \begin{eqnarray}\label{7.10}
 \mu(x)= \sum_{y>x} p(y) =\left\{ \begin{array} {ll} A\lambda_{n} + O(\lambda_{n+1}) &\qquad(x_{n-1}\leq x < x_{n})\\
 o(\lambda_n\ep_n)  &\qquad(x \geq x_{n}).
 \end{array}\right.
 \end{eqnarray}
 Thus,  uniformly for  $|t|< 1/{x_{n-1}}$,
 \beqn\label{7.21}
1-\psi(t)=t\al(t) = 2A\lambda_n\Big [ 1-\cos \, x_nt +\ep_n  (x_n t)^2(1- \eta^{(1)}_{n,1,t})\Big] +o(\lambda_n\ep_n).
 \eeqn
Also, $2A\lambda_n(1-\cos x_nt)=A \lambda_n(x_n t)^2(1-\eta^{(2)}(n,t))$ with $0\leq \eta^{(2)}\leq 1/2$ for $|t|<1/x_n$. 
Remember that
\[
\bar a(x)=\frac1{\pi}\int_0^\pi \frac{1-\cos xt}{1-\psi(t)}\, dt.
\]
We break the integral into  three parts
\beq
(\pi A)\bar a(x)&=&\bigg(\int_0^{1/x_n}+\int_{1/x_n}^{1/x_{n-1}}+\int_{1/x_{n-1}}^\pi \bigg)\frac{1-\cos x t}{[1-\psi(t)]/A}dt\\
&=&I(x)+I\!I(x)+I\!I\!I(x)\quad  \mbox{(say).}
\eeq

First we compute \textit{an upper bound of} $\bar a(x_n)$. On using the trivial inequality  $1-\psi(t) \geq 2p(x_n) (1-\cos x_n t)$ 
\[
I(x)\leq 1/2 \lambda_n x_n= x_n^{\al-1}/2.
\]
Put $r=r(n,x) = x/x_n$.  Then by   (\ref{7.21})  it follows  that for sufficiently large $n$,
\begin{eqnarray}
I\!I(x) &\leq&  \int_{1/x_n}^{1/x_{n-1}}\frac{1-\cos x t}{2\lambda_n (1-\cos x_n t)+\ep_n\lambda_n(x_n t)^2}dt \nonumber\\
&=&\frac1{\lambda_n x_n} \int_1^{x_n/x_{n-1}}\frac{1-\cos r u}{2(1-\cos u)+\ep_n u^2}du.
\label{IIx}
\end{eqnarray}
In case $x=x_n$ so that $r=1$ this leads to
\[
I\!I(x_n) \leq  x_n^{\alpha-1} \sum_{k=0}^{2^{2n}} \int_{-\pi}^{\pi} \frac{1-\cos u}{2(1-\cos u)+\ep_n(u+2\pi k)^2}du.
\]
Noting $u+2\pi k >2\pi k -\pi$ and  $\frac16u^2< 1-\cos u <\frac12 u^2$ for $-\pi<u<\pi$,  one observes
\[
\int_0^\infty dk\int_0^\pi \frac{u^2}{u^2+\ep_n k^2}du=\frac{\pi/ 2}{\sqrt {\ep_n}}\int_0^\pi udu =\frac{\pi^3/4}{\sqrt {\ep_n}}
\]
to conclude  
\[
I\!I(x_n)\leq  \frac{C_1x_n^{\al-1}}{\sqrt {\ep_n}}.
\]

For the evaluation of $I\!I\!I$ we apply  (\ref{7.21}) with $n-1$ in place of $n$ to deduce that
\beq 
\int_{1/x_{n-1}}^{1/x_{n-2}} \frac{1-\cos x_n t}{[1-\psi(t)]/A}dt &\leq& \int_{1/x_{n-1}}^{1/x_{n-2}}\frac2{2\lambda_{n-1}(1-\cos x_{n-1}t)+\ep_{n-1}\lambda_{n-1}(x_{n-1}t)^2}dt\\
&=&\frac2{\lambda_{n-1}x_{n-1}}\int_1^{x_{n-1}/x_{n-2}}\frac1{2(1-\cos u)+\ep_{n-1}u^2}du.
\eeq
The last integral is less than
\beq
&&\int_1^\pi  \frac1{u^2/3+\ep_{n-1}u^2}du+\sum_{k=1}^{2^{2n}}\int_{-\pi}^\pi \frac 1{u^2/3+\ep_{n-1}(-\pi+2\pi k)^2}du\\
&&\leq  C_1\sum_{k=1}^{2^{2n}}\frac1{\sqrt{\ep_n}\,k} \leq  C_2 \frac{n}{\sqrt{\ep_n}}.
\eeq
Since $\lambda_n x_n/\lambda_{n-1}x_{n-1} = 2^{-(\al-1)(2n-1)}$ and the remaining parts of 
$I\!I\!I$ is smaller,
\[
I\!I\!I\leq C_3n 2^{-2(\al-1)n}x_{n}^{\al-1}/\sqrt{\ep_n}.
\]
Consequently
\beqn\label{upp.bd}
\bar a(x_n) \leq  Cx_{n}^{\al-1}/\sqrt{\ep_n}.
\eeqn
\v2
\textit{ The lower bound of $\bar a(x)$}. \; Owing to (\ref{7.21}), for $n$ sufficiently large,
\[
[1-\psi(t)] /A \leq  2\lambda_n(1-\cos x_nt) + 4\ep_n\lambda_n(x_n t)^2 \quad \mbox{for} \;\;1/x_n\leq   |t|< 1/x_{n-1}.
\]

Let $\frac12  < x/x_n \leq 1$ and put $y=x_n-x$ and $b=y/x_n$ $(\geq 2)$. Then 
\beqn\label{IIc}
\frac{I\!I(x)}{x_n^{\alpha -1}} \geq 
\int_{0}^{x_n/x_{n-1}} \frac{1-\cos (1-b)u }{2(1-\cos u)+ 3\ep_n u^2}du.
\eeqn
The RHS is bounded from below by
$$ \sum_{k=1}^{2^{2n-1}/2\pi}
\int_{-\pi}^\pi \frac{1-\cos\big[(1-b)u-2\pi b k \big] }{u^2+ 3\pi^2\ep_n  (2k+1)^2}du.
$$
In case  $b \leq \sqrt{\ep_n}= 2^{(1-\alpha)/2} 2^{-(2-\alpha) n}$, restricting the summation to $k<1/ \sqrt{\ep_n}$, one sees that
\beqn\label{IIa}
 \frac{I\!I(x)}{x_n^{\alpha -1}}  \geq C \sum_{k=1}^{1/\sqrt{\ep_n}}\int_{0}^\pi \frac{u^2 + b^2k^2}{u^2 +\ep_n k^2} du\geq \frac{C\pi/2}{\sqrt{\ep_n}} \quad\mbox{if}\quad 1- \sqrt{\ep_n}\leq \frac{x}{x_n} \leq 1.
 \eeqn
For $b\geq \sqrt{\ep_n}$, we restrict the summation to the intervals  $(\nu + \frac14)/b \leq  k    \leq (\nu+\frac34)/b$, $\nu=0,1, 2, \ldots, \big\lfloor 2^{(2-\alpha)n}b/2 \big\rfloor$ to see that
\beqn\label{IIb}
\frac{I\!I(x)}{x_n^{\alpha -1}}  \geq C \sum_{k= 1/4b}^{1/\sqrt{\ep_n}} \int_{-\pi /2}^{\pi/2} \frac{du}{u^2+\ep_n k^2} 
\sim C\frac{\log [4b/\sqrt{\ep_n}\,]}{\sqrt{\ep_n}} \quad\mbox{if}\quad  \frac12 \leq \frac{x}{x_n} \leq 1- \sqrt{\ep_n}.
\eeqn
This especially shows that if $0<\de< 2-\alpha$,  $\bar a(x) \geq Cn x_n^{\alpha-1}/\sqrt{\ep_n}$ for $\frac12 x_n < x/x_n <1-2^{-\de n}$.

Let $1\leq j < (2-\alpha)n +1$ and $2^{-j-1} <x/x_n \leq 2^{-j}$. Put 
$$x_n^{(j)} = 2^{-j}x_n,\;\;  y=x_n^{(j)} - x \quad \mbox{and} \quad b= y/x_n^{(j)}.$$
 Then we have  (\ref{IIc})  with $(1-b)$ replaced by $2^{-j}(1-b)$  and making the changes of  variable   $u= s+2\pi(2^{j-1} +k)$ in  (\ref{IIc}) with $k=0,1, \ldots, $ we  infer that
$$\frac{I\!I(x)}{x_n^{\alpha-1}} \geq c_1\sum_{k=0}^{[2^{2n}-2^j]/4\pi} 
 \int_{-\pi}^\pi \frac{1-\cos  \theta_{j,y}(s,k) }{2(1-\cos s) +3\ep_n\big[2\pi(2^{j}+k)\big]^2}ds$$
where $\theta_{j,y}(s,k) = 2^{-j}(1-b)s + \pi + 2\pi \big[2^{-j}(1-b)k - \frac12b\big]$.  Noting $b < \frac12,$ we see that
 $$\cos\theta_{j,y}(s,k) \leq \frac1{\sqrt 2}
 \quad\mbox{if} \quad |s|<\frac{2^j\pi}{4} \quad \mbox{and}\;\;  \nu - \frac14  \leq \frac{(1-b)}{2^{j}}k -\frac12 b \leq \nu +\frac14 \;\; $$
 for $\nu = 0, 1, 2, \ldots $ and then find that the sum above is lager than a constant multiple of
 \begin{eqnarray}
 \sum_{k=1}^{2^{2n}/5\pi} 
 \int_0^1 \frac{ds}{s^2+\ep_n (2^j+k )^2} 
 &\sim&\sum_{k=2^{j-1}}^{2^{2n}/5\pi}  \frac{1}{\sqrt{\ep_n}\,k} \arctan \frac1{\sqrt{\ep_n}\,k} \\
 &\geq& \sum_{k=2^{j-1}}^{1/\sqrt{\ep_n}}  \frac{1/2}{\sqrt{\ep_n}\,k}\sim \frac{(2-\alpha)n -j +1}{[2/\log 2] \sqrt{\ep_n}} .
 \end{eqnarray}
This together with (\ref{IIa}) and (\ref{IIb}) entails  (\ref{Ex10}).

  In the present example,   $c(x)/m(x)$ oscillates between $0$ and $1$. 
 Indeed without difficulty one can see
 that for $x_{n-1}\leq x\leq x_{n}$,
 \[
 c(x)/A \sim {\textstyle \frac12}(x^2-x_{n-1}^2)\lambda_{n} + {\textstyle \frac12} x_{n-1}^2\lambda_{n-1},
 \]
 \[
 x\eta(x)/A \sim x(x_{n}-x)\lambda_{n} +x x_{n+1}\lambda_{n+1},
 \]
 and
 \beqn\label{m_asymp}
 m(x)/A \sim  x(x_{n}- {\textstyle \frac12} x)\lambda_{n} + {\textstyle \frac12}x_{n-1}^2\lambda_{n-1} \asymp \left\{ \begin{array} {ll}
 x x_n^{1-\alpha} \quad & x> \ep_n x_n\\
 (x_{n-1})^{2-\alpha} \quad  & x< \ep_n x_n
 \end{array} \right.
 \eeqn
and then  that
 $x\eta(x)/c(x)$ tends to zero for $x$ with  $1- o(1)< x/x_{n-1}<1+  o(2^{2(\al-1)n})$  and diverges to infinity for $x$ satisfying  $2^{2(\al-1) n}x_{n-1} <\!< x <\!< x_n$.
 From (\ref{7.21}) one infers that $\al(t)/[tc(1/t)]$ oscillates between $8$ and $4\ep_n t^2 x_n^2\{1-o(1)\}$ about $M/2\pi$ times when $t$ ranges over the interval $[1/x_n, M/x_n]$ and that $\al(2\pi/x_n)/\al(\pi/x_n) =O(\ep_n)$.

This example also exhibits  that the converse of Theorem \ref{th:1_2}(ii) is not true. 
It holds that   $\liminf \alpha(t)/\eta(1/t)=0$ and $\limsup \alpha(t)/\eta(1/t) \geq \limsup \alpha(t)/tm(1/t) >0$. On the other hand
\beqn\label{AP0}
\lim \frac{\bar a(x)}{x/m(x)} =\infty \quad if \;\; \alpha< 4/3
\eeqn
as is verified below. Thus the condition $\lim \alpha(t)/tm(1/t) =0$ is not necessary for the ratio $\bar a(x)/[x/m(x)]$ to diverge to infinity.
[For $\alpha > 4/3$  the ratio  $\bar a(x)/[x/m(x)] $  is bounded along the sequence $x_n^* =\ep_n x_n$---as one may easily show.]
For verification of (\ref{AP0}) we  prove
\beqn\label{III}
\bar a(x) \geq c_1(x_{n-1})^{\alpha-1}/\sqrt{\ep_n} \quad  \mbox{for}\quad x\geq x_{n-1}.
\eeqn
Observe that  $\alpha < 4/3$ entails---in fact, is equivalent to---each of the relations $x_n^{\alpha-1} = o(x_{n-1}^{\alpha-1}/\sqrt{\ep_n})$ and  $\ep_n x_n = o(x_{n-1}/\sqrt{\ep_n})$. Then,  comparing the above lower bound of $\bar a(x)$  to the asymptotic form of  $x/m(x)$ obtained from (\ref{m_asymp}) ensures the truth 
of (\ref{AP0}).  
For the proof of (\ref{III})  we evaluate $I\!I\!I(x)$. Changing the variable $t=u/x$ we have
$$I\!I\!I(x)\geq  \frac1{3\lambda_{n-1} x}\int_{x/x_{n-1}}^{x/x_{n-2}} \frac{(1-\cos u)du}{1-\cos (x_{n-1} u/x) 
+ \ep_{n-1} (x_{n-1}u/x)^2 }.$$
The integral on the RHS is easily evaluated to be larger than 
$$\frac14\int_{x/x_{n-1}}^{x/x_{n-2}} \frac{du}{1+ \big[\sqrt{\ep_n}(x_{n-1}/x)u\big]^2} 
=\frac{x/x_{n-1}}{4\sqrt{\ep_n}} \int_{\sqrt{\ep_n}}^{2^{2n-3}\sqrt{\ep_n}}\frac{ds}{1+s^2}\sim \frac{\pi x}{8x_{n-1}\sqrt{\ep_n}},
$$
 showing (\ref{III}) since  $ \bar a(x)\geq  I\!I\!I(x)/(\pi A)$.

 

\section{Appendix}

{\bf (A)} Here we  present two results derived from the standard facts about the s.v. functions. 
The first and  second ones  are used  for the proof of Theorem \ref{Thm8} and for the deduction of  Corollary \ref{Cor9}(i) from Theorem \ref{Thm8}, respectively.
\v2
(1) {\it Let $L(x)> 0$ be a s.v. function at infinity. The following are equivalent
\v2
\;\;\; {\rm (a)} \quad $m(x)\sim L(x)/2$.
\v2
\;\;\; {\rm (b)} \quad $x\eta(x) = o\big(m(x)\big)$ 
\v2
\;\;\; {\rm (c)} \quad $c(x)\sim L(x)/2$.
\v2
\;\;\; {\rm (d)} \quad  $x^2\mu(x)\big/\int_{-x}^x t^2dF(t) \to 0$.
\v2
\;\;\; {\rm (e)} \quad $ \int_{-x}^x t^2dF(t) \sim L(x)$.}
\v2
\pf\;  The implication  (a) $\Rightarrow$  (b) follows from  $\frac12 x\eta(x)\leq m(x)- m(\frac12 x)$ and  its converse  from $m(x) = m(1)e^{\int_1^x \ep(t)dt/t}$, where $\ep(t)
=t\eta(t)/m(t)$. Combining (a) and (b) shows (c) in view of   $m(x) = c(x)+x\eta(x)$; conversely   (c)  entails $x^2\mu(x) = o(L(x))$ by which  one deduces  $x\eta(x) \ll x\int_x^\infty L(t)t^{-2}dt\sim L(x)$. Thus  (b) $\Leftrightarrow (c)$.     The equivalence of (d) and (e)  follows from   Theorem VIII.9.2 of \cite{F} and shows  the equivalence of   (e) and   (c)  since $\int_{-x-}^{x+} t^2dF(t) = -x^2\mu(x) + 2c(x)$.  \qed

\v2
(2)  Let  $m_+/m\to 0$. Then 
 by $\int_0^x t^2dF(t) \leq 2c_+(x) =o(m_-(x))$ the slow variation of  $m_-(x)$ entails  $\int_{-x}^x t^2 dF(t)\sim m_-(x)$ in view of the equivalence  (a) 
 $\Leftrightarrow$ (c) (applied with $m_-$ in place of $m$ in (1)). 
 \v2\noindent
{\bf (B) }\; 
Let $T_0=0$ and $T= (T_n)_{n=0}^\infty$ be a r.w. on  $\{0,1,2,\ldots\}$  with  i.i.d.   increments. Put
$$u_x = \sum_{n=0}^\infty P[T_n=x]\;\;  (x = 1,2,\ldots), \quad  L(t) =\int_{0}^t P[T_1 > s]ds  \;\;  (t\geq 0)$$
and suppose that $T_1$ is aperiodic so that  $u_x$ is positive for all sufficiently large $x$. 
Erickson  \cite[\S 2(ii)]{Ec} shows that $\lim u_x L(x) = 1 $ if  $tP[T_1>t]$  is  s.v. at infinity.   
This restriction on $T_1$ is relaxed in  the following lemma,  which is used to obtain  (\ref{u/ell}). 
\begin{lemma} \label{lemA}\;   If $L$ is s.v. at infinity, then
$u_x \sim 1/L(x)$ as $x\to\infty$.
\end{lemma}
\v2\noindent
\pf \, We follow the argument made by Erickson \cite{Ec}. 
 Let $\phi(\theta) = E\exp \{i\theta T_1\}$.  Unlike \cite{Ec} we take up the sine series of coefficients $u_x$ that represents the imaginary part of $1/(1-\phi(\theta))$.  Suppose $ET_1=\infty$, otherwise the result being the well-known renewal theorem.  Fourier inversion  yields 
\beqn\label{F_invers}
u_x = \frac{2}{\pi} \int_0^\pi S(\theta)  \sin x\theta d\theta, \quad \mbox{where} \quad S(\theta) =  \Im \bigg(\frac{1}{1-\phi(\theta)}\bigg).
\eeqn 
where the integral is absolutely convergent as is verified shortly.  The proof of this representation of $u_x$  will be given after  we derive the assertion  of the lemma  by  taking it for granted. 
 The assumed slow variation of $L$  implies---in fact equivalent to---each of 
\[\int_0^t  sP[T_1 \in ds] \sim L(t)\quad\mbox{and} \quad   
tP[T_1>t]/ L(t) \; \to \; 0\quad (t\to\infty)
\]
(cf.  \cite[Theorem VIII.9.2]{F},  \cite[Corollary 8.1.7]{BGT}).  
Using this we observe that as $\theta \downarrow 0$,  $\int_0^{\ep /\theta}t P[T_1 \in dt] \sim L(1/\theta)$ for each  $\ep>0$ and hence
\begin{eqnarray} \label{AB2}
1-\phi(\theta) 
&=&\bigg[\int_0^{1/\theta}+\int_{1/\theta}^\infty \bigg](1- e^{i \theta t})P[T_1\in dt] \nonumber\\
&=& -i  \theta L(1/\theta)\{1+o(1)\}.
\end{eqnarray}
In particular  $S(\theta)\sin x\theta $ is summable on $(0,\pi)$ as mentioned above.  

Decomposing 
 $u_x = \frac{2}{\pi} \int_0^{B/x} + \frac{2}{\pi} \int_{B/x}^\pi =J_1+J_2,$  we  deduce from (\ref{AB2}) that
 \beqn\label{AB3}
 J_1 = \frac{2}{\pi}  \int_0^{B}  \frac{\sin u\, du}{u G(x/u)}\{1+o(1)\}
 \eeqn
with $o(1)\to 0$ as  $x\to\infty$ for each $B> 1$. On the other hand
\beq
 \pi J_2 &=& \int_{B/x}^\pi \Big[S(\theta) - S\Big(\theta +\frac \pi{x}\Big)\Big]\sin x\theta \,d\theta + \bigg(\int^{(B+\pi)/x}_{B/x} -  \int_\pi^{\pi+\pi/x}\bigg)S(\theta)\sin x\theta \,d\theta\\
 &=& J_2' + J_2'' \quad (\mbox{say}).
\eeq
 With the help of  
$$|\phi(\theta) -\phi(\theta')| \leq 2|\theta-\theta'|G(1/|\theta-\theta'|) \quad (\theta\neq \theta')$$
(Lemma 5 of \cite{Ec})  the same proof as given in \cite[(5.15)]{Ec} yields the  bound  $J_2'\leq C'/BG(x)$. By the aperiodicity  of $T_1$ and  (\ref{AB2})  $|S(\theta)|\leq C/\theta G(1/\theta)$ ($0 < \theta\leq \pi $) and it is easy to see that  $ |J_2''|\leq C'' [B/x + 1/BG(x/B)]$. Thus $\lim_{x\to\infty} G(x)J_2 \leq C'/B$.  Since  (\ref{AB3}) implies that    $ G(x)J_1\to 1$
as $x\to\infty$ and $B\to\infty$ in this order,  we conclude $G(x)u_x \to 1$ as desired. \qed

\v2
{\it Proof of  (\ref{F_invers}).}  The corresponding \lq cosine formula', namely  $u_n= \frac2{\pi}\int_0^\pi C(\theta)\cos x\theta\,d\theta$, where $C(\theta):= \Re(1-\phi(\theta))^{-1}$,  is applied in \cite{GL} and \cite{Ec} without proof. In  \cite{GL} is cited the article \cite{Hrg}  which  proves   the Herglotz representation  theorem (named after its author)  of the analytic functions in the unit open disc with non-negative real parts. Noting  $\Re(1-\phi(\theta))^{-1}\geq 0$ ($|z|<1$)    
  one can easily  obtain  the cosine formula by deducing  the summability of   $C(\theta)$   from the representation theorem.  The same argument does not go through  for (\ref{F_invers}), $S(\theta)$ being not always  summable.  An elementary proof of the cosine formula is given in \cite[p.98-99]{S} (cited in \cite{Ec}) and this is easily modified to obtain (\ref{F_invers}) as given below.

 Put for $1/2 <r<1$
$$w_r(\theta) = \frac1{2\pi} \sum_{n=0}^\infty r^n \phi^n(\theta)= \frac{1}{2\pi(1- r\phi(\theta))}.$$
 Noting  $P[T_n =x] = (2\pi)^{-1}\int_{-\pi}^{\pi} \phi^n(\theta)e^{-ix\theta}d\theta$ and   $u_x = \lim_{r\uparrow 1} \,\sum_{n=0}^\infty r^n P[T_n =x]$ for all  $x\in \mathbb{Z}$ as well as  $u_{-x}=0$ for $x\geq 1$ one deduces  that
$$u_x = u_x- u_{-x} = \lim_{r\uparrow 1} \int_{-\pi}^\pi w_r(\theta) (-i2\sin x\theta)d\theta.$$
Since $\Re w_r(\theta)$ is even and $\Im w_r(\theta)$ is odd   we see  that the  limit  above equals 
$$ 4\lim_{r\uparrow 1} \int_{0}^\pi \lim_{r\uparrow 1} \Im w_r(\theta) \sin x\theta\, d\theta = \frac{2}{\pi}\int_0^\pi S(\theta)\sin x\theta \,d\theta.$$ 
Here the equality is justified by the bound   $|\Im w_r(\theta) \leq 1/\big[2\pi r |\Im \phi(\theta)| \big]\leq 1/\theta G(1/\theta)$ for   $\theta>0$ small enough that follows from  (\ref{AB2}).
 Thus (\ref{F_invers}) is obtained.
 
  \v2\noindent


\begin{thebibliography}{99}
\baselineskip=9pt

\bibitem{Bt} J. Bertoin, L\' evy Processes, Cambridge Univ. Press, Cambridge  (1996).



\bibitem {Belk}  B. Belkin, An invariance principle for conditioned recurrent random walk attracted to a stable law, Zeit. Wharsch. Verw. Gebiete {\bf 21}  (1972), 45-64.


\bibitem {BGT} N.H. Bingham, G.M. Goldie and J.L. Teugels, Regular variation,  Cambridge Univ. Press,  Cambridge, 1989.


  

\bibitem {Cho} Y. S. Chow, On the moments of ladder height variables, Adv. Appl. Math. {\bf 7} (1986), 46-54.

\bibitem {BD} J. Bertoin and R. A. Doney, Spitzer's condition for random walks and
L\' evy processes, Ann. Inst. Henri Poincar\' e, {\bf 33} (1997), 167-178



\bibitem {DM}  R.A. Doney and R. A. Maller, The relative stability of the overshoot for L\' evy processes, Ann. Probab. \textbf{ 30} (2002), 188-212.

\bibitem {D}  R.A. Doney, A note on a condition satisfied by certain random walks, J. Appl. Probab. {\bf 14}, (1977), 843-849.






\bibitem {D_L} R. A. Doney, Fluctuation theory for L\'evy processes,   Lecture Notes in Math. 1897 (2007). Springer, Berlin.

\bibitem {Em}  D. J. Emery,  On a condition satisfied by certain random walks,   Zeit. Wharsch. Verw. Gebiete  {\bf 31}, (1975),  125-139.


\bibitem {Ec} K.B. Erickson, Strong renewal theorems with infinite mean, Trans. Amer. Math. Soc. \textbf{ 151} (1970), 263-291.

\bibitem {E1} K. B. Erickson, The strong law of large numbers when the mean is undefined, TAMS. {\bf 185} (1973), 371-381.

\bibitem{Hrg} G. Herglotz, \"Uber Potenzreihen mit positivem reellen Teil im Einheitskreis, Ber. Verh.  S\" achs. Akad. Wiss. Leipzig. Math. Nat. Kl.  {\bf 63} (1911) 

\bibitem {GL} A. Garsia and J. Lamperti, A discrete renewal theorem with infinite mean, Comment. Math. Helv. {\bf 37} (1962/3), 221-234. MR {\bf 26} $\sharp$5630.


\bibitem{F} W. Feller, An Introduction to Probability Theory and Its Applications, vol. 2, 2nd edn.  John Wiley and  Sons, Inc. NY. (1971)

\bibitem{GM}  P. Griffin and T. McConnell, Gambler's ruin and the first exit position of random walk from large spheres, Ann. Probab. {\bf 22} (1994), 1429-1472.

\bibitem{GOT} P. E. Greenwood, E. Omey and J.I Teugels, Harmonic renewal measures, Z. Wahrsch. verw. Geb.  \textbf{  59} (1982), 391-409.






\bibitem {K1} H. Kesten,  On a theorem of Spitzer and Stone and random walks with absorbing barriers,
  Illinois J. Math. {\bf 5} (1961), 246-266. 



\bibitem {K2} H. Kesten, Random walks with absorbing barriers and Toeplitz forms, Illinois J. Math. {\bf 5} (1961), 267-290. 

\bibitem {K3} H. Kesten,  Ratio limit theorems II, Journal d'Analyse Math. {\bf 11}, (1963), 323-379.



\bibitem {K4} H. Kesten, Problem 5716, Amer. Math. Monthly {\bf 77} 197 (1970)

\bibitem {K} H. Kesten, The limit points of a normalized random walk. Ann. Math. Statist. \textbf{ 41} (1970) 1173-1205.

\bibitem {KM0} H. Kesten and R. A. Maller, Stability and other limit laws for exit times of random walks from a strip or a half line, Ann. Inst. Henri Poincar\'e, \textbf{ 35} (1999), 685-734.
 
 \bibitem {KM} H. Kesten and R. A. Maller, Infinite limits and infinite limit points of random walks and trimmed sums,  Ann. Probab., \textbf{ 22} (1994), 1473-1513.
 
 

\bibitem {M}  R. A. Maller,  Relative stability, characteristic functions and stochastic compactness, J. Austral. Math. Soc. (Series A) {\bf 28} (1979), 499-509.
 
 
 \bibitem {Pt} E.J. G. Pitman, On the behaviour of the characteristic function of a probability distribution in the neighbourhood of the origin, JAMS(A)  \textbf{ 8}  (1968), 422-43.
 
 
 \bibitem {R} B.A. Rogozin,  On the distribution  of the first ladder moment and height and fluctuations of a random walk, Theory Probab. Appl. \textbf{ 16} (1971), 575-595.
 


 



\bibitem {S} F. Spitzer, Principles of Random Walks, Van Nostrand, Princeton, 1964. 








\bibitem {Unote}  K. Uchiyama,   A note on  summability of  ladder heights 
  and the distributions  of  ladder epochs
  for  random walks,  Stoch. Proc.  Appl. \textbf{ 121} (2011), 1938-1961.
  
  

 
      
\bibitem {Uf.s}    K. Uchiyama,  Asymptotic behaviour of  a random walk  killed on a finite set, Potential Anal.  \textbf{ 46}(4), (2017), 689-703.

 \bibitem {Uexp}  K. Uchiyama, On the ladder heights of random walks attracted to stable laws of exponent 1, Electron. Commun. Probab. {\bf 23} (2018), no. 23, 1-12.  doi.org/10.1214/18-ECP122

 \bibitem {Uladd}  K. Uchiyama,  The potential function and ladder variables  of a recurrent random walk on $\mathbb{Z}$  with infinite variance. 
(preprint:  available at: http://arxiv.org/abs/1805.03971.)

  \bibitem {Uattrc}  K. Uchiyama, Asymptotically stable random walks
of index $1<\alpha<2$ killed on a finite set. Stoch. Proc.  Appl. {\bf 129}  (2019), 5151-5199. 

  \bibitem {Uexit}  K. Uchiyama,  A note on the exit problem from an interval for  random walks oscillating on
  $\mathbb{Z}$ with  infinite variance, (2019) preprint.  http://arxiv.org/abs/1908.00303

    \bibitem {Urenw}  K. Uchiyama,  A  renewal theorem for relatively stable variables, to appear in Bull. London Math. Soc. (2020) 


\bibitem{VW}  V. A. Vatutin and V. Wachtel, Local probabilities for random walks conditioned to stay positive, Probab. Theory Rel. Fields, {\bf 143} (2009), 177-217



\bibitem {Zol}  V. M. Zolotarev, Mellin-Stieltjes  transform in probability theory, Theor. Probab.  Appl. (1957), 433-460.

\bibitem {Z} A. Zygmund, Trigonometric series, vol. 2, 2nd ed., Cambridge Univ. Press (1959)  

\end{thebibliography}
\end{document}